\documentclass[10pt, a4paper]{article}

\usepackage{amsmath, amsthm, amssymb,  mathtools}
\usepackage{lmodern}
\usepackage[utf8]{inputenc}
\usepackage[T1]{fontenc}
\usepackage[margin=2cm]{geometry}
\usepackage{enumitem}
\usepackage{comment}
\usepackage[notcite, notref, final]{showkeys}
\usepackage{hyperref}

\usepackage{color}

\usepackage[backend=bibtex, sorting=nyt, style=numeric, url=false, isbn=false, eprint=false, maxbibnames=4, giveninits=true]{biblatex}
\renewbibmacro{in:}{\ifentrytype{article}{}{\printtext{\bibstring{in}\intitlepunct}}}
\DeclareFieldFormat
[article,inbook,incollection,inproceedings,patent,thesis,unpublished]
{title}{#1\isdot}
\DeclareFieldFormat[article]{volume}{\mkbibbold{#1}}
\DeclareFieldFormat[article]{number}{\mkbibbold{#1}}

\newenvironment{system}[1]
{
\left\{\begin{aligned}{#1}} 
{ \end{aligned}\right.}

\newcommand\dd{\mathrm{d}}
\newcommand\vect[1]{{#1}}

\newtheorem{thm}{Theorem}[section]
\newtheorem{lem}[thm]{Lemma}
\newtheorem{prop}[thm]{Proposition}

\theoremstyle{definition}
\newtheorem{defi}{\textbf{Definition}}
\newtheorem{defn}[defi]{\textbf{Definition}}

\theoremstyle{remark}
\newtheorem{rem}[thm]{Remark}
\newtheorem{assumption}{\textbf{Assumption}}

\newcommand\parmatrix[1]{\begin{pmatrix}#1\end{pmatrix}}

\newcommand\diag{\mathrm{diag}}

\newcounter{stepnum}
\newcommand\step{\typeout{Use of step outside a stepping environment ! }}
\newenvironment{stepping}{
	\setcounter{stepnum}{1}
	\renewcommand\step{{\bf Step \thestepnum:~}\stepcounter{stepnum}}
}{\renewcommand\step{\alert{Use of step outside a stepping environment !}}}

\allowdisplaybreaks

\newcommand{\writefoot}[1]{
    \renewcommand{\thefootnote}{}
    \footnotetext{\hspace{-16.5pt}\scriptsize#1}
    \renewcommand{\thefootnote}{\arabic{footnote}}
}

\begin{document}
\writefoot{\small \textbf{AMS subject classifications (2020).} Primary: 35K40; Secondary: 35K57, 35K58, 35C07, 92D25. \smallskip}
\writefoot{\small \textbf{Keywords.} Periodic reaction-diffusion systems, spreading speed, cooperative system, KPP-type equations, anisotropic propagation, homogenization, strong coupling, Darwinian evolution.\smallskip}
\writefoot{\small \textbf{Acknowledgements:} The initial part of the research was conducted when QG was a JSPS International Research Fellow (FY2017; Graduate School of Mathematical Sciences, The University of Tokyo and MIMS, Meiji University). The research was partially supported by JSPS KAKENHI 16H02151. QG was partially supported by a PEPS-JCJC grant from CNRS (2019).}

\begin{center}
    \begin{minipage}{0.8\textwidth}
	\centering
    \LARGE{\bf Propagation dynamics of solutions to spatially periodic reaction-diffusion systems with hybrid nonlinearity}\bigskip
    \end{minipage}

    \Large
    Quentin Griette\medskip \\
    \normalsize
    {\it Institut de Math\'ematiques de Bordeaux, Universit\'e de Bordeaux, \\
    CNRS, IMB, UMR 5251, \\ 
    351, cours de la Lib\'eration, F-33400 Talence, France.}\\
    {\tt quentin.griette@u-bordeaux.fr}
    \bigskip

    \Large
    Hiroshi Matano \medskip \\ 
    \normalsize
    {\it Meiji Institute for Advanced Study of Mathematical Sciences, Meiji University, \\
    4-21-1 Nakano, Tokyo 164-8525, Japan 
    }\\
    {\tt matano@meiji.ac.jp}

\end{center}

\begin{abstract}
    In this paper we investigate the dynamical properties of a {spatially periodic} reaction-diffusion system {whose reaction terms are of hybrid nature in the sense that they are partly competitive and partly cooperative depending on the value of the solution. This class of problems includes various biologically relevant models and in particular many models focusing on the Darwinian evolution of species. We start by studying the principal eigenvalue of the associated differential operator and establishing a minimal speed formula for linear monotone systems. In particular, we show that the {\it generalized Dirichlet principal eigenvalue} and the periodic principal eigenvalue may not coincide when the reaction matrix is not symmetric, in sharp contrast with the case of scalar equations. We establish a sufficient condition under which equality holds for the two notions. We also show that the propagation speed may be different depending on the  direction of propagation, even in the absence of a first-order advection term, again in a sharp contrast with scalar equations. 
Next we reveal the relation between the hair-trigger property of front propagation and the sign of the periodic principal eigenvalue.  
    Finally, we discuss the linear determinacy of the propagation speed and also establish the existence of travelling waves travelling whose speeds greater than the minimal speed is also proved. We apply our results to an important class of epidemiological models with genetic mutations.
    }
\end{abstract}
\section{Introduction}

In this paper we are interested in the following reaction-diffusion system:
\begin{equation}\label{eq:gen-sys}
	\begin{system}
	    \relax &\vect{u}_t=\mathcal L\vect{u}+\vect{f}(x, \vect{u}),& t>0, x\in\mathbb R, \\
		\relax &\vect{u}(t=0, x)=\vect{u}_0(x), & x\in\mathbb R,
	\end{system}
\end{equation}
where $u(t, x)\in\mathbb R^d$ is a nonnegative vector-valued function of a space variable $x\in\mathbb R$ and a time variable $t\geq 0$; $\mathcal L$ is a diagonal matrix of second-order elliptic differential operators with {spatially} $L$-periodic coefficients and $f(x, u)$ is a reaction term {that} is $L$-periodic {in} $x\in\mathbb R$. We will assume throughout the article that $f(x, u)$ is cooperative when $u$ {lies in the vicinity of} the boundary of the positive cone of $\mathbb R^d$. An important example, which {has motivated} the current study, is the {following} two-components system 
\begin{equation}\label{eq:main-sys}
	\begin{system}
	    \relax &u_t=\sigma_u(x) u_{xx}+(r_u(x)-\kappa_u(x)(u+v))u+\mu_v(x)v-\mu_u(x)u, & t>0, x\in\mathbb R, \\
		\relax &v_t=\sigma_v(x) v_{xx}+(r_v(x)-\kappa_v(x)(u+v))v+\mu_u(x)u-\mu_v(x)v, &t>0, x\in\mathbb R.
	\end{system}
\end{equation}
Here $u(t, x)$, $v(t,x) $ stand for the density of a population of individuals living in  a periodically heterogeneous environment. We assume that  the reproduction rates $r_u(x) $ and $ r_v(x)$ are $L$-periodic functions, and that the competition coefficients $\kappa_u(x)$ and $\kappa_v(x)$ are $L$-periodic and positive. Finally, the coefficients $\mu_u(x)>0$, $\mu_v(x)>0$ {(also $L$-periodic)} denote the mutation rates between the two populations, {which creates an effect of}  cooperative coupling in the region where both $u$ and $v$ are small.

{In the context of epidemiology, System \eqref{eq:main-sys} describes}  the propagation of a genetically unstable pathogen in a population of hosts which {exhibits} heterogeneity in space. This heterogeneity may {simply come from a heterogeneous repartition of the host population \cite{Sat-Mat-Sas-94}.} {Spatial heterogeneity in the use of antibiotics, fungicides or insecticides affects the transmission of pathogens and pests and is explored as a way to minimize the risk of emergence of drug resistance \cite{Deb-Len-Gan-09}.}    
{Beaumont et al \cite{Bea-Bur-Duc-Zon-12} study a related model of propagation of salmonella in an industrial hen house. In their study the heterogeneity comes from the alignment of cages separated by free space that allow farmers to take care of the animals.}

System \eqref{eq:main-sys} has some similarity with Fisher-KPP equations. In their seminal work of 1937, Fisher \cite{Fis-1937} and Kolmogorov, Petrovsky, and Piskunov \cite{Kol-Pet-Pis-1937} introduced the following model, later called Fisher-KPP equation,
\begin{equation}\label{eq:KPP}
    u_t-d u_{xx}=ru(1-u).
\end{equation}
They observed that there exists  traveling wave solutions of speed $c$ for $c\geq c^*:=2\sqrt{dr}$. They claimed that the spreading speed from localized initial data should coincide with $c^*$, the minimal speed of traveling waves.  
The spreading property starting from localized initial data was analysed more rigorously by Aronson and Weinberger \cite{Aro-Wei-75, Aro-Wei-78} and  Weinberger \cite{Wei-82}. {Not long after, people started to consider  equations in periodically heterogeneous environments; among others, the paper of Shigesada, Kawasaki and Teramoto \cite{Shi-Kaw-Ter-86} was a pioneer. A more systematic mathematical theory was developed later (see Xin \cite{Xin-00}, Berestycki and Hamel \cite{Ber-Ham-02} and Weinberger \cite{Wei-02}).} 

It is sometimes possible to compute the propagation speed of initially localized solutions to a reaction-diffusion equation by analyzing the one of the  linearized equation in a neighborhood of zero. When this happens, the equation is said to be {\it linearly deteminate}, {or we say that {\it linear determinacy} holds. This property has been studied for scalar equations (see, for instance, \cite{Wei-82}) but also in the context of homogeneous systems (see Lui \cite{Lui-89} and Weinberger, Lewis and Li \cite{Wei-Lew-Li-02}). In the case of systems, it is required that both  the nonlinear and the linear equations be order-preserving in time. An important example equations that are monotone in time are reaction-diffusion systems which are coupled only in their zero-order term, and which coupling is {\it cooperative}. In the case of matrices, this means that all off-diagonal entries are nonnegative.} Traveling wave for cooperative systems have also been studied, by Volpert, Volpert and Volpert \cite{Vol-Vol-Vol-94} and Ogiwara and Matano \cite{Ogi-Mat-99-DCDS, Ogi-Mat-99-PRSE}, among others.
Spreading speeds and linear determinacy have been generalized to Banach lattices in the work of Liang and Zhao \cite{Lia-Zha-10}.

In the case of System \eqref{eq:main-sys} with homogeneous coefficients, it can be seen that it is cooperative near 0 because the quadratic term can be neglected. However,  far away from the unstable equilibrium $\vect{0}$, the nonlinearity becomes competitive in general (that is, if $\mu_u$ and $\mu_v$ are not too big). This means that equation \eqref{eq:main-sys} is of hybrid nature, so that the theory developed in \cite{Lui-89, Wei-Lew-Li-02} cannot be used directly.  In fact, solutions to System \eqref{eq:main-sys} actually reach a non-monotone regime (see, in particular, \cite{Gri-Rao-16} in which non-monotone waves are constructed).   Several other models consider traveling waves in a non-monotone setting. Hsu and Zhao \cite{Hsu-Zha-08}, for instance, considered a non-monotone integro-difference equation that is different from ours and proved the existence of traveling waves of speed $c$ for any $c$ greater than a minimal speed $c^*$ --- as is the case in many KPP situations. Their idea is to construct super and sub-solutions by replacing the nonlinearity by its monotone envelope from above and from below. Such a method cannot be applied directly to our system, unfortunately, since our equation is vector-valued. 

In a spatially homogeneous setting, several results exist already for systems of reaction-diffusion equations for which the monotone theory does not apply directly.  Let us mention the work of Wang \cite{Wan-11}, who studied spreading speeds and traveling waves for non-monotone systems in a case where the nonlinearity can be framed by two cooperative functions. Morris, B\"orger and Crooks \cite{Mor-Bor-Cro-19} studied a two-component system quite similar to \eqref{eq:main-sys} and gave precise estimates on the tails of the fronts. Girardin \cite{Gir-18, Gir-18-MMMAS} proved the existence of traveling waves and studied their asymptotic behavior in a quite general setting of homogeneous KPP-type systems similar to \eqref{eq:gen-sys}. Our approach here is different, since we want our argument to work for periodic coefficients and the canonical equation for traveling waves is not elliptic in this context. We use the Poincar\'e map of the  time-dependent problem and a fixed-point theorem to construct the traveling waves. In the process we use a monotone subsystem to obtain a lower estimate on the solution.

In the case of a system with spatially homogeneous coefficients, we can further prove the convergence of the traveling waves and time-dependent problem to the unique constant stationary solution in many cases (Theorem \ref{thm:ltb}). Previously there were results on the existence of traveling waves \cite{Gri-Rao-16, Gir-18} and their qualitative behavior \cite{Gir-18-MMMAS, Mor-Bor-Cro-19} but the long-time convergence to a stationary solution was only studied in a bounded domain \cite{Can-Cos-Yu-18}.

In the case of periodically heterogeneous equations, pulsating traveling waves for \eqref{eq:main-sys} traveling at the candidate minimal speed were constructed in \cite{Alf-Gri-18} by a vanishing viscosity method applied to the equation in the moving frame, but the minimality of the speed was not known. Here we not only show that this speed is indeed minimal, but also prove that it corresponds to the spreading speed of front-like initial data and construct traveling waves for larger speeds. The crucial remark which allows such a construction is that one can identify a cooperative system to which any solution of \eqref{eq:main-sys} is a supersolution, which provides a way to estimate the solutions to \eqref{eq:main-sys} from below. 

{Before stating our results, let us discuss some  technical notions. One of the first natural questions that one might ask when investigating models like \eqref{eq:gen-sys}  is whether a population can survive in time. Indeed, in equation \eqref{eq:main-sys} for instance, taking $r_u(x)\leq-\delta<0$ and $r_v(x)\leq -\delta$ leads to the global extinction of any solution starting from a bounded initial data.  It turns out that, for our class of problems, the answer to this question only depends on  the linearization of \eqref{eq:gen-sys} and, more precisely, on the {\it generalized principal eigenvalue} $\lambda_1^\infty$ defined as the limit, as $R\to\infty$, of the principal eigenvalue in the bounded domain $(-R, R)$,
\begin{equation}\label{eq:Dir-princ-eig}
    \left\{\begin{aligned}\relax
	&-\mathcal L \varphi^R = Df(x, 0)\varphi^R + \lambda_1^{R}\varphi^R, \\
	&\varphi^R(-R)=\varphi(R)=0, 
    \end{aligned}\right.
\end{equation}
where $Df(x, u)$ is the Jacobian matrix of $f$ in the variable $u$ only, 
under the requirement that $\varphi^R(x)> 0$ componentwise on $(-R, R)$. That $\lambda_1^R$ is unique and that it admits a limit when $R\to+\infty$ is classical but will be recalled in the present paper (Proposition \ref{prop:gen-princ-eig}). To distinguish it from other notions of principal eigenvalues (see Berestycki and Rossi \cite{Ber-Ros-06, Ber-Ros-15} and Nadin \cite{Nad-09} for an overview of these notions) we will often call $\lambda_1^\infty$ the {\it generalized Dirichlet principal eigenvalue}.  

The generalized Dirichlet principal eigenvalue characterizes the survival of {\em compactly supported initial data}. More precisely, any solution starting from nontrivial compactly supported initial data becomes uniformly positive as $t\to +\infty$ when $\lambda_1^\infty<0$, and some solution gets extinct when $\lambda_1^\infty>0$. Another important notion of principal eigenvalue that will be used in the present paper is the {\it periodic principal eigenvalue}, defined as the solution to 
\begin{equation}\label{eq:main-def-k}
    \left\{\begin{aligned}\relax
	&-\mathcal L\varphi^{per} = Df(x, 0)\varphi^{per} + \lambda_1^{per} \varphi^{per}, \\
	&\varphi^{per} \text{ is $L$-periodic}, 
    \end{aligned}\right.
\end{equation}
under the requirement that $\varphi^{per}(x)> 0$ componentwise on $\mathbb R$. This notion characterizes the survival of {\it periodic initial data}. More precisely, any solution starting from nontrivial periodic initial data becomes uniformly positive as $t\to +\infty$ when $\lambda_1^{per}<0$, and any bounded solution gets extinct when $\lambda_1^{per}>0$. In a way, $\lambda_1^\infty$ characterizes the {\it survival of the species in an initially empty space} (compactly supported initial data) and $\lambda_1^{per}$ characterizes the {\it survival of the species in an already invaded space} (periodic initial data).

There is no necessity in general that these two notions be equal; the most that can be said is that
\begin{equation*}
    \lambda_1^{per}\leq \lambda_1^\infty.
\end{equation*}
An interpretation of this inequality is that {\it it is more difficult to survive in an empty space than in an already invaded space}. It turns out that, for scalar reaction-diffusions without a first-order (advection) term, the equality $\lambda_1^{per}=\lambda_1^\infty$ is always true. This fact was remarked by Nadin \cite[Proposition 3.2]{Nad-09}. As we will see in the present paper (Proposition \ref{prop:strong-coupling}), the situation for systems is in sharp contrast with what happens for scalar equations, as it is possible to construct a system with no advection and $\lambda_1^{per}<\lambda_1^\infty$. We recover the equality $\lambda_1^{per}=\lambda_1^\infty$ between the two notions under some symmetry assumption on the coefficients of the equation, detailed in Assumption \ref{as:isotropic} (in particular, it is true for constant coefficients).

Next we turn to the formula for the propagation speed. It involves again spectral notions related to the linearized problem, this time the function $k(\lambda)$ defined as 
\begin{equation}\label{eq:main-def-k}
    \left\{\begin{aligned}\relax
	& -e^{\lambda x}\mathcal L\big( \varphi^\lambda e^{-\lambda x}\big) = Df(x, 0)\varphi^\lambda + k(\lambda) \varphi^\lambda, \\
	&\varphi^{\lambda} \text{ is $L$-periodic}, 
    \end{aligned}\right.
\end{equation}
under the requirement that $\varphi^\lambda(x)> 0$ componentwise on $\mathbb R$. This function will be extensively studied in Proposition \ref{prop:k(lambda)}. The propagation speed towards $+\infty$ of solutions starting from  front-like initial data supported in $(-\infty, 0)$  to \eqref{eq:gen-sys} can then be expressed as
\begin{equation}\label{eq:formula-speed}
	c^*=\inf_{\lambda>0}\frac{-k(\lambda)}{\lambda},
\end{equation}
which is a well-known formula in the scalar case \cite{Xin-00, Wei-02, Ber-Ham-02, Ber-Ham-Roq-05-II}. However, once again, systems do not behave exactly like scalar equations. When investigating the speed of towards $-\infty$ of solutions starting from  front-like initial data supported in $(+\infty, 0)$  to \eqref{eq:gen-sys}, we naturally arrive at the formula 
\begin{equation}\label{eq:formula-speed-left}
    c^*_{\text{left}}=\inf_{\lambda>0}\frac{-k(-\lambda)}{\lambda},
\end{equation}
which is not necessarily equal to $c^*_{\text{right}}$ defined by \eqref{eq:formula-speed}. For scalar equations without advection, it turns our that the equality $c^*_{\text{left}}=c^*_{\text{right}}$ is always true, because the function $k(\lambda)$ is even (this can be seen from \cite[proof of Proposition 3.2]{Nad-09}). In the context of systems it is possible to construct counterexamples in which $c^*_{\text{left}}\neq c^*_{\text{right}}$ even though there is no advection (Remark \ref{rem:strong-coupling-speed}). Thus the situation for systems is, once again, in sharp contrast with the one of scalar equations. We recover the equality $c^*_{\text{left}}=c^*_{\text{right}}$ (Proposition \ref{prop:isotropic-speed}) under an additional symmetry assumption on the coefficients of the equation (Assumption \ref{as:isotropic}).
}

{Among other main results of the paper, we show the linear determinacy (Theorem \ref{thm:lin-det}) and existence of traveling waves (Theorem \ref{thm:TW}) for solutions to \eqref{eq:gen-sys} with sublinear nonlinearity, under some additional requirements. We require in particular that the Jacobian matrix be cooperative and irreducible.
We also} study the case of rapidly oscillating coefficients and show that the qualitative properties of such systems are very close to the ones of homogeneous systems. {Regarding the general system \eqref{eq:gen-sys}, we prove a homogenization formula for the speed (Theorem \ref{thm:large-diff}). The homogenization limit allows us to study the particular case of \eqref{eq:main-sys} in more details.} In particular, we prove the existence, uniqueness and global stability of the equilibrium for rapidly oscillating coefficients under some conditions,  by using dynamical system arguments (see Theorem \ref{thm:rapidosc}). This gives a non-trivial example of non-homogeneous systems for which the global behavior can be determined. This part of the study is based on the homogenization theory for elliptic and parabolic operators, see {\it e.g.} \cite{Ben-Lio-Pap-11} for an introduction to the theory. In the case of scalar equations, the homogenization limits of spreading speeds and pulsating traveling waves have been studied by El Smaily \cite{ElS-08, ElS-10} and El Smaily, Hamel and Roques \cite{ElS-Ham-Roq-09}.

The structure of the paper is as follows. In Section \ref{sec:main-results} we state our main results, concerning the original system \eqref{eq:gen-sys} with sub-linear nonlinearity and the particular case \eqref{eq:main-sys}. In Section \ref{sec:KPP-type} we prove the results in the general framework of KPP-type nonlinearities for $d$-dimensional systems. In Section \ref{sec:singular-limits} we propose two singular limits of our systems.  In Section \ref{sec:long-time}, we prove the results which are specific to the model \eqref{eq:main-sys}, including the local stability of the constant equilibrium, global stability under more restrictive assumptions and the homogenization limit.

\section{Main results}
\label{sec:main-results}

In this Section we state the main results presented in the paper. We first state the results we obtain on the specific equation \eqref{eq:main-sys}, then present the more general results on generic one-dimensional systems.

Our interest lies in systems of the form
\begin{equation}\label{eq:gen-syst}
	\vect{u}_t =\mathcal L\vect{u} + \vect{f}(x, \vect{u}), 
\end{equation}
set on the real line, 
where $\mathcal L$ is an elliptic differential operator written either in divergence form 
\begin{equation}\label{eq:divergence-form}
	\mathcal L\vect{u}=\mathcal L^d\vect{u}:=(\vect{\sigma}(x)\vect{u}_{x})_x + \vect{q}(x)\vect{u}_x,
\end{equation}
 or in nondivergence form 
\begin{equation}\label{eq:nondivergence-form}
	 \mathcal L\vect{u}=\mathcal L^{nd}\vect{u}:=\vect{\sigma}(x)\vect{u}_{xx} + \vect{q}(x)\vect{u}_x,
\end{equation}
where  $\vect{\sigma}\in C^{1, \alpha}_{per}(\mathbb R,  M_d(\mathbb R))$ is a positive diagonal matrix field,  $\vect{q}\in C^{\alpha}_{per}(\mathbb R,  M_d(\mathbb R))$ is a diagonal matrix field, and $\vect{f}\in Lip(\mathbb R\times\mathbb R^d, \mathbb R^d)$ are $L$-periodic in the variable $x$.
First we introduce some definitions and notations.

\paragraph{Order on $\mathbb R^d$} Let $u=(u_1, \ldots, u_d)^T\in \mathbb R^d $ and $ v=(v_1, \ldots, v_d)^T\in\mathbb R^d $. We denote by  $u\leq v$ the component-wise order of $\mathbb R^d$, that is to say
\begin{equation*}
	u\leq v \quad \Longleftrightarrow \quad \left( u_i\leq v_i \text{ for all } i\in\{1, \ldots, d\}\right).
\end{equation*}
Recall that $(\mathbb R^d, \leq) $ is a Banach lattice which positive cone is $\mathbb R^d_+:=\{u\in\mathbb R^d\,|\, u\geq 0\}$. We will use the notation $u\ll v$ to denote the component-wise strict order
\begin{equation*}
	u\ll v \quad \Longleftrightarrow \quad \left( u_i< v_i \text{ for all } i\in\{1, \ldots, d\}\right).
\end{equation*}
When $u$ and $v$ are two vectors, then $\min(u,v) $ and $\max(u, v) $ are the usual component-wise minimum and maximum of $u$ and $v$:
\begin{equation*}
	\min(u, v) = \big(\min(u_1, v_1), \ldots, \min(u_d, v_d)\big)^T, \quad \max(u, v) = \big(\max(u_1, v_1), \ldots, \max(u_d, v_d)\big)^T.
\end{equation*}
Finally we denote $\mathbf{1}:=(1, 1, \ldots, 1)^T\in \mathbb R^d$ the $d$-dimensional vector with all components equal to $1$.

\subsection{The linear problem. Principal eigenvalues and spreading speeds.}

We first focus on the linear part of System \eqref{eq:gen-syst}, that is when $f(x, u)$ is a linear function of $u$. Our interest lies in systems which preserve the canonical partial order on $\mathbb R^d$.  
\begin{defn}[Cooperative matrix]
Let $A(x)=(a_{ij}(x))_{1\leq i, j\leq d} $ be a matrix-valued function (from $\mathbb R$ to $M_d(\mathbb R)$).  $A(x)$ is \textit{cooperative} if $a_{ij}(x)\geq 0$ for all $i\neq j$ and $x\in\mathbb R$.
\end{defn}

Next we introduce the notion of fully coupled system. This corresponds, in a way,  to systems that cannot be split into two independent subsystems. 

\begin{defn}[Fully coupled matrix.]\label{def:fully-coupled} 
	Let $A(x)=(a_{ij}(x))_{1\leq i, j\leq d} $ be a matrix-valued function (from $\mathbb R$ to $M_d(\mathbb R)$). We say that $A(x)$  {\it fully coupled} if there exists $\nu>0$  and  $r>0$ such that for any non-trivial partition $ I, J\subset \{1, \ldots, d\}$ (i.e. $I\cup J=\{1,\ldots, d\}$ and $I\cap J=\varnothing$), there exists $i\in I$, $j\in J$, and a ball $B(x_{ij}, r)$ for some $x_{ij}\in\mathbb R$, such that 
			\begin{equation}\label{eq:local-irreducible}
				\underset{x\in B(x_{ij}, r)}{\inf}a_{ij}(x)\geq \nu>0. 
			\end{equation}
\end{defn}
Note that, if $A(x) $ is a constant matrix, then it is fully coupled in the sense introduced above if and only if it is cooperative and irreducible. 

We suspect that the ball $B(x_0, r)$ above could be replaced by a measurable set of positive Lebesgue measure,  as is done in \cite{Bus-Sir-04}, but we will not pursue such generality as it would add unnecessary complexity to the proofs; moreover it is not essential in our analysis.

As usual in {sublinear} situations, the principal eigenvalue of the system under consideration plays a crucial role in the survival of the population. We define the notion of {\it periodic principal eigenvalue} in the case of systems with $d$ components
\begin{defn}[Principal eigenpairs]\label{def:princ-eig}
	Let $A(x) $ be a fully coupled cooperative matrix function and $\mathcal L$ be a diagonal uniformly elliptic operator.

	By  a {\it periodic principal eigenpair} associated with system \eqref{eq:gen-syst}, we refer to a solution pair $(\lambda_1^{per}, \vect{\varphi}^{per}(x))$ to the system
	\begin{equation}\label{eq:def-princ-per}
		-\mathcal L\vect{\varphi}^{per} = A(x)\vect{\varphi}^{per}(x) + \lambda_1^{per}\vect{\varphi}^{per}(x)
	\end{equation}
	under the $L$-periodic boundary conditions, satisfying $\varphi^{per}(x)\gg 0$.
	Here $\lambda_1^{per}$ is called the {\it periodic principal eigenvalue} and $\vect{\varphi}^{per}(x)$  a {\it periodic principal eigenfunction} or {\it periodic principal eigenvector}.

	Similarly, by {\it Dirichlet principal eigenpair} associated with system \eqref{eq:gen-syst} in the interval of radius $R>0$ we refer to a solution pair $(\lambda_1^{R}, \vect{\varphi}^R(x))$ to the system
	\begin{equation}\label{eq:def-princ-Dir}
		-\mathcal L\vect{\varphi}^R = A(x)\vect{\varphi}^R(x) + \lambda_1^{R}\vect{\varphi}^R(x)\quad \text{ in } (-R, R),
	\end{equation}
	satisfying the  Dirichlet  boundary conditions $\vect{\varphi}^R(-R)=\vect{\varphi}^R(R)=\vect{0}$ and $\varphi^R(x)\gg 0$ in $(-R, R)$. 
	Here $\lambda_1^{R}$ is called the {\it Dirichlet principal eigenvalue} and $\vect{\varphi}^R(x)$ a {\it Dirichlet principal eigenfunction} or {\it Dirichlet principal eigenvector}. We denote 
	\begin{equation}
		\lambda_1^\infty:=\lim_{R\to\infty}\lambda_1^R.
	\end{equation}
\end{defn}
The existence of the principal eigenpairs $(\lambda_1^{per}, \varphi^{per})$ and $(\lambda_1^R, \varphi^R)$ and the uniqueness of $\lambda_1^{per} $ and $\lambda_1^R$ follow immediately from the Krein-Rutman Theorem. Moreover there always holds $\lambda_1^{R'}>\lambda_1^R>\lambda_1^{per}$ for $0<R'<R$. Consequently, we have
\begin{equation}\label{eq:order-eigenvalues}
	\lambda_1^\infty\leq \lambda_1^{per}.
\end{equation}
The above two notions of principal eigenvalue  correspond to very different qualitative properties of the solutions to \eqref{eq:gen-syst}. The Dirichlet eigenvalue $\lambda_1^R$ gives a criterion for the {survival in the bounded domain $ (-R, R)$} under the Dirichlet boundary conditions, and $\lambda_1^\infty=\lim_{R\to\infty}\lambda_1^R$ characterizes, in a sense, the {survival of solutions with compactly supported initial conditions} on the real line. More precisely, the species does not get extinct if $\lambda_1^\infty<0$ (see Theorem \ref{thm:hair-trigger} below), while it  converges to $0$ (extinction) as $t\to\infty$ if $\lambda_1^\infty>0$. On the other hand, $\lambda_1^{per} $ characterises the {survival of solutions starting from positive periodic initial conditions} or, more generally, the {sustainability of an already invaded space}. We will state a condition under which both eigenvalues have the same sign in Proposition \ref{prop:isotropic}. 

{Whether or not equality holds in \eqref{eq:order-eigenvalues} depends on the situation. {Proposition \ref{prop:strong-coupling} provides a counterexample to the equality in \eqref{eq:order-eigenvalues} in the case of a system of two equations with no advection term and  strong coupling.}  These properties, along with some other properties of those eigenpairs, will be proved in subsection \ref{sec:eigenpb} and in Appendix B. Here we collect some useful properties of the principal eigenvalues.}
\begin{prop}[On the  Dirichlet principal eigenvalue for cooperative systems]\label{prop:gen-princ-eig}
	Let $ A(x) $ be a cooperative and fully coupled $d$-dimensional $L$-periodic matrix field, $\mathcal L$ be a $L$-periodic diagonal uniformly elliptic operator.
	Then:
	\begin{enumerate}[label={\rm(\roman*)}]
		\item \label{item:eigenfun}
			For any $R\in (0, +\infty)$, there exists a principal eigenfunction ${\varphi}^R>0$ associated with $\lambda_1^R$, which is unique up to the multiplication by a positive scalar. 
		\item \label{item:princ-eig}
			For any $R\in(0,+\infty)$, we have 	
			\begin{equation}\label{eq:gen-princ-eig}
				\lambda_1^R:=\sup\{\lambda\in\mathbb R\,|\,\exists {\phi}\in C^2((-R, R), \mathbb R^d)\cap C^1([-R, R], \mathbb R^d), {\phi}>0, -\mathcal L{\phi} - A(x){\phi} -\lambda{\phi} \geq { 0}\}. 
			\end{equation}
		\item\label{item:princ-eig-dec-R}
			The mapping $R\mapsto \lambda_1^R$ is decreasing.
		\item \label{item:limitRbig}
			There exists a positive eigenfunction associated with $\lambda_1^\infty$.
		\item \label{item:minimax}
			For $R\in (0, +\infty]$, we have 
			\begin{equation}\label{eq:eigprinc-minimax}
				\lambda_1^R=\underset{\phi>0}\max\underset{x\in (-R, R)}{\inf} \underset{1\leq i\leq d}{\min}\frac{(-\mathcal L\phi -A(x)\phi)_i}{\phi_i}
			\end{equation}
	\end{enumerate}
\end{prop}
Next we introduce an important object for the study of the spatial behavior of the solutions to  \eqref{eq:gen-syst}. Given $\lambda\in\mathbb R$, we let $L_\lambda {\phi}(x):=e^{\lambda x}\mathcal L(e^{-\lambda x}{\phi}(x))$ and $k(\lambda)$  be the principal eigenvalue of the operator $-L_\lambda -A(x)$ restricted on $L$-periodic functions. Equivalently, we have
\begin{equation*}
	L_\lambda \vect{\phi}=(\vect{\sigma}(x)\vect{\phi}_{x})_x +(-2\lambda \vect{\sigma}(x) +\vect{q}(x))\vect{\phi}_x+(-\lambda\vect{\sigma}_x(x)-\lambda \vect{q}(x)+\lambda^2 \vect{\sigma}(x))\vect{\phi},
\end{equation*}
if $\mathcal L$ is written in divergence form, or
\begin{equation*}
	L_\lambda \vect{\phi}=\vect{\sigma}(x)\phi_{xx} +(-2\lambda \vect{\sigma}(x) +\vect{q}(x))\vect{\phi}_x+(-\lambda \vect{q}(x)+\lambda^2 \vect{\sigma}(x))\vect{\phi},
\end{equation*}
if $\mathcal L $ is written in nondivergence form, and $k(\lambda)$ is the unique real number for which there exists a solution $\phi>0$ to
\begin{equation}\label{eq:defk(lambda)}
	\begin{system}
		\relax &-L_\lambda\vect{\phi}(x)-A(x)\vect{\phi}(x)=k(\lambda)\vect{\phi}(x),& x&\in\mathbb R, \\
		\relax &\vect{\phi}>\vect{0},\qquad \vect{\phi}\text{ is }L\text{-periodic}.
	\end{system}
\end{equation}
The map $k(\lambda)$ plays a crucial role regarding the spatial properties of \eqref{eq:gen-syst} as it the center of a formula for the spreading speed associated with \eqref{eq:gen-syst}. It also provides a connection between the generalized Dirichlet an periodic principal eigenvalue, as will be stated in the next Proposition. {However, in order to state our results, we first need to introduce an assumption which ensures that the dynamics of the model is the same in both directions of $\mathbb R$. 
\begin{assumption}[Isotropic behavior]\label{as:isotropic}
	We assume that the operator $\mathcal L$ has no advection: ${q}(x)\equiv {0}$. Furthermore,  we assume that either of the following conditions is satisfied.  
	\begin{enumerate}[label={\alph*)}]
		\item  Both ${\sigma}(x)$ and $A(x)$ are even in $x$,
		\item  $\mathcal L=\mathcal L^d$ is written in divergence form \eqref{eq:divergence-form} and $A(x)$ is a symmetric matrix.
	\end{enumerate}
\end{assumption}
}
\begin{prop}[On $k(\lambda)$]\label{prop:k(lambda)}
	Let $\mathcal L$ be a $L$-periodic diagonal uniformly elliptic operator, $A(x)$ be a cooperative and fully coupled $L$-periodic matrix field. Then:
	\begin{enumerate}[label={\rm(\roman*)}]
		\item \label{item:eigenpairk(lambda)}
			For each $\lambda\in\mathbb R$, there exists a principal eigenpair $(k(\lambda), \vect{\phi}^\lambda(x))$ with $\phi^\lambda\gg 0$ which solves \eqref{eq:defk(lambda)}, and $\vect{\phi}^\lambda$ is unique up to the multiplication by a positive scalar.
		\item \label{item:minimaxk(lambda)}
			The following characterization of $k(\lambda)$ is valid:
			\begin{equation}\label{eq:minimax-k(lambda)}
				k(\lambda)=
				\underset{\phi\in C^2_{per}(\mathbb R), \mathbb R^d}{\underset{{{\phi}>{0}}}{\max}}\underset{x\in \mathbb R}{\inf}\, \underset{1\leq i\leq d}{\min}\frac{(-L_\lambda{\phi} -A(x){\phi})_i}{\phi_i}
			\end{equation}
		\item \label{item:k(lambda)concave}
			The map $\lambda\mapsto k(\lambda) $ is analytic and strictly concave. {Furthermore, there exist constants $\alpha>0$ and $\beta>0$ such that 
			\begin{equation}\label{eq:klesssquare}
				k(\lambda) \leq \alpha -\beta \lambda^2 \text{ for all } \lambda\in\mathbb R.
			\end{equation}
			}
		\item \label{item:lambda_1k(lambda)}
			The following equality holds:
			\begin{equation*}
				\lambda_1^\infty=\max_{\lambda\in\mathbb R}k(\lambda).
			\end{equation*}
		\item \label{item:k(lambda)even}
			If $\mathcal L$ satisfies Assumption \ref{as:isotropic},  then the mapping $\lambda\mapsto k(\lambda)$ is even.
	\end{enumerate}
\end{prop}
Finally, we introduce a formula which gives the minimal speed of traveling waves to \eqref{eq:gen-syst}, and show some related properties. 
\begin{prop}[On the formula for the minimal speed]\label{prop:minspeed}
    Let $\mathcal L$ be a $L$-periodic diagonal uniformly elliptic operator, $A(x)$ be a cooperative and fully coupled $L$-periodic matrix field. Suppose that $\lambda_1^{per}<0$ and let 
	\begin{equation}\label{eq:defc*}
		c^*:=\inf_{\lambda>0}\frac{-k(\lambda)}{\lambda}.
	\end{equation}
	Then:
	\begin{enumerate}[label={\rm(\roman*)}]
		\item if $c<c^*$, then for any $\lambda>0$, we have
			$c\lambda<-k(\lambda)$,
		\item if $c=c^*$, then there exists a unique $\lambda^*>0$ such that $\lambda^*c^*=k(\lambda^*)$, and for any $\lambda>0$ with $\lambda\neq\lambda^*$ we have $\lambda c^*<-k(\lambda)$, 
		\item if $c>c^*$, there exists $\lambda_1^* <\lambda_2^*$ such that $\lambda_1^*c=-k(\lambda_1^*) $ and $\lambda_2^*c=-k(\lambda_2^*) $. For $\lambda \in (\lambda_1^*, \lambda_2^*) $ we have $\lambda c>-k(\lambda)$, while for $\lambda \not\in [\lambda_1^*, \lambda_2^*]$ we have $\lambda c < -k(\lambda)$.
		\item $c^*$ is continuous in $A$ with respect to the supremum norm.
	\end{enumerate}
\end{prop}
As we will discuss in Theorem \ref{thm:lin-det}, the speed $c^*$ defined by \eqref{eq:defc*} is the natural speed of propagation of solutions to \eqref{eq:gen-syst} starting from  front-like initial data $u_0$ supported in $(-\infty, 0]$ and with $\liminf_{x\to -\infty}u_0(x)>0$. In order to catch the propagation speed of solutions starting from initial data supported in  $[0, +\infty)$ and with $\liminf_{x\to +\infty}u_0(x)>0$, it suffices to introduce the quantity
\begin{equation}\label{eq:defc*L}
    c^*_{\text{left}} := \inf_{\lambda<0}\frac{k(\lambda)}{\lambda}=\inf_{\lambda>0}\frac{-k(-\lambda)}{\lambda}. 
\end{equation}
Whether the rightward speed $c^*$ and the leftward speed $c^*_{\text{left}}$ are equal depends, {again,} on the situation. In many cases, including the case of constant coefficients, such a property is true. A sufficient condition for this property to hold is given in Assumption \ref{as:isotropic}. However it is false in general {as can be seen as a straightforward consequence of Proposition \ref{prop:strong-coupling}; see Remark \ref{rem:strong-coupling-speed}.}

A consequence of the above results is the following Proposition.
\begin{prop}\label{prop:speed-survival}
    Suppose that $\lambda_1^{per}<0$ and that the rightward speed $c^*=:c^*_{\text{right}}$ (defined by \eqref{eq:defc*}) and the leftward speed (defined by \eqref{eq:defc*L}) are both positive.  Then
	\begin{equation*}
		\lambda_1^\infty<0.
	\end{equation*}
\end{prop}
The proof is immediate so we sketch it here. When $c^*_{\text{right}}>0$, it follows from the definition of $c^*_{\text{right}}$ that $0>-\lambda c^* \geq k(\lambda)$ for all $\lambda>0$. Similarly since $c^*_L>0$ there holds that $0>\lambda c^* \geq k(\lambda)$ for all $\lambda<0$. Finally since $k(0)=\lambda_1^{per}<0$, we have that $k(\lambda)<0$ for all $\lambda\in\mathbb R$ and it follows from \eqref{eq:klesssquare} that  
\begin{equation*}
	\lambda_1^\infty=\max_{\lambda\in\mathbb R}k(\lambda)<0. 
\end{equation*}
Proposition \ref{prop:speed-survival} is proved.

{When $\lambda_1^{per}=0$, it is known the hair-trigger property may fail for nonlinear problems even in the case of a scalar equation. This can be shown by considering the classical Fisher-KPP equation : 
\begin{equation*}
    u_t(t, x) = u_{xx}(t, x) +u(t, x)\big(r- u(t, x)\big). 
\end{equation*}
When $r=0$, any  bounded nonnegative solution to the above equation converges to $0$ as $t\to\infty$. This can be shown by comparison with the ODE $u_t=-u^2$.}

Next we derive a condition under which there is equality between the periodic principal eigenvalue and the limit of the Dirichlet eigenvalues in large balls. As in the scalar case (see \cite{Nad-09, Nad-10}), it may happen that the two principal eigenvalues $\lambda_1^{per} $ and $\lambda_1^\infty$  differ if for instance individuals are ``blown away'' to $\pm \infty$, even when the population is capable of surviving. Similarly, because of the dependency in $x$ in the diffusion coefficient, the speed of propagation may differ when looking at solutions spreading to the right or to the left. 
Under this assumption we have the following properties:
\begin{prop}[Dirichlet and periodic principal eigenvalues]\label{prop:isotropic}
	Let Assumption \ref{as:isotropic} hold. Then
	\begin{equation}\label{eq:isotropic}
		\lambda_1^\infty:=\lim_{R\to\infty}\lambda_1^R=\lambda_1^{per}.
	\end{equation}
\end{prop}
Since $\lambda_1^{per}$ is generally easier to estimate than $\lambda_1^\infty$, the above proposition gives a useful criterion for the survival of solutions whose initial data is compactly supported in view of Theorem \ref{thm:hair-trigger}.
\begin{prop}\label{prop:isotropic-speed}
	Under Assumption \ref{as:isotropic}, the rightward and leftward spreading speeds are the same.
\end{prop}
If Assumption \ref{as:isotropic} fails to hold, the rightward speed and the leftward speed may not be the same, even if there is no advection, i.e. $q(x)\equiv 0$. {As explained in Remark \ref{rem:strong-coupling-speed}, Proposition \ref{prop:strong-coupling} provides a counterexample in the case of strong coupling.}
	This is in {sharp} contrast with the scalar case, where it is known that the two speeds are always the same in the absence of an advection.

Last we turn our interest to systems with rapidly oscillating coefficients and give a description of the asymptotic behavior of the spreading speed. 
\begin{thm}[The speed of rapidly oscillating systems]\label{thm:large-diff}
	Let $\vect{\sigma}(x)>\vect{0}$, $\vect{q}(x)$ and $A(x)$ be $1$-periodic. For each $\varepsilon\in (0,1)$, let 
	\begin{equation*}
		\mathcal L^\varepsilon\vect{u}:= (\vect{\sigma}^\epsilon(x)\vect{u}_x)_x+\vect{q}^\varepsilon(x)\vect{u}=  \left(\vect{\sigma}\left(\frac{x}{\varepsilon}\right)\vect{u}_x\right)_x+\vect{q}\left(\frac{x}{\varepsilon}\right)\vect{u}
	\end{equation*}
	be a uniformly elliptic operator  and $A^\varepsilon(x):=A\left(\frac{x}{\varepsilon}\right)$ be a cooperative fully coupled matrix field. We let $c^*_\varepsilon$ be the spreading speed associated with $\mathcal L^\varepsilon$ and $A^\varepsilon(x)$.  Then, we have 
\begin{equation}\label{eq:lim-hom}
	\lim_{\varepsilon\to 0}{c^*_\varepsilon}=c^*(\overline{\mathcal L}^H+\overline{A}), 
\end{equation}
	where:
	\begin{equation*}
		\overline{\mathcal L}^H\vect{u}:= \overline{\vect{\sigma}}^H\vect{u}_{xx}+\overline{\vect{q}}^H \vect{u}_x, 
	\end{equation*}
	\begin{align*}
		\overline{\sigma}^H_i&:=\left(\int_0^1\frac{1}{\sigma_i(z)}\dd z\right)^{-1}, & \overline{q}^H_i&:= \overline{\sigma}^H_i\int_0^1\frac{q_i(z)}{\sigma_i(z)}\dd z, & \overline{A}:=\int_0^1A(z)\dd z, 
	\end{align*}
	and $c^*(\overline{\mathcal L}^H+\overline{A}) $ is given by:
	\begin{equation*}
		c^*(\overline{\mathcal L}^H+\overline{A})=\inf_{\lambda>0}\frac{\lambda_{PF}\left(\lambda^2\overline{\vect{\sigma}}^H - \lambda\overline{\vect{q}}^H+\overline{A}\right)}{\lambda}
	\end{equation*}
	where $\lambda_{PF}(X)$ is the Perron-Frobenius eigenvalue of an constant irreducible cooperative matrix $X$.
\end{thm}

\subsection{Spreading in {equations} of hybrid nature and traveling waves.}

In this subsection we derive some proerties of the solutions to the nonlinear equation \eqref{eq:gen-syst}. We first recall some notions that we will use in the statement of our results. 
\paragraph{Lipschitz continuity} Let $U\subset\mathbb R^d$ be given. We say that $f(x,u)=(f_1(x, u), \ldots, f_d(x, u)):\mathbb R\times U\to \mathbb R^d$ is locally Lipschitz continuous with respect to $u$ if, {for all $M>0$,}  there is a constant $K>0$ such that 
\begin{equation*}
	\Vert f(x, u)-f(x, v)\Vert\leq K\Vert u-v\Vert \text{ for all } x\in\mathbb R\text{ and }u,v\in U \text{ with } \Vert u\Vert \leq M \text{ and } \Vert v\Vert\leq M.
\end{equation*}

\begin{defn}[Cooperative function]\label{def:cooperative}
	Let $U\subset\mathbb R^d$ and $\vect{f}:\mathbb R\times U\to\mathbb R^d$. The function $\vect{f}=(f_1,\ldots, f_d)$ is {\it cooperative} (or equivalently, \textit{quasi-monotone}) on $U$ if there is a real number $\gamma>0$ such that $f(x, u)+\gamma u$ is monotone non-decreasing with respect to $u$ for the usual component-wise order. 
\end{defn}
\begin{rem}
	The notion of cooperative function is equivalent to the notion of quasi-monotonicity, which is more commonly used in the dynamical systems community.
\end{rem}
Alternatively, a function  $f$ is {cooperative} if, and only if, for any $x\in \mathbb R$, $\vect{u}=\left(u_1, \ldots, u_d\right)\in U$,  $1\leq i, j\leq d$ such that $i\neq j$,  the function $ v\mapsto f_i(x, u_1, \ldots, u_{j-1}, v, u_{j+1}, \ldots, u_d)$ is nondecreasing.

Next we define the notion of {sublinear} nonlinearity that we will use throughout the paper:
\begin{defn}[{Sublinear} nonlinearity]\label{def:KPP}
    We say the nonlinearity $\vect{f}(x, u) = (f_1(x, u), \ldots, f_d(x, u)) $ is {sublinear} provided it is continuous in both variables,  Lipschitz continuous with respect to $u$ and
	\begin{enumerate}[label={(\roman*)}]
		\item for all $x\in\mathbb R$,  $\vect{f}(x, 0)=0$.
		\item $ f(x, u)$ is differentiable at $u=0$ uniformly in $x$.
		\item for each $x\in\mathbb R$ and each $\vect{u}\geq 0$, we have 
			\begin{equation*}
				\vect{f}(x, \vect{u})\leq D\vect{f}(x, 0) \vect{u}.
			\end{equation*}
	\end{enumerate}
\end{defn}
Finally, in order to compute the spreading speed, we need an additional regularity assumption on the  properties of the nonlinearity in a vicinity of 0.
\begin{assumption}[Regularity in a neighborhood of 0]\label{as:coop}
    We assume that $f$ is a {sublinear} nonlinearity and that the differential matrix field $Df(x, 0)$ is cooperative and fully coupled. Moreover, we assume that there exists a family of cooperative and fully coupled matrix fields $(A^\delta(x))_{\delta\in(0, 1)}$ satisfying
	\begin{equation*}
		\sup_{x\in\mathbb R} \|A^\delta(x)-Df(x, 0)\|_{{M}_d(\mathbb R)}\to 0 \text { as }\delta\to 0,
	\end{equation*}
	and  for each $\delta\in (0,1)$ there exists $\eta>0$ such that whenever $\Vert u\Vert\leq \eta $ and $u>0$, we have
	\begin{equation*}
		f(x,u)\geq A^\delta(x)u.
	\end{equation*}
\end{assumption}
As an example of nonlinearity satisfying Assumption \ref{as:coop}, one can remark that if the Jacobian matrix $Df(x,0) $ has only positive coefficients, then $f$ satisfies Assumption \ref{as:coop} with $A^\delta(x)=(1-\delta)Df(x, 0)$. 

\begin{defn}[Monotone lower barrier] \label{def:super-monotone}
	Let $f=(f_1, \ldots, f_d)\in Lip(\mathbb R\times\mathbb R^d, \mathbb R^d)$ be a $L$-periodic function. We say that $f^-\in Lip(\mathbb R\times\mathbb R^d, \mathbb R^d)$ is a \textit{monotone lower barrier} of $f$  if there exists a constant $\eta>0$ such that 
	\begin{enumerate} 
		\item \label{item:lowercoop} $f(x, u)\geq f^-(x, u)$ for all $u=(u_1, \ldots, u_d)\geq 0$ with $u_i\leq \eta$ for some $i\in\{1, \ldots, d\}$.
		\item \label{item:samediff} $Df^-(x, 0)u=Df(x, 0)u$ for all $(x, u)\in\mathbb R\times \mathbb R^d$. 
		\item for all $i, j\in\{1, \ldots, d\}$ with $j\neq i$, the function $u_j\mapsto f_i^-(x, u_1, \ldots, u_d)$ is non-decreasing whenever $|u_i|\leq \eta$.
	\end{enumerate}
	Note that the above assumptions imply, in particular,  that $f^-$ is cooperative (or equivalently, quasi-monotone) in a neighborhood of $0$, more precisely  on the domain $\mathbb R\times B^+_\infty(0, \eta)$ (where $B_\infty(0, \eta):=\{u\geq 0\,|\,\Vert u\Vert_{\infty}\leq \eta\}$). 
\end{defn}

Equipped with these notions, we now state the first result on nonlinear equations of this paper. Theorem \ref{thm:hair-trigger} showed that there is a hair-trigger effect when the Dirichlet principal eigenvalue $\lambda_1^\infty$ is negative. More precisely, any solution starting from a non-trivial initial data becomes locally uniformly positive when $t\to+\infty$.
\begin{thm}[Hair-Trigger effect]\label{thm:hair-trigger}
    Let $\mathcal L$ be a diagonal uniformly elliptic operator and $f$ be a {sublinear} function. Assume that $f$ admits a monotone lower barrier in the sense of Definition \ref{def:super-monotone} {and} suppose finally that $\lambda_1^{\infty}<0$. Then there exists  $\delta>0$ with such that whenever $u(t, x)$ is a solution of \eqref{eq:gen-syst} with an initial condition $u(0, x):=u_0(x)$ which is non-negative and non-trivial, then
	\begin{equation*}
		{\liminf_{t\to+\infty}u(t, x) \geq \delta\mathbf{1}, }
	\end{equation*}
	uniformly in bounded sets of $\mathbb R$. 
\end{thm}

Next we introduce our notion of spreading speed.
\begin{defn}[Spreading speed]\label{def:prop-speed}
	The real number $c^*$ is the {\it spreading speed} associated with system \eqref{eq:gen-syst} if all non-negative solutions  $\vect{u}(t, x) $ of \eqref{eq:gen-syst}  satisfy
	\begin{enumerate}[label={(\roman*)}]
		\item if $\liminf_{x\to -\infty}u \gg 0$, then for each $c<c^*$ we have
			\begin{equation*}
				\liminf_{t\to+\infty}\inf_{ x\leq  ct}  u(t, x) \gg 0.
			\end{equation*}

		\item if there is $K\in\mathbb R$ such that  $u(0, x)\equiv 0$ for all $x\geq K$, then for all $c>c^*$ we have 
			\begin{equation*}
				\limsup_{t\to+\infty}\left[\sup_{x\geq ct}\Vert u(t,x)\Vert\right]=0. 
			\end{equation*}
	\end{enumerate}
\end{defn}
Note the we impose by convention that the propagation happens towards the right. It may happen that the rightward and leftward spreading speeds differ, as remarked in the previous subsection.

Next we prove the linear determinacy of {sublinear} systems  which have a monotone lower barrier. Such systems need not possess a comparison principle (system \eqref{eq:main-sys}, in particular, does not), therefore the classical theory cannot be applied directly. 
\begin{thm}[Linear determinacy]\label{thm:lin-det}
	Let $\mathcal L$ be a $L$-periodic elliptic operator, and $\vect{f}\in Lip(\mathbb R\times \mathbb R^d,\mathbb R^d)$ is $L$-periodic in $x$, admits a monotone lower barrier as defined in Definition \ref{def:super-monotone} and satisfies Assumption \ref{as:coop}. We denote $A(x):=Df(x, 0)$ and assume that the periodic principal eigenvalue $\lambda_1^{per}$ is negative. Then:
	\begin{enumerate}[label={\rm(\roman*)}]
		\item \label{item:prop-speed}
			System \eqref{eq:gen-syst} has a spreading speed $c^*$ as in Definition \ref{def:prop-speed}.
		\item 
			\label{item:lin-det}
			We have 
			\begin{align}
				c^*&= \inf_{\lambda>0}-\frac{k(\lambda)}{\lambda},\label{eq:c*}
			\end{align}
			where $ (k(\lambda), \vect{\varphi_\lambda}(x)\gg\vect{0}) $ is defined by \eqref{eq:defk(lambda)}. 
	\end{enumerate}
\end{thm}

We finally specify what we mean by {\it traveling wave}:
\begin{defn}[Traveling wave]
	A {\it traveling wave} $\vect{u}$ with speed $c>0$ and period $L$ for equation \eqref{eq:gen-syst} is a nonnegative entire solution to \eqref{eq:gen-syst} which satisfies the following condition:
	\begin{equation*}
		\forall x\in\mathbb R, \forall t\in\mathbb R, \vect{u}\left(t+\frac{L}{c}, x\right)=\vect{u}(t, x-L),
	\end{equation*}
	as well as the boundary conditions at $\pm\infty$ for all $t\in\mathbb R$:
	\begin{align*}
		\lim_{x\to+\infty}{u}(t, x)&=\vect{0}, \\
		\liminf_{x\to-\infty}\vect{u}(t, x)&\gg 0.
	\end{align*}
\end{defn}

With a little more regularity on $f$ we get the existence of traveling waves for $c\geq c^*$.
\begin{thm}[Existence of traveling waves]\label{thm:TW}
	In addition to the assumptions of Theorem \ref{thm:lin-det}, suppose that 
	there exist constants $M>0$ and $\beta>0$ such that 
	\begin{equation}\label{eq:conv-f}
		\Vert f^-(x, u)-A(x)u\Vert_\infty\leq M\Vert u\Vert_\infty^{1+\beta} \text{ for all }x\in\mathbb R\text{ and }\Vert u\Vert_{\infty}\leq \eta.
	\end{equation}
	Then, there exists a traveling wave for \eqref{eq:gen-syst} for all $c\geq c^*$.
\end{thm}

\begin{rem}[On monotone sub-solutions of \eqref{eq:main-sys}]\label{rem:speed-main-sys}
	Theorem \ref{thm:lin-det} allows us to compute the spreading speed and construct traveling waves for system \eqref{eq:main-sys}.  Indeed the modified system
	\begin{equation}\label{eq:modified-sys}
		\begin{system}
			\relax &u_t=\sigma_u(x) u_{xx}+\big(r_u(x)-\mu_u(x) - \kappa_u(x)u-\beta u\big)u+v\big(\mu_v(x)-\kappa_u(x)u\big), \\
			\relax &v_t=\sigma_v(x) v_{xx}+\big(r_v(x)-\mu_v(x) - \kappa_v(x)v -\beta v\big)v+u\big(\mu_u(x)-\kappa_v(x)\big),
		\end{system}
	\end{equation}
	is a monotone lower barrier for the original system (which corresponds to $\beta=0$). The original system itself is a monotone lower barrier in the region 
	\begin{equation*}
		\left\{0\leq u\leq \inf_{x\in\mathbb R}\left(\frac{\mu_v(x)}{\kappa_u(x)}\right)\right\}\times \left\{0\leq v\leq \inf_{x\in\mathbb R}\left(\frac{\mu_u(x)}{\kappa_v(x)}\right)\right\}.
	\end{equation*}
	However, in order to estimate solutions to \eqref{eq:modified-sys} when $t$ becomes large, we need to construct a monotone lower barrier which leaves the interval $[0, \eta\mathbf{1}]:=\{u\,|\,0\leq u\leq \eta\mathbf{1}\}$ invariant. This is precisely achieved for $\beta>0$ sufficiently large.
	
	In particular, Theorem \ref{thm:main-lindet} {below} is a direct consequence of Theorem \ref{thm:lin-det}.
\end{rem}

\subsection{On System \eqref{eq:main-sys}}

Our first result concerns the formula for the spreading speed of \eqref{eq:main-sys}, which provides a way  to compute the speed of traveling waves for \eqref{eq:main-sys}. The framework in which we prove this linear determinacy property is the following.
\begin{assumption}[Cooperative-competitive system]\label{as:coop-comp}
	We let $\sigma_u(x)>0$, $\sigma_v(x)>0$,   $\kappa_u(x)>0$, $\kappa_v(x)>0$, $\mu_v(x)> 0$, $\mu_u(x)> 0$, be $L$-periodic positive continuous functions  
	and  $r_u(x)$, $r_v(x)$ be $L$-periodic continuous functions of arbitrary sign. 
\end{assumption}
Our first result concerns the propagation of solutions to the parabolic equations \eqref{eq:main-sys}. 

\begin{thm}[Spreading speed for \eqref{eq:main-sys}]\label{thm:main-lindet}
	Let Assumption \ref{as:coop-comp} be satisfied. Assume that the principal eigenvalue of the linearised system \eqref{eq:main-eigenpair} is negative. Then, there exists a real number $c^*$ such that for any nonnegative initial condition $(u_0(x)\geq 0, v_0(x)\geq 0)$,
	\begin{enumerate}[label={\rm(\roman*)}]
		\item if $ \inf_{x\leq K}\min(u_0(x), v_0(x))>0$ for some $K\in\mathbb R$, then 
			\begin{align*}
				&\liminf_{t\to\infty} \left[\inf_{K\leq x\leq ct}\min(u(t, x), v(t,x))\right] > 0 , &\text{ for all } 0&<c<c^*,
			\end{align*}
		\item if there is $K>0$ such that $u_0(x)\equiv 0$ and $v_0(x)\equiv 0$ for all $x\geq K$, then
			\begin{align*}
				&\limsup_{t\to\infty} \left[\sup_{x\geq ct}\max(u(t, x), v(t,x))\right] = 0, & \text{ for all }c&>c^*, 
	\end{align*}
			
	\end{enumerate}
	where $(u(t,x), v(t,x))$ is the solution to the Cauchy problem \eqref{eq:main-sys} starting from the initial condition $(u_0(x), v_0(x))$.

	Moreover, we have the formula
	\begin{equation*}
		c^*=\inf_{\lambda>0}\frac{-k(\lambda)}{\lambda},
	\end{equation*}
	where $k(\lambda)$ is {defined in \eqref{eq:defk(lambda)}.}
\end{thm}
\begin{rem}
    {As shown in Remark \ref{rem:strong-coupling-speed}, with some particular choice of the parameters, the spatial behavior of System \eqref{eq:main-sys} approaches the one of a scalar KPP-type equation with an arbitrary first-order advection term. In particular, we expect that the spreading speed to the right is different than the spreading speed to the left. One may even reach a situation in which the speed to the right is positive, but the speed to the left is negative. In such a situation, compactly supported initial data would propagate to the right but also regress in the same direction, causing a pulse-like behavior with variable width that doesn't achieve a positive infimum in any bounded interval in the long run, even when the periodic principal eigenvalue is positive. Therefore we have no hope to have a hair-trigger effect in general for our kind of system when $\lambda_1^{per}<0$. The correct notion of principal eigenvalue for a hair-trigger effect is the {\it Dirichlet principal eigenvalue} $\lambda_1^{\infty} $, which will be introduced in Definition \ref{def:princ-eig} in the next Section. We refer to Theorem \ref{thm:hair-trigger} for a precise statement of the hair-trigger effect. }
\end{rem}
Next we introduce the notion of traveling wave solutions, which are entire solutions propagating at a fixed speed $c.$
\begin{defn}[Traveling wave solutions]
	Let $(u(t, x), v(t,x))$ be an entire solution to \eqref{eq:main-sys}, {\it i.e.} a solution that is defined for all $t\in\mathbb R$ and $x\in\mathbb R$. We say that $(u(t, x), v(t, x))$ is a \textit{traveling wave solution} traveling at speed $c$ if it satisfies 
	\begin{equation}\label{eq:TW-propagating}
		u\left(t+\frac{L}{c}, x\right) = u(t, x-l), \; v\left(t+\frac{L}{c}, x\right) = v(t, x-l), \text{ for all } (t,x)\in\mathbb R^2, 
	\end{equation}
	as well as the boundary conditions
	\begin{align*}
		\lim_{x\to+\infty}u(t, x)=0,\; \lim_{x\to+\infty}v(t, x)=0,\;  \text{ for all } t\in\mathbb R, \\ 
		\liminf_{x\to-\infty}u(t, x)>0,\; \liminf_{x\to-\infty}v(t, x)>0,\;  \text{ for all } t\in\mathbb R. 
	\end{align*}
\end{defn}
\begin{thm}[Existence of traveling waves]\label{thm:TW1}
	{Let Assumption \ref{as:coop-comp} hold. There exists a traveling wave for \eqref{eq:main-sys} with speed $c$ if, and only if, $c\geq c^*$. } 
\end{thm}
\begin{rem}
    Just as in the case with the spreading speed, the above theorem implies that the minimal speed of the traveling wave propagating to the right direction coincides with the spreading speed. The speed to the left may not be the same (see Remark \ref{rem:strong-coupling-speed}) and might even be negative.
\end{rem}
As we will see, Theorems \ref{thm:main-lindet} and \ref{thm:TW1} are direct consequences of results on more general cooperative-competitive systems, namely Theorems \ref{thm:lin-det} and \ref{thm:TW}.

Next we turn to the long-time behavior of the solutions to the Cauchy problem \eqref{eq:main-sys}, starting from a bounded nonnegative nontrivial initial condition. In the case where the coefficients are independent of $x$, we were able to show convergence to a unique stationary state.  More precisely, we consider the homogeneous problem
\begin{equation}\label{eq:syst-hom-rd}
	\begin{system}
		\relax &u_t-\sigma_u u_{xx}=(r_u-\kappa_u(u+v))u+\mu_vv-\mu_uu \\
		\relax &v_t-\sigma_v v_{xx}=(r_v-\kappa_v(u+v))v+\mu_uu-\mu_vv,
	\end{system}
\end{equation}
where $r_u\in\mathbb R$, $r_v\in\mathbb R$, $\kappa_u>0$, $\kappa_v>0$, $\mu_u>0$, $\mu_v>0$. The linearization of the right-hand side of \eqref{eq:syst-hom-rd} around $(u, v)=(0,0)$ is given by the matrix 
	\begin{equation}\label{eq:cond-small-mut}
		A:=\begin{pmatrix} r_u-\mu_u & \mu_v \\ \mu_u & r_u-\mu_v \end{pmatrix}.
	\end{equation}
	Since the off-diagonal entries of $A$ are positive, we easily see that $A$ has real eigenvalues. Let $\lambda_A$ denote the largest  eigenvalue of $A$
	\begin{equation} \label{eq:lambdaA}
		\lambda_A:=\max\{\lambda\in\mathbb R\, |\,\lambda \text{ is an eigenvalue of A}\}.
	\end{equation}
	Then by the Perron-Frobenius theory, the eigenvector corresponding to $\lambda_A$ is positive: $(\varphi_A^u, \varphi_A^v)^T$, $\varphi_A^u>0$, $\varphi_A^v>0$.
	\begin{assumption}\label{as:cond-instab-0}
		We assume that $(0, 0)$ is linearly unstable for the ODE problem \eqref{eq:syst-ode},
		\begin{equation}\label{eq:syst-ode}
			\begin{system}
				\relax &u_t=(r_u-\kappa_u(u+v))u+\mu_vv-\mu_uu \\
				\relax &v_t=(r_v-\kappa_v(u+v))v+\mu_uu-\mu_vv.
			\end{system}
		\end{equation}
		{\it i.e.} $\lambda_A>0$.
	\end{assumption}

It can be seen that the condition $\lambda_A>0$ is always satisfied when $r_u>0$ and $r_v>0$, and always fails when $r_u<0$ and $r_v<0$. The situation when $r_u$ and $r_v$ do not have the same sign is more intricate. In this case, there may exist a threshold depending on the values of $\mu_u$, $\mu_v$, such that $(0, 0)$ is stable for small values of $\mu_u$, $\mu_v$, and unstable for larger values. We discuss this threshold later in the article, in Lemma \ref{lem:stab-0-ode}.

Since system \eqref{eq:syst-hom-rd} has a {sublinear} nonlinearity, the sign of the eigenvalue $\lambda_A$ is a sharp condition for the existence of a non-trivial non-negative stationary solution. Indeed, the matrix $A$ is cooperative and therefore admits a unique eigenpair with a positive eigenvector $ (\lambda_A, \varphi_A)$; if $\lambda_A<0$, then for all $M>0$ $(\bar u(t), \bar v(t)):=Me^{\lambda_At}(\varphi_A^u, \varphi_A^v)$ is a super-solution to \eqref{eq:syst-hom-rd} which converges to $0$, and a direct application of the maximum principle shows that the solution $ (u,v)$ satisfies $(u(t, x), v(t,x))\leq (\bar u(t), \bar v(t))$ if $M>\max(\Vert u(0, \cdot)\Vert_{L^\infty}, \Vert v(0, \cdot)\Vert_{L^\infty})$. The non-existence of a stationary solution when $\lambda_A=0$ was treated in \cite[Theorem 1.4 (ii)]{Gir-18} and can also be seen as a direct consequence of \cite[Theorem 13.1 (c)]{Bus-Sir-04}.

Before turning to the PDE problem \eqref{eq:syst-hom-rd}, we first describe the long-time behavior of the associated ODE system\eqref{eq:syst-ode}.
\begin{prop}[Long-time behavior of the ODE system]\label{prop:longtime-ode}
	Let $(u(t), v(t))$ be the solution of \eqref{eq:syst-ode} starting from a non-negative non-trivial initial condition $(u_0, v_0)$.
	\begin{enumerate}[label={(\roman*)}]
		\item If $\lambda_A>0$, there is a unique positive equilibrium $(u^*, v^*)$ for \eqref{eq:syst-ode}, and $(u(t), v(t))$ converges to $(u^*, v^*)$ as $t\to\infty$.
		\item If $\lambda_A\leq 0$, then $(u(t), v(t))$ converges to $(0,0)$ as $t\to+\infty$.
	\end{enumerate}
\end{prop}

Next we turn to the local asymptotic stability of the PDE, {\it i.e.} the long-time  convergence of the solution to the parabolic equation \eqref{eq:syst-hom-rd} starting from an initial condition  in a vicinity of the constant stationary solution.
\begin{thm}[Local stability of the constant stationary solution]\label{thm:hom-local-stab} 
	Assume that $\lambda_A>0$  and let $(u^*, v^*)$ be the unique stationary solution for the ODE \eqref{eq:syst-ode}. Then $(u^*, v^*)$ is locally asymptotically stable as a stationary solution to \eqref{eq:syst-hom-rd} in the space $BUC(\mathbb R)^2$. More precisely, $(u^*,v^*)$ is stable and there exists $\delta>0$ such that for any $(u_0(x), v_0(x))\in BUC(\mathbb R)^2$ satisfying $\Vert u_0-u^*\Vert_{BUC(\mathbb R)}\leq \delta$ and $\Vert v_0-v^*\Vert_{BUC(\mathbb R)}\leq \delta$, then
	\begin{equation*}
		\lim_{t\to+\infty} \sup_{x\in\mathbb R}|u(t, x)- u^*|=0 \text{ and } \lim_{t\to+\infty} \sup_{x\in\mathbb R}|v(t, x)- v^*|=0, 
	\end{equation*}
	and the convergence is exponential in time.
\end{thm}
Note that the difficulty in this result is to overcome the absence of a comparison principle, even asymptotically (in the case where $\sigma_1\neq \sigma_2$ and Assumption \ref{as:LTPDE} does not hold). To this end we had to introduce an argument coming from semigroup theory \cite{Mag-Rua-18, Caz-Har-98}.

While Assumption \ref{as:cond-instab-0} is sufficient to describe the long-time behavior of the ODE problem, we require a little more for the study of the PDE problem \eqref{eq:syst-hom-rd}. We extend the  Lyapunov argument which was used for the ODE system in the non-cooperative case, though only  when $\sigma_u=\sigma_v$, and in the remaining cases the long-time behavior may be determined by using the comparison principle for either cooperative or two-component competitive systems. The cases under which global stability can be shown are summarized in Assumption \ref{as:LTPDE}.
\begin{assumption}\label{as:LTPDE}
	We assume that either $\max(r_u-\mu_u, r_v-\mu_v)\leq 0$, $\min(r_u-\mu_u-\mu_v, r_v-\mu_v-\mu_u)> 0$, or $\sigma_u=\sigma_v$.
\end{assumption}

Under this assumption, we can prove that the solutions to the Cauchy problem associated with \eqref{eq:syst-hom-rd} converge in long time to the unique nonnegative nontrivial stationary solution.

\begin{thm}[Long-time behavior of the homogeneous problem]\label{thm:ltb}
	Let Assumptions \ref{as:cond-instab-0} and \ref{as:LTPDE} be satisfied. Let $(u_0(x)\geq 0, v_0(x)\geq 0)$ be bounded continuous nontrivial functions, and $c^*$ be the spreading speed associated with \eqref{eq:syst-hom-rd}. Then, the solution $(u(t,x), v(t,x))$ to the Cauchy problem \eqref{eq:syst-hom-rd} converges as $t\to\infty$ to the unique stationary solution $(u^*, v^*)$ to \eqref{eq:syst-hom-rd}, uniformly in the sense that for each $0<c<c^*$ we have:
	\begin{equation*}
		\lim_{t\to+\infty}\sup_{|x|\leq ct}\max(|u(t,x)-u^*|, |v(t,x)-v^*|)=0.
	\end{equation*}
\end{thm}

Finally, we were able to extend this result to the case of rapidly oscillating coefficients by using arguments from the theory of dynamical systems and the homogenization of solutions to parabolic equations with rapidly oscillating coefficients. To this end it is more convenient to write the heterogeneous system \eqref{eq:main-sys} in divergence form \eqref{eq:rapidosc}.

In order to state our results for the homogenization limit of parabolic systems, we restrict ourselves to the case $L=1$, without loss of generality. For each $1$-periodic function $\sigma_u(x)$, $\sigma_v(x)$, $r_u(x)$, $r_v(x)$, $\kappa_u(x)$, $\kappa_v(x)$, $\mu_u(x)$, $\mu_v(x)$, we denote:
\begin{align}\label{eq:mean-coeffs}
	\overline{r_u}&:=\int_0^1r_u(x)\dd x , & \overline{\kappa_u}&:=\int_0^1\kappa_u(x)\dd x , & \overline{\mu_u}&:=\int_0^1\mu_u(x)\dd x, \\
	\overline{r_v}&:=\int_0^1r_v(x)\dd x, &\overline{\kappa_v}&:=\int_0^1\kappa_v(x)\dd x, & \overline{\mu_v}&:=\int_0^1\mu_v(x)\dd x, \notag 
\end{align}
and finally:
\begin{align*}
	\overline{\sigma_u}^H&:=\left(\int_0^1\frac{1}{\sigma_u(x)}\dd x\right)^{-1}, & \overline{\sigma_v}^H&:=\left(\int_0^1\frac{1}{\sigma_v(x)}\dd x\right)^{-1}.
\end{align*}

\begin{thm}[Homogenisation]\label{thm:rapidosc}
	Let $\sigma_u(x)$, $\sigma_v(x)$, $r_u(x)$, $r_v(x)$, $\kappa_u(x)$, $\kappa_v(x)$, $\mu_u(x)$ and $\mu_v(x)$ be $1$-periodic functions such that $\overline{\sigma_u}^H$, $\overline{\sigma_v}^H$, $\overline{r_u}$, $\overline{r_v} $, $\overline{\kappa_u}$, $\overline{\kappa_v}$, $\overline{\mu_u}$ and $\overline{\mu_v}$ satisfy Assumption \ref{as:cond-instab-0} and Assumption \ref{as:LTPDE}. Consider 
	\begin{equation}\label{eq:rapidosc}
		\begin{system}
			\relax &u_t=(\sigma_u^\varepsilon(x) u_{x})_x+(r_u^\varepsilon(x)-\kappa_u^\varepsilon(x)(u+v))u+\mu_v^\varepsilon(x)v-\mu_u^\varepsilon(x)u \\
			\relax &v_t=(\sigma_v^\varepsilon(x) v_{x})_x+(r_v^\varepsilon(x)-\kappa_v^\varepsilon(x)(u+v))v+\mu_u^\varepsilon(x)u-\mu_v^\varepsilon(x)v.
		\end{system}
	\end{equation}
	Then, there is $\bar\varepsilon>0$ such that for each $0<\varepsilon<\bar \varepsilon$, 
	\begin{enumerate}[label={\rm(\roman*)}]
		\item \label{item:rapid-osc-unique}
			there exists a unique positive nontrivial stationary solution $(u^*_\varepsilon(x), v^*_\varepsilon(x))$ to \eqref{eq:rapidosc}, 
		\item \label{item:entire-sol}
			the $\omega$-limit set of any  $(u^\varepsilon(t, x), v^\varepsilon(t, x))$ solution to \eqref{eq:rapidosc} starting from a nonnegative nontrivial bounded initial condition is $\{(u^*_\varepsilon(x), v^*_\varepsilon(x))\}$.
		\item \label{item:rapid-osc-cv}
			any solution to the Cauchy problem \eqref{eq:rapidosc} starting from a nonnegative bounded initial condition converges as $t\to+\infty$ to $(u^*_\varepsilon(x), v^*_\varepsilon(x))$, uniformly in the sense that for any $0<c<c^*_\varepsilon$ we have
			\begin{equation*}
				\lim_{t\to+\infty}\sup_{|x|\leq c t}\max(|u(t,x)-u^*_\varepsilon(x)|, |v(t,x)-v^*_{\varepsilon}(x)|)=0,
			\end{equation*}
			where $c^*_\varepsilon$ is the minimal speed defined in \eqref{eq:formula-speed}.
	\end{enumerate}
\end{thm}

\section{Proofs of the results on general cooperative-competitive systems}
\label{sec:KPP-type}

In Section \ref{sec:super-monotone}, we show that solutions to equations can be estimated from below by a monotone lower barrier.
In Section \ref{sec:eigenpb} we prove some properties on principal eigenproblems for periodic system, including the equivalence between the various notions of principal eigenvalue on the real line for operators satisfying Assumption \ref{as:isotropic}.
In Section \ref{sec:coop-sys} we prove the linear determinacy for {sublinear} functions satisfying Assumption \ref{as:coop} (Theorem \ref{thm:lin-det}) by adapting an argument of Weinberger \cite{Wei-02}. In Section \ref{sec:proofs-gen} we prove   and the existence of traveling waves, Theorem \ref{thm:TW}.    Finally in Section \ref{sec:large-diff} we prove Theorem \ref{thm:large-diff}.

Before resuming the proofs, let us mention two important conventions. The constant $d>0$ stands for the dimension of the system being investigated. Also, whenever $u$ is a vector, we denote $(u)_i $ the $i$-th component of $u$, or simply $u_i$ if the context is clear.

\subsection{A comparison principle for systems with a monotone lower barrier.}
\label{sec:super-monotone}

In this Section we prove that a function that admits a monotone lower barrier generates a semiflow that remains above the one generated by the lower barrier. More precisely, we show that as long as the solution $u(t,x)$ of the equation corresponding to the lower monotone barrier stays in the quasi-monotone area, then the solution $v(t,x)$ is componentwise greater than $u(t, x)$ (even if it leaves the quasi-monotone domain).
\begin{thm}[Comparison principle]\label{thm:super-monotone} 
    Let $f$ be a given {sublinear} function  and $\mathcal L$ be a $d$-dimensional diagonal uniformly elliptic operator. We assume that {$Df(x, 0)$ is cooperative and fully coupled and that}  $f$  admits a monotone lower barrier $f^-$ in the sense of Definition \ref{def:super-monotone}.

	Let $T\in (0, +\infty] $ and $u(t, x)$ and $v(t, x)$ solve
	\begin{equation*}
		\left\{\begin{aligned}\relax
			&u_t(t, x)-\mathcal Lu(t,x) = f(x, u(t,x))\\
			&u(0, x)=u_0(x),  
		\end{aligned}\right.
		\text{ and }
		\left\{\begin{aligned}\relax
			&v_t(t, x)-\mathcal Lv(t,x) = f^-(x, v(t,x))\\
			&v(0, x)=v_0(x),  
		\end{aligned}\right.
	\end{equation*}
	for $t\in [0, T]$ and $x\in\mathbb R$, $u_0, v_0\in BUC(\mathbb R)$.

	Suppose that  $\Vert v(t,x)\Vert_\infty< \eta$ for all $t\in[0, T] $.  Then $u(t,x) $ satisfies 
	\begin{equation*}
		v(t,x)\leq u(t,x)\text{ for any } t\in[0,T] \text{ and }x\in\mathbb R.
	\end{equation*}
\end{thm}
\begin{proof}
	We show the result under the assumption that $u_0(x)\geq v_0(x)+\delta\mathbf{1}$  and $\Vert v(t, x)\Vert\leq \eta-\delta$ for some $\delta\in (0, \eta)$. The general result is obtained by taking the limit $\delta\to 0$. Since $t\mapsto u(t,x)$ is continuous at $t=0$, there exists $t_0>0$ such that $v(t,x)\leq u(t,x)$ for all $t\in [0, t_0]$ and $x\in\mathbb R$.
	
	We define 
	\begin{equation*}
		t^*:=\sup\{t>0\,|\, v(t,x)\leq u(t,x) \text{ for all }x\in\mathbb R\}.
	\end{equation*}
	Then by definition $t^*\geq t_0$. Assume by contradiction that $t^*<T$. Then, because of the definition of $t^*$,  there exists a sequence $(t_n, x_n)\in [0, T]\times U$ and $i\in\{1, \ldots, d\}$ such that $t_n\to t^*$, $t_n\geq t^*$ and 
	\begin{equation*}
		v_i(t_n, x_n) \to u_i(t_n, x_n) \text{ and } v_j(t_n, x_n)\leq u_j(t_n, x_n) \text{ for all } j\neq i.
	\end{equation*}
	If $ x_n$ is bounded, then we may extract a subsequence such that $x_n\to x$. By the continuity of $v_i$ and $u_i$ we have then $v_i(t^*,x ) = u_i(t^*,x)$. Since moreover $v_i(t,x)\leq \Vert v(t,x)\Vert_\infty \leq \eta$ we have $f^-\big(x, v(t,x)\big)\leq f^-\big(x, u(t,x)\big)\leq f(x, u(t,x))$. Testing the $i$-th equation we get
	\begin{equation*}
		(v_i)_t(t,x)-\mathcal L v_i(t,x)=(f^-)_i(x, v(t,x))\leq f_i(x, u(t,x))=(u_i)_t-\mathcal Lu_i(t,x),
	\end{equation*}
	and there is a contradiction by the strong maximum principle. If $(t_n, x_n)$ is unbounded we get a similar contradiction by extracting a converging subsequence from the sequence of functions $u(t+t_n, x+x_n)$ and $v(t+t_n, x+x_n)$. This proves that $t^*=T$, therefore the result holds. 
\end{proof}

\begin{prop}\label{prop:sub-sub-barrier}
    Let $f=(f_1, \ldots, f_d)^T$ be a given {sublinear} function  and $\mathcal L$ be a $d$-dimensional diagonal uniformly elliptic operator. We assume {that $Df(x, 0)$ is cooperative and fully coupled,} that $\lambda_1^{per}<0$ and that $f$  admits a monotone lower barrier $f^-$ in the sense of Definition \ref{def:super-monotone}. Let $\eta>0$ be as in Definition \ref{def:super-monotone}.

	There exists a monotone lower barrier $f^{*-}(x, u)$ for $f$ with the properties that 
	\begin{enumerate}
		\item  we have
			\begin{equation*}
				f^{*-}(x, u)\leq f^-(x, u) \text{ for all } x\in\mathbb R \text{ and } u\geq 0.
			\end{equation*}
		\item there exists a $L$-periodic equilibrium $p(x)=(p_1(x), \ldots, p_d(x))$ such that $0\leq p(x)\leq \eta\mathbf{1}$, 
			\begin{equation*}
				-\mathcal L p(x) = f^{*-}\big(x, p(x)\big), 
			\end{equation*}
			and $p$ attracts every nontrivial periodic initial condition $u_0(x) $ satisfying $0\leq u_0(x)\leq p(x)$ for all $x\in\mathbb R$.
	\end{enumerate}
\end{prop}
\begin{proof}
	Let $\beta>0$ be given and define 
	\begin{equation*}
		f^-_\beta(x, u):=f^-(x, u)-\beta u^2= (f^-_1(x, u)-\beta u_1^2, \ldots, f^-_d(x, u)-\beta u_d^2)^T. 
	\end{equation*}
	It is clear that $f^-_\beta(x, u)\leq f^-(x, u)$ for all $x\in\mathbb R$ and $u\in\mathbb R^d$, and that $f^-(x, u)$ and $f^-_\beta(x, u)$ have the same Jacobian matrix near $u=0$, $Df^-(x,0)=Df^-_\beta(x, 0) = :A(x)$. 

	We investigate the equation
	\begin{equation}\label{eq:sub-beta}
		u_t-\mathcal L u = f^-_\beta(x, u).
	\end{equation}
	If 
	\begin{equation*}
		\beta\geq \beta^*:= \dfrac{\sup_{x\in\mathbb R} \max_{i\in\{1, \ldots, d\}}\sum_{j=1}^da_{ij}(x)}{\eta}, 
	\end{equation*}
	where $\eta>0$ is the constant from Definition \ref{def:super-monotone}, then the constant vector $\eta\mathbf{1}$ satisfies
	\begin{equation*}
		f^-_\beta(x, \eta\mathbf{1})\leq A(x)\eta\mathbf{1}-\beta \eta^2\mathbf{1}\leq \eta\left(\sup_{x\in\mathbb R}\max_{i\in\{1, \ldots, d\}}\sum_{i=1}^d a_{ij}(x)-\beta\eta\right), 
	\end{equation*}
	therefore $\eta\mathbf{1}$ is a super-solution to \eqref{eq:sub-beta}. 

	Next we look for a sub-solution to \eqref{eq:sub-beta}. Recall that we denote $(\lambda_1^{per}, \varphi^{per}(x)>0)$ the periodic principal eigenpair as in Definition \ref{def:princ-eig}, with $\Vert \varphi^{per}\Vert_{L^\infty(\mathbb R)^d}=1$, and recall that $\lambda_1^{per}<0$.
Define 
	\begin{equation*}
		\kappa:=\inf_{x\in \mathbb R}\min_{1\leq i\leq d}\frac{\varphi^{per}_i(x)}{\Vert\varphi^{per}(x)\Vert_{\infty}}, 
	\end{equation*}
	which is finite and positive by the elliptic strong maximum principle. Because of the differentiability of $u\mapsto f^-_\beta(x, u)$, there exists $\varepsilon_0>0$ such that for each $u\geq 0$ with $\Vert u\Vert\leq \varepsilon_0$, we have
	\begin{equation*}
		\Vert f^-_\beta(x, u)-A(x)u\Vert_{\infty} \leq -\lambda_1^{per}\kappa\Vert u\Vert_{\infty}. 
	\end{equation*}
	Reducing $\varepsilon_0$ if necessary, we may assume that $\varepsilon_0<\eta$. Then for $0<\varepsilon\leq \varepsilon_0$, $\varepsilon\varphi^{per}(x)$ is a sub-solution to \eqref{eq:sub-beta}. Indeed, 
	\begin{align*}
		-\mathcal L\varepsilon\varphi^{per}(x) &= A(x)\varepsilon \varphi^{per}(x)+\lambda_1^{per}\varepsilon\varphi^{per}(x)\leq A(x)\varphi^{per}(x)+\lambda_1^{per}\kappa \Vert\varepsilon \varphi^{per}\Vert_{\infty} \\
		&\leq A(x)\varepsilon\varphi^{per} + f^-_\beta\big(x, \varepsilon\varphi^{per}(x)\big)-A(x)\varepsilon\varphi^{per}(x)=f^-_\beta\big(x, \varepsilon\varphi^{per}(x)\big).
	\end{align*}

	Let $\varepsilon>0$ be fixed and $\underline{u}^{\varepsilon}(t,x)$ be the solution to the initial-value problem \eqref{eq:sub-beta} with $\underline{u}^\varepsilon(0, x)=\varepsilon\varphi^{per}(x)$. It follows from the parabolic comparison principle and the strong maximum principle that $\underline{u}^\varepsilon(t,x)\gg \varepsilon\varphi^{per}(x)$ for all $t>0$ and $x\in\mathbb R$.
	Next, fix $\tau>0$. Then  $\underline{u}^{ \varepsilon}(\tau, x)> \varepsilon\varphi^R(x)=\underline{u}^{\varepsilon}_0(x)$ and therefore 
	\begin{equation*}
		\underline{u}^{ \varepsilon}(t+\tau ,x)>\underline{u}^{\varepsilon}(t,x), 
	\end{equation*}
	in other words, $\underline{u}^{\varepsilon}(t,x) $ is strictly increasing in time. Thus the limit
	\begin{equation*}
		V^{\varepsilon}:=\lim_{t\to+\infty} \underline{u}^{\varepsilon}(t,x)
	\end{equation*}
	exists and is an equilibrium of \eqref{eq:sub-beta}. It is not difficult to show, by using Serrin's sweeping method, that 
	\begin{equation*}
		V^{\varepsilon}(x)\geq \varepsilon_0\varphi^{per}(x).
	\end{equation*}
	Indeed, define
	\begin{equation*}
		\varepsilon_1:=\sup\{\varepsilon'\geq 0\, |\, \varepsilon'\varphi^{per}(x)\leq V^\varepsilon(x)\}.
	\end{equation*}
	Then clearly $\varepsilon_1\geq \varepsilon$ (by the parabolic comparison principle). If $\varepsilon_1<\varepsilon_0$, then there exists a contact point $x_0$ such that $\varepsilon_1\varphi^{per}(x_0)\leq V^{\varepsilon}(x_0)$ and $\varepsilon_1\varphi^{per}_i(x_0)=V^{\varepsilon}_i(x_0)$ for some $i\in\{1, \ldots, d\}$. We find a contradiction by applying the elliptic strong maximum principle in the $i$-th equation of the system. Since $V^\varepsilon(x)$ is an equilibrium and $V^\varepsilon(x)\geq \varepsilon_0\varphi^{per}(x)$, we conclude that 
	\begin{equation*}
		V^\varepsilon(x)\geq V^{\varepsilon_0}(x).
	\end{equation*}

	Define $p^\beta(x):=V^{\varepsilon_0}(x)$ and let $u_0(x)\leq p^\beta(x)$  be a nontrivial $L$-periodic initial data. Let $u(t,x)$ be the solution of \eqref{eq:sub-beta} satisfying $u(0, x)=u_0(x)$. Fix some $t_0>0$. Then $u(t_0,x)>0$ for all $x\in\mathbb R$. Since $u(t_0,x)$ is periodic in $x$, there exists $\varepsilon>0$ sufficiently small, so that $\varepsilon\varphi^{per}(x)\leq u(t_0, x)$. 
	It follows from the comparison principle that 
	\begin{equation*}
		p^\beta(x)=V^{\varepsilon_0}(x)\leq V^\varepsilon(t,x)\leq \lim_{t\to+\infty} u(t,x) \leq p^\beta(x). 
	\end{equation*}
	Therefore we have found $f^{*-}(x, u):=f^-_\beta(x, u)$ satisfying the requirements of 
	Proposition \ref{prop:sub-sub-barrier}. 
\end{proof}

\subsection{The principal eigenvalue of cooperative systems with periodic coefficients} \label{sec:eigenpb}

In this section we focus on the principal eigenvalue problem for general cooperative elliptic systems with periodic coefficients. For $1\leq i\leq d$ and $\alpha>0$, let ${\sigma}(x)\in C^{1+\alpha}_{per}(\mathbb R, \mathbb R^d)$ be positive, and ${q}(x)\in C^\alpha_{per}(\mathbb R, \mathbb R^d)$ be given. We recall that: 
\begin{equation*}
	\mathcal L_iu:=(\sigma_i(x)u_{x})_x + q_i(x)u_x, 1\leq i\leq d\,; \qquad \mathcal L{u}=\big(\mathcal L_iu_i)_{1\leq i\leq d} ,
\end{equation*}
if $\mathcal L$ is written in divergence form, and 
\begin{equation*}
	\mathcal L_iu:=\sigma_i(x)u_{xx} + q_i(x)u_x, 1\leq i\leq d\,; \qquad \mathcal L{ u}=\big(\mathcal L_iu_i)_{1\leq i\leq d} 
\end{equation*}
if $\mathcal L$ is written in non-divergence form. The particular choice of writing the operator in divergence form or non-divergence form makes little difference for the principal eigenproblem, except when a symmetry property is involved; non-divergence form systems are better suited for systems which are composed of even functions, and systems in divergence form are most convenient when a symmetry for the canonical $H^1$ scalar product is needed.

We start with some elementary properties of the Dirichlet principal eigenvalue.

\begin{proof}[Proof of Proposition \ref{prop:gen-princ-eig}]
	We prove each statement separately.\medskip

	\begin{stepping}
		\step Existence of a principal eigenfunction.

		The existence and uniqueness of a principal eigenfunction associated with $\lambda_1^R$ in the case $R<+\infty$ is an immediate consequence of the Krein-Rutman Theorem. \medskip

		\step Proof of \eqref{eq:gen-princ-eig}.

		Assume by contradiction that there exists $\lambda\in\mathbb R$ and $\vect{\phi}\in C^2((-R, R), \mathbb R^d)\cap C^1([-R, R], \mathbb R^d)$, $ \vect{\phi}>\vect{0}$, such that 
		\begin{equation*}
			-\mathcal L\vect{\phi} - A(x)\vect{\phi} -\lambda\vect{\phi} \geq \vect{0},
		\end{equation*}
		and $\lambda>\lambda_1^R$.
		On the one hand, it follows from Hopf's Lemma that, for each $i\in\{1, \ldots, d\}$,  we have $\frac{\dd \varphi^R_i}{\dd x}(-R)>0 $ and $\frac{\dd \varphi^R_i}{\dd x}(R)<0 $. On the other hand, for each $i\in\{1,\ldots, d\}$, either $\phi_i(\pm R)>0$ or $\pm\phi_i(\pm R)<0$ by Hopf's Lemma.  Therefore, the set  $\{\zeta>0\,|\,\zeta{\varphi}^R\leq {\phi}\}$ is nonempty and admits a supremum $\eta>0$. We remark that, by definition of $\eta$, the inequality in  $\zeta\vect{\varphi}^R\leq \vect{\phi}$ is an equality for a $x_0\in [-R, R]$. Moreover, we have:
		\begin{equation*}
			-\mathcal L({\phi}-\eta{{\varphi}}^R)-A(x)({\phi}-\eta{\varphi}^R)-\lambda({\phi}-\eta{\varphi}^R)\geq (\lambda-\lambda_1)\eta{\varphi}^R\geq{0},
		\end{equation*}
		thus for $1\leq i\leq d$, either $\phi_i(\pm R)-\eta\varphi^R_i(\pm R)>0$ or, by Hopf's Lemma,  $\pm\frac{\dd (\phi_i-\eta\varphi^R_i)}{\dd x}(\pm R)<0 $. In particular, we have $x_0\in(-R, R)$ and there exists $j\in \{1, \ldots, d\}$ such that $\phi_j(x_0)=\eta\varphi^R_j(x_0)$.
		At this point, we have
		\begin{align*}
			0&\geq -\mathcal L(\phi_j-\eta\varphi^R_j)(x_0) - (A(x_0)(\phi-\eta\varphi^R)(x_0))_j\\
			&=\lambda\phi_j(x_0)-\lambda_1^R\varphi^R_j(x_0)=\phi_j(x_0)(\lambda-\lambda_1^R), 
		\end{align*}
		which shows $\lambda\leq\lambda_1^R$. This is a contradiction. We conclude that  $\lambda\leq \lambda_1^R$. Since $\lambda$ and $\phi$ are arbitrary, we have shown  
		\begin{equation*}
				\lambda_1^R\geq\sup\{\lambda\in\mathbb R\,|\,\exists {\phi}\in C^2((-R, R), \mathbb R^d)\cap C^1([-R, R], \mathbb R^d), {\phi}>0, -\mathcal L{\phi} - A(x){\phi} -\lambda{\phi} \geq { 0}\}. 
		\end{equation*}
		Finally, the equality in \eqref{eq:def-princ-Dir} shows the reverse inequality. Statement \ref{item:princ-eig} is proved.\medskip

		\step $R\mapsto \lambda_R$ is decreasing.

		Let $R<R'$. Then, the function $\vect{\varphi}^{R'}$ is a valid test function in the characterization \eqref{eq:gen-princ-eig} of $\lambda_1^R$. Therefore $\lambda_1^{R'}\leq \lambda_1^R$. Since the equalities $\lambda_1(R')=\lambda_1^{R'}$ and $\lambda_1^R=\lambda_1^R$ hold, we have $\lambda_1^R\geq \lambda_1^{R'}$. A direct application of the strong maximum principle shows that equality cannot be achieved.  Statement \ref{item:princ-eig-dec-R} is proved.\medskip

		\step Existence of a principal eigenfunction for $\lambda_1^\infty$ and limit of $\lambda_1^R$.

		Let $R_n\to+\infty$, then $\lambda_1^{R_n} $ is a nonincreasing sequence and thus converges to  $\lambda_1^\infty$. Let $\vect{\varphi}^n$ be the associated principal eigenfunction satisfying $\varphi^n_1(0)=1$. Then, by the classical Schauder estimates and the  Harnack inequality for fully coupled elliptic systems \cite[Theorem 8.2]{Bus-Sir-04}, the sequence $({\varphi}^n)_{n>0}$ converges locally uniformly to a limit ${\varphi}^\infty$ which satisfies $-\mathcal L{\varphi}^\infty - A(x){\varphi}^\infty=\lambda_1^\infty{\varphi}^\infty$ (up to the extraction of a subsequence).

		Let us show that $\lambda_1^\infty=\lambda_1(+\infty)$. Let $ (\lambda, {\phi})$ be such that $-\mathcal L{\phi} - A(x){\phi} -\lambda{\phi} \geq {0}$. Then by \eqref{eq:gen-princ-eig}, for any $n\in\mathbb N$ we have $\lambda \leq \lambda_1^{R_n}$. Taking the limit in the inequality, we find $\lambda_1(+\infty)\leq \lambda_1^\infty$. Since $(\lambda_1^\infty, {\varphi}^\infty)$ satisfies the equality $-\mathcal L{\varphi}^\infty - A(x){\varphi}^\infty=\lambda_1^\infty{\varphi}^\infty$, we have $\lambda_1^\infty \leq \lambda_1(+\infty)$.   
		Statement \ref{item:limitRbig} is proved. \medskip

		\step We prove the minimax formula \eqref{eq:eigprinc-minimax}.

		Using  $\vect{\varphi}^R$ as a test function in \eqref{eq:eigprinc-minimax}, we find that 
		\begin{equation*}
			\lambda_1^R\leq \lambda^*:=\underset{{\phi}>0}{\sup}\,\underset{x\in (-R, R)}{\inf} \underset{1\leq i\leq d}{\min}\frac{(-\mathcal L{\phi} -A(x){\phi})_i}{\phi_i}.
		\end{equation*}
		Let us show the converse inequality. Let $\epsilon>0$ be given, then by definition of $\lambda^*$ there exists  $\vect{\phi}>\vect{0}$ such that 
		\begin{equation*}
			\forall x\in (-R, R), \forall i\in\{1, \ldots, d\}, \frac{(-\mathcal L{\phi}(x)-A(x){\phi}(x))_i}{\phi_i(x)}\geq \lambda^*-\epsilon, 
		\end{equation*}
		and thus for all $x\in\mathbb R$ we have $-\mathcal L{\phi}(x)-A(x){\phi}(x)-(\lambda^*-\epsilon){\phi}(x)\geq {0}$. By \eqref{eq:gen-princ-eig}, we have $\lambda^*-\epsilon\leq\lambda_1^R$. Since $\epsilon>0$ is arbitrary, $\lambda^*\leq \lambda_1^R$. Finally, since $\lambda^*=\lambda_1^R$, the supremum is attained for the principal eigenfunction.  Statement \ref{item:minimax} is proved.
	\end{stepping}
\end{proof}

\begin{proof}[Proof of Lemma \ref{prop:k(lambda)}]
	Statement \ref{item:eigenpairk(lambda)} is a direct consequence of the Krein-Rutman Theorem, and statement \ref{item:minimaxk(lambda)} is a consequence  of Lemma \ref{prop:gen-princ-eig} statement \ref{item:minimax} (by modifying the elliptic operator $\mathcal L$). Therefore we concentrate on the remaining statements. \medskip

	{\bf Proof of Statement \ref{item:k(lambda)concave}.} We first note that the analyticity of $k(\lambda)$ is classical. In the terminology of Kato \cite{Kat-95}, the family $L_\lambda$ is a holomorphic family of type (A) \cite[Paragraph 2.1 on page 375]{Kat-95} and the principal eigenvalue is isolated in the spectrum by the Krein-Rutman Theorem; therefore the spectral projection and the principal eigenvalue are analytic (see  \cite[Remark 2.9 on page 379]{Kat-95}. The analyticity of $k(\lambda)$ and a well-chosen parameterization of the principal eigenvector $\phi^\lambda$ with respect to $\lambda$ follow. 
	
	{Let us prove \eqref{eq:klesssquare}. Let $(k(\lambda), \phi^\lambda)$ be a solution to \eqref{eq:defk(lambda)}. Because $\phi^\lambda$ is periodic, there exists a point $x\in\mathbb R$ and an index $i\in\{1, \ldots, d\}$ such that $\phi^\lambda_i(x)$ minimizes $(y, j)\mapsto \phi^\lambda_j(y)$ with $y\in\mathbb R$ and $j\in\{1, \ldots, d\}$.  Therefore
	\begin{align*}
		k(\lambda) \phi_i^\lambda(x)&= -\big(L_\lambda \phi^\lambda\big)_i - \big(A(x)\phi^\lambda(x)\big)_i \leq  +\lambda q_i(x)\phi^\lambda_i(x) - \lambda^2 \sigma_i(x)\phi^\lambda_i(x)- \sum_{j=1}^d a_{ij}(x)\phi^\lambda_j(x)\\ 
		&\leq \lambda q^\infty\phi^\lambda_i(x) - \lambda^2 \sigma_0\phi^\lambda_i(x) - a_0 \phi^\lambda_i(x), 
	\end{align*}
	where $q^\infty := \sup_{y\in\mathbb R, j\in \{1, \ldots, d\}}|q_j(y)|$, $\sigma_0 := \inf_{y\in \mathbb R, j\in\{1, \ldots, d\}}\sigma_j(x)$ and $a_0 := \inf_{y\in\mathbb R, j\in \{1, \ldots, d\}} a_{j}(y)$.
		}
	{This proves \eqref{eq:klesssquare}.}
	
	Next we follow \cite[Proposition 2.10]{Nad-09} to prove the concavity of $\lambda\mapsto k(\lambda)$. By the assumption that $\vect{\sigma}\in C^{1, \alpha}(\mathbb R,  M_d(\mathbb R))$, we need only consider  the non-divergence case. 

	We first remark that \eqref{eq:minimax-k(lambda)} can be rewritten as:
	\begin{equation*}
		k(\lambda)=\underset{e^{\lambda x}{\psi}(x) \text{ is } L-\text{periodic}}{\underset{{\psi}>{0}}{\max}} \inf_{x\in\mathbb R}\min_{1\leq i\leq d} \frac{(-\mathcal L {\psi}(x) - A(x){\psi}(x))_i}{\psi_i(x)}. 
	\end{equation*}
	Let $\lambda_1<\lambda_2$ and $r\in (0,1)$. Let ${\psi}^1$ and ${\psi}^2 $ be such that $e^{\lambda_1 x}{\psi}^1(x) $ and $e^{\lambda_2 x}{\psi}^2(x)$ are $L$-periodic in $x$. Define further ${z}^1=(\ln(\psi^1_i))_{1\leq i\leq d}$, ${z}^2=(\ln(\psi^2_i))_{1\leq i\leq d}$, ${z}(x)=r{z}^1+(1-r){z}^2(x)$, and finally $\lambda=r\lambda_1 +(1-r)\lambda_2$. Elementary computations then show that ${\psi}(x):=e^{{z}(x)}:=(e^{z_i(x)})_{1\leq i\leq d}$ is a valid test function for $k(\lambda)$ since $e^{\lambda x}{\psi}(x)$ is $L$-periodic. Thus:
	\begin{equation*}
		k(\lambda)\geq \inf_{x\in\mathbb R}\min_{1\leq i\leq d}\frac{(-\mathcal L{\psi} - A(x){\psi})_i}{\psi_i(x)}.
	\end{equation*}
	We compute:
	\begin{align}
		\frac{-\mathcal L_i\psi_i(x)}{\psi_i(x)}&=\frac{1}{\psi_i(x)}\left[\sigma_i(x)\left(-r\frac{\psi^1_{i, xx}(x)}{\psi^1_i(x)} -(1-r)\frac{\psi^2_{i, xx}(x)}{\psi^2_i(x)}+r(1-r)\left(\frac{\psi^1_x(x)}{\psi^1(x)}-\frac{\psi^2_x(x)}{\psi^2(x)}\right)^2\right)\right. \nonumber \\
		&\quad\left. -q_i(x)\left(r\frac{\psi^1_{i,x}(x)}{\psi^1_i(x)}+(1-r)\frac{\psi^2_{i,x}(x)}{\psi^2_i(x)}\right)\right]e^{z_i(x)} \nonumber \\
		&\geq \left[r\frac{-\mathcal L_i\psi^1_i(x)}{\psi^1_i(x)} + (1-r) \frac{-\mathcal L_i\psi^2_i(x)}{\psi^2_i(x)} +\sigma_i(x) r(1-r)\left(\frac{\psi^1_x(x)}{\psi^1(x)}-\frac{\psi^2_x(x)}{\psi^2(x)}\right)^2\right]\frac{e^{z_i(x)}}{\psi_i(x)}. \label{eq:ineqconcave-diff}
	\end{align}
	Then, we remark that 
	\begin{align}
		\underset{j=1}{\overset{d}{\sum}}a_{ij}(x)\frac{\psi_j(x)}{\psi_i(x)}&=\underset{j=1}{\overset{d}{\sum}}a_{ij}(x)\frac{\exp\left(r\ln(\psi^1_j(x))+(1-r)\ln(\psi^2_j(x))\right)}{\exp\left(r\ln(\psi^1_i(x))+(1-r)\ln(\psi^2_i(x))\right)}\nonumber\\
		&=\underset{j=1}{\overset{d}{\sum}}a_{ij}(x)\exp\left(r\ln\left(\frac{\psi^1_j(x)}{\psi^1_i(x)}\right) + (1-r)\ln\left(\frac{\psi^2_j(x)}{\psi^2_i(x)}\right)\right)\nonumber \\
		&\leq\underset{j=1}{\overset{d}{\sum}}a_{ij}(x)\left[r\left(\frac{\psi^1_j(x)}{\psi^1_i(x)}\right) + (1-r)\left(\frac{\psi^2_j(x)}{\psi^2_i(x)}\right)\right]\label{eq:ineqconcave-exp}\\
		&=r\frac{\underset{j=1}{\overset{d}{\sum}}a_{ij}(x)\psi^1_j(x)}{\psi^1_i(x)} + (1-r)\frac{\underset{j=1}{\overset{d}{\sum}}a_{ij}(x)\psi^2_j(x)}{\psi^2_i(x)},\nonumber
	\end{align}
	where the last inequality holds by the convexity of $x\mapsto e^x$. The inequality \eqref{eq:ineqconcave-exp}, together with \eqref{eq:ineqconcave-diff}, implies
	\begin{equation*}
		\frac{(-\mathcal L{\psi} - A(x){\psi})_i}{\psi_i(x)}\geq r \frac{(-\mathcal L{\psi}^1 - A(x){\psi}^1)_i}{\psi^1_i(x)} + (1-r)\frac{(-\mathcal L{\psi}^2 - A(x){\psi}^2)_i}{\psi^2_i(x)}.
	\end{equation*}
	Taking the infimum over $x$ and the supremum over all admissible ${\psi}^1$ and ${\psi}^2$ leads to the concavity of $k(\lambda)$, as in \cite[Proposition 2.10]{Nad-09}. 

	To get the strict concavity, we observe that the particular choice ${\psi}^1(x)={\phi}^{\lambda_1}(x)e^{-\lambda_1x} $,  ${\psi}^2(x)={\phi}^{\lambda_2}(x)e^{-\lambda_2x} $ and consequently ${\psi}(x)={\phi}(x)e^{\lambda x}$, where ${\phi}^1$ and ${\phi}^2 $ are the corresponding solutions to \eqref{eq:defk(lambda)} and $\vect{\phi}(x)= \exp(r\ln({\phi}^1(x)) + (1-r)\ln({\phi}^2(x)))$, also leads to 
	\begin{equation*}
		k(\lambda)\geq\inf_{x\in\mathbb R}\min_{1\leq i\leq d}\frac{-L_\lambda {\phi}(x)-A(x){\phi}(x)}{\phi_i(x)}
		=\inf_{x\in\mathbb R}\min_{1\leq i\leq d}\frac{(-\mathcal L{\psi} - A(x){\psi})_i}{\psi_i(x)}\geq r k(\lambda_1) + (1-r)k(\lambda_2),
	\end{equation*}
	and the first inequality is strict unless $\vect{\phi}= \vect{\phi}^\lambda$ is the periodic principal  eigenfunction associated with $k(\lambda)$. In this case,  recalling \eqref{eq:ineqconcave-diff} and \eqref{eq:ineqconcave-exp}, one must have $\frac{(\psi^1_{i})_{x}}{\psi^1_i}\equiv \frac{(\psi^2_{i})_{x}}{\psi^2_i}$ for all $1\leq i\leq d$, which (after integration) results in $\psi^1\equiv\psi^2$ (up to a multiplicative factor). This is a contradiction.
	
	Statement \ref{item:k(lambda)concave} is proved.\medskip

	{\bf Proof of Statement \ref{item:lambda_1k(lambda)}.} The proof is inspired by \cite[Theorem 2.11]{Nad-09}. Again, since we allow $\vect{q}$ to be non-zero, we need only prove the result in the non-divergence case.

	We first remark that for any $\lambda\in\mathbb R$, the function $ e^{-\lambda x}{\varphi}^\lambda(x)$, where ${\varphi}^\lambda$ solves \eqref{eq:defk(lambda)},  satisfies 
	\begin{equation*}
		-\mathcal L({\varphi}(x)e^{-\lambda x})-A(x)e^{-\lambda x}{\varphi}^\lambda(x)-k(\lambda) e^{-\lambda x}{\varphi}^\lambda(x)=0, 
	\end{equation*}
	hence $\lambda_1^\infty\geq  k(\lambda)$ for all $\lambda\in\mathbb R$. Therefore $\lambda_1^\infty\geq \sup_{\lambda\in\mathbb R}k(\lambda)$.

	Let ${\varphi}>{0}$ be a principal eigenfunction associated with $\lambda_1^\infty$. 
	We let 
	\begin{equation*}
		\vect{\psi}(x):=\left(\frac{\varphi_i(x+L)}{\varphi_i(x)}\right)_{1\leq i\leq d}.
	\end{equation*}
	Then, applying the Harnack inequality for fully coupled elliptic systems \cite[Theorem 8.2]{Bus-Sir-04}  to ${\varphi}$, the function ${\psi}(x)$ is uniformly bounded. We let 
	\begin{equation*}
		m:=\sup_{x\in\mathbb R}\max_{1\leq i\leq d} \psi_i(x)<+\infty.
	\end{equation*}
	Let $k\in\{1, \ldots, d\} $ be such that $\sup_{x\in\mathbb R}\psi_k(x)=m$ and  $x_n$ be a sequence such that $ \lim_{n\to\infty}\psi_k(x_n)=m$. Define $\vect{\psi}^n(x):=\vect{\psi}(x+x_n)$, and $\vect{\varphi}^n(x):=\frac{1}{\varphi_k(x_n)}\vect{\varphi}(x+x_n)$. 

	We remark that ${\psi} $ satisfies the equation:
	\begin{align*}
		\mathcal L_i \psi_i(x)&= \frac{\mathcal L_i\varphi_i(x+L)}{\varphi_i(x)}-\psi_i(x)\frac{\mathcal L_i\varphi_i(x)}{\varphi_i(x)}-2\sigma_i(x)\frac{\varphi_{i,x}(x)}{\varphi_i(x)}\psi_{i, x}(x)\\
		&=-2\sigma_i(x)\frac{\varphi_{i,x}(x)}{\varphi_{i}(x)}\psi_{i,x}(x)+\underset{j=1}{\overset{d}{\sum}}a_{ij}(x)\frac{\varphi_j(x)}{\varphi_i(x)}(\psi_j(x)-\psi_i(x)).
	\end{align*}
	Using the classical elliptic estimates, and up to the extraction of a subsequence, the sequence  $\vect{\psi}_n$ converges locally uniformly to a limit function $\vect{\psi}^\infty$, and $\vect{\varphi}^n $ converges to $\vect{\varphi}^\infty$. Extracting further, there exists $x\in [0, L] $ such that $x_n\to x^\infty\mod L$. Then, the function $\vect{\psi}^\infty$ satisfies the equation:
	\begin{equation*}
		-\mathcal L_i\psi^\infty_i(x) + 2\sigma_i(x+x^\infty)\frac{\varphi^\infty_{i,x}(x)}{\varphi^\infty_i(x)}\psi^\infty_{i, x}(x)-\underset{j=1}{\overset{d}{\sum}}a_{ij}(x+x^\infty)\frac{\varphi^\infty_j(x)}{\varphi^\infty_i(x)}(\psi^\infty_j(x)-\psi^\infty_i(x))=0.
	\end{equation*}
	Then, defining $\tilde{\mathcal L}_i\phi(x):=\mathcal L_i\phi(x) + 2\sigma(x+x^\infty)\frac{\varphi^\infty_{i,x}(x)}{\varphi^\infty_i(x)}\phi$ and the cooperative matrix field $\tilde A(x):=\left(a_{ij}(x+x^\infty)\frac{\varphi^\infty_j(x)}{\varphi^\infty_i(x)}\right)_{1\leq i,j\leq d}$, we have 
	\begin{equation*}
		-\tilde{\mathcal L}{\psi}(x)-\tilde A(x)\tilde\psi(x)\leq 0.
	\end{equation*}
	Since $\tilde A(x)$ is fully coupled, and the global maximum of ${\psi}$ is attained at $x=0$, the strong maximum principle \cite[Proposition 12.1]{Bus-Sir-04} implies ${\psi}^\infty(x)\equiv{\psi}(0)\equiv m$. Then, define $\lambda=-\ln m$. Since $\frac{\varphi^\infty_i(x+L)}{\varphi^\infty_i(x)}=\psi^\infty_i(x)\equiv m$ for $x\in\mathbb R$, the function ${\varphi}^\infty(x)e^{\lambda x}$ is $L$-periodic. Since $-\mathcal L{\varphi}^\infty-A(x){\varphi}^\infty=\lambda_1^\infty{\varphi}^\infty$, we have
	\begin{equation*}
		-L_{\lambda}(e^{\lambda x}{\varphi}^\infty(x))-A(x)e^{\lambda x}{\varphi}^\infty(x)=e^{\lambda x}\left(-\mathcal L{\varphi}^\infty - A(x){\varphi}^\infty\right)=\lambda_1^\infty e^{\lambda x}{\varphi}^\infty,
	\end{equation*}
	hence $\vect{\varphi}^\infty (x) e^{-\lambda x}$ is the periodic principal eigenfunction of $-L_\lambda -A(x)$. By the uniqueness of the periodic principal eigenvalue, $\lambda_1^\infty=k(\lambda)$. This shows $\lambda_1=\max_{\lambda\in\mathbb R}k(\lambda)$, which finishes the proof of Statement \ref{item:lambda_1k(lambda)}
	\medskip

	{\bf Proof of Statement \ref{item:k(lambda)even}:} We first deal with Assumption \ref{as:isotropic} case a), {\it i.e.} the case where both $\vect{\sigma}$ and $A$ are even. We write the proof for $\mathcal L$ written in nondivergence form, however the computations are similar in the case $\mathcal L$ is written in divergence form. Recalling the formula \eqref{eq:minimax-k(lambda)}, 
	\begin{equation*}
		k(\lambda)=\underset{\phi\in C^2_{per}(\mathbb R, \mathbb R^d)}{\underset{{{\phi}>{0}}}{\max}}\underset{x\in \mathbb R}{\inf}\, \underset{1\leq i\leq d}{\min}\frac{(-L_\lambda{\phi} -A(x){\phi})_i}{\phi_i},
	\end{equation*}
	we notice that the set of admissible test functions is invariant by the change of variables $x\leftarrow -x$. More precisely, for any $\vect{\phi}\in C^2_{per}(\mathbb R, \mathbb R^d)$ with $\vect{\phi}>\vect{0}$, there exists $\check{\vect{\phi}}(x):=\vect{\phi}(-x)$ satisfying $\check{\vect{\phi}}\in C^2_{per}(\mathbb R, \mathbb R^d)$, $\check{\vect{\phi}}>0$ and $L_\lambda \vect{\phi}(x)=(\check{L_{-\lambda}}\check{\vect{\phi}})(-x)=L_{-\lambda}\check{\vect{\phi}}(-x)$, so that:
	\begin{equation*}
		\underset{x\in \mathbb R}{\inf}\, \underset{1\leq i\leq d}{\min}\frac{(-L_{\lambda}\vect{\phi}(x) -A(x)\vect{\phi}(x))_i}{\phi_i(x)} =  \underset{x\in \mathbb R}{\inf}\, \underset{1\leq i\leq d}{\min}\frac{(-L_{-\lambda}\check{\vect{\phi}}(-x) -A(x)\check{\vect{\phi}}(-x))_i}{\check{\phi}_i(-x)}.
	\end{equation*}
	Hence, for all $\vect{\phi}\in C^2_{per}(\mathbb R, \mathbb R^d)$ satisfying $\vect{\phi}>\vect{0}$:
	\begin{align*}
		\underset{x\in \mathbb R}{\inf}\, \underset{1\leq i\leq d}{\min}\frac{(-L_\lambda\vect{\phi}(x) -A(x)\vect{\phi}(x))_i}{\phi_i(x)}
		&\leq\underset{\phi'\in C^2_{per}(\mathbb R), \mathbb R^d}{\underset{{\vect{\phi'}>\vect{0}}}{\max}}\underset{x\in \mathbb R}{\inf}\, \underset{1\leq i\leq d}{\min}\frac{(-L_{-\lambda}\vect{\phi'} -A(x)\vect{\phi'})_i}{\phi'_i}=k(-\lambda).
	\end{align*}
	Taking the supremum on $\vect{\phi}$, we get $k(\lambda)\leq k(-\lambda)$. Changing $\lambda $ into $-\lambda$, we similarly get $k(-\lambda)\leq k(\lambda) $, which shows that the equality $k(-\lambda)=k(\lambda)$ holds.\medskip

	Next we consider that Assumption \ref{as:isotropic} case b) holds, {\it i.e.},  $\mathcal L=\mathcal L^d$ is in divergence form \eqref{eq:divergence-form} and  $A$ is a symmetric matrix (i.e. equals its transpose for all $x\in\mathbb R$), then the operator $L_{-\lambda}$ is the adjoint of the operator $L_\lambda$ for the canonical scalar product in $L^2_{per}(\mathbb R)^d$:
	\begin{equation*}
		\langle \vect{\varphi}, \vect{\psi}\rangle:=\langle \vect{\varphi}, \vect{\psi}\rangle_{L^2_{per}(\mathbb R)^d}=\sum_{i=1}^d\int_0^1\varphi_i(x)\psi_i(x)\dd x,
	\end{equation*}
	and it follows easily from the Krein-Rutman Theorem \cite[Theorem 7.C]{Zei-86} that $k(-\lambda)=k(\lambda)$. 
	
	This finishes the proof of Statement \ref{item:k(lambda)even}.
\end{proof}

We are now in a position to prove Proposition \ref{prop:isotropic}.
\begin{proof}[Proof of Proposition \ref{prop:isotropic}]
	We first notice that $\lambda_1^{per} = k(0)$. Since the operator $\mathcal L$ in Proposition \ref{prop:isotropic} satisfies Assumption \ref{as:isotropic}, Lemma \ref{prop:k(lambda)} statement \ref{item:k(lambda)even} implies that $k(\lambda) $ is even, and Lemma \ref{prop:k(lambda)} statement \ref{item:k(lambda)concave} implies that it is concave and continuous. Hence, 
	\begin{equation*}
		\lambda_1^{per}=k(0)=\max_{\lambda\in\mathbb R}k(\lambda).
	\end{equation*}
	Finally, by Lemma \ref{prop:k(lambda)} statement \ref{item:lambda_1k(lambda)} we have $\lambda_1=\max_{\lambda\in\mathbb R}k(\lambda)$, and by Lemma \ref{prop:gen-princ-eig} statement \ref{item:limitRbig} we have $\lambda_1=\lim_{R\to+\infty}\lambda_1^R$. This ends the chain of equalities:
	\begin{equation*}
		\lambda_1^{per}=k(0)=\max_{\lambda\in\mathbb R}k(\lambda)=\lambda_1=\lim_{R\to+\infty}\lambda_1^R,
	\end{equation*}
	which proves Proposition \ref{prop:isotropic}.
\end{proof}
Last, we prove our statements on the formula for the minimal speed.
\begin{proof}[Proof of Proposition \ref{prop:minspeed}]
    Statement $\mathrm{(i)}$ is a direct consquence of the definition of $c^*$ in \eqref{eq:defc*}. Next, by the fact that $k(0)=\lambda_1^{per} <0$ and \eqref{eq:klesssquare}, the infimum on the right-hand side of \eqref{eq:defc*} is attained by some $\lambda^*>0$. Furthermore, since $k(\lambda)$ is strictly concave, $\lambda^*$ is the only solution of the equation $\lambda c^*=k(\lambda)$. This proves $\mathrm{(ii)}$. Statements $\mathrm{(ii)} $ and $ \mathrm{(iii)}$ follow directly from the strict concavity of $\lambda\mapsto k(\lambda)$. The continuity with respect to $A$ follows easily from the sequential characterisation of continuity and the regularising properties of elliptic operators.
\end{proof}

\subsection{Speed of {sublinear} systems}
\label{sec:coop-sys}

Our main goal is to prove Theorem \ref{thm:lin-det}. This theorem follows as a direct consequence of  Lemma \ref{lem:lower-spreading} and Lemma \ref{lem:upper-spreading} below.
\begin{lem}[Lower spreading speed]\label{lem:lower-spreading}                                                              
    Let $\mathcal L$ be a diagonal uniformly elliptic $L$-periodic operator, and $\vect{f}=(f_1, \ldots, f_d)$ be a cooperative {sublinear} nonlinearity satisfying Assumption \ref{as:coop}.   
	Assume that there is a periodic function $p(x)\gg 0$ solution to the equation                    
	\begin{equation*}                  
		p_t-\mathcal Lp =f(x, p(x))  
	\end{equation*}                     
	which attracts every nontrivial periodic initial condition $p(x)\geq u_0(x)\geq \delta\mathbf{1}\gg 0$.           
                                                                                                                           
	Let $u(t, x)$ be a solution of \eqref{eq:gen-syst} associated with an initial condition which is positive on a half-line
	 \begin{equation*}
		 \inf_{x\leq -K}\min_{1\leq i\leq d}u_i(0, x)>0
	 \end{equation*} for some $K>0$. Then  for any  $c<c^*$, we have          
	\begin{equation*}                           
		\limsup_{t\to+\infty}\sup_{x\leq ct} \Vert u(t,x)-p(x)\Vert =0,  
	\end{equation*}                                                                          
	where $ c^*=\inf_{\lambda>0}-\frac{k(\lambda)}{\lambda} $ is defined by \eqref{eq:defc*} in Proposition \ref{prop:minspeed}. 
\end{lem}                                                                                                                  
\begin{proof}                                                                                                              
		Let $\mathcal H:=\underset{1\leq i\leq d}{\bigcup}\mathbb R\times \{i\}\subset \mathbb R^2$. We remark that any continuous function $\vect{u}:\mathbb R\times\mathbb R\to\mathbb R^d$ can be represented as a function $u:\mathbb R\times\mathcal H\to \mathbb R$ by letting $u(t, x, i)=u_i(t, x)$. Hence a vector function can be represented by a scalar function on the  
		habitat $\mathcal H$. In particular, system \eqref{eq:gen-syst} makes sense as an equation on $u\in BUC(\mathcal H)$.
                                                                                                                           
		Let $\delta>0$, $\tau>0$ be given and let  $Q:BUC(\mathcal H)\to BUC(\mathcal H)$ be defined by            
		\begin{equation*}                                                                                          
			Q[v_0](x):=v(\tau, x).       
		\end{equation*}                                                                                            
		where $v(t,x)$  is the solution to \eqref{eq:gen-syst} satisfying $v(0, x)=v_0(x)$. It follows from standa rd arguments that $Q$ is monotone.    
                                                                                                                           
		Let us now check that the Hypothesis 2.1 in \cite{Wei-02} are satisfied for $\mathcal H$ and $Q^\delta$. L et us mention at this point that our setting is a little different from the one of the paper of Weinberger \cite{Wei-02}, since $\mathcal H$ is not left invariant by a 2-dimensional lattice, as it is bounded in one direction. In our case, $\mathcal H$ is periodic with respect to the 1-dimensional lattice $L\mathbb Z\times\{0\}$ (for which $Q^\delta$ is periodic). However, as stated in Section 8 of \cite{Wei-02} (Partially bounded habitats), all the results in \cite{Wei-02} can be adapted in directions $\xi$ which are not orthogonal to all members of our lattice $L\mathbb Z\times\{0\}$ (which are the directions in which the spreading happens). In the rest of the proof we will use those results. 
                                                                                                                           
		Let us now check point by point that Hypothesis 2.1 is satisfied:                                          
		\begin{enumerate}[label={\it\roman*.}]                                                                    
			\item $\mathcal H$ is not contained in any 1-dimensional subset of $\mathbb R^2$.                  
			\item $Q$ is monotone because $f$ is quasi-monotone. 
			\item $\mathcal H$ is invariant under translation by elements of $ L\mathbb Z\times\{0\}$, and $Q^ \delta$ is periodic with respect to $ L\mathbb Z\times\{0\}$. Moreover there is a bounded subset $P:=[0, 1)\times \{1, \ldots. d\}\subset \mathcal H$, such that any $x\in\mathcal H$ has a unique representation of the form $x=l+p$ with $l\in L\mathbb Z\times\{0\}$ and $p\in P$. 
			\item $Q(0)=\pi_0:\equiv 0$, and there exists $\pi_1:=p(x)>0$    
				which is the unique nonnegative nontrivial fixed point of $Q^\delta$. 
			\item $Q$ is  continuous.                                                                          
			\item Due to the classical parabolic estimates, $Q$ is sequentially compact for the topology of th e local uniform convergence on $BUC(\mathcal H)$.
		\end{enumerate}                                                                                            
                                                                                                                           
		In particular, \cite[Theorem 2.1]{Wei-02} applies to $Q$ and there exists a spreading speed $c^*$ associated with $Q^\delta$. Moreover, because of Assumption \ref{as:coop}, \cite[Theorem 2.4]{Wei-02} implies 
		\begin{equation*}                                                                                          
			c^*(Q)\geq \inf_{\lambda>0}\frac{-k^\delta(\lambda)}{\lambda} =: c_\delta^*,                      
		\end{equation*}
		where $(k^\delta(\lambda), \varphi^{\delta, \lambda}(x))$ is the periodic principal eigenvalue solution to
		\begin{equation*}
			-e^{\lambda x}\mathcal L(\varphi^{\delta, \lambda}(x)e^{-\lambda x})-A^\delta(x)\varphi^{\delta, \lambda}(x) = k^\delta(\lambda)\varphi^{\delta, \lambda}(x).
		\end{equation*}

		Since $A^\delta(x)\to A(x) $ as $\delta\to 0$, it follows from classical arguments that $c^*_\delta\to c^*$ as $\delta\to 0$.
		
		This completes the proof of Lemma \ref{lem:lower-spreading}.
\end{proof}

Let us turn to the  upper estimates of the spreading speed:
\begin{lem}[Upper spreading speed]\label{lem:upper-spreading}
    Let $\mathcal L$ be a uniformly elliptic $L$-periodic operator, and $\vect{f}$ be a $L$-periodic {sublinear} nonlinearity. Assume that  $D\vect{f}(x, \vect{0})=:A(x)$ is cooperative and fully coupled. Then, for any $c>c^*(\mathcal L+A(x))$,  we have:
	\begin{equation*}
		\limsup_{t\to+\infty}\sup_{x\geq ct} \max_{1\leq i\leq d}u_i(t,x)=0,
	\end{equation*}
	for any $\vect{u}(t,x)$  solution to \eqref{eq:gen-syst}, provided there is $K>0$ such that $u(0, x)\equiv 0$ for all $x\geq K$. 
\end{lem}
\begin{proof}
	The result is an immediate consequence of the comparison principle applied to ${u}(t,x)$ and the function $M{\varphi}^{\lambda_*}(x)e^{-\lambda_*(x-ct)}$, where $\lambda_*>0$ is a minimizer for $-\frac{k(\lambda)}{\lambda}$, ${\varphi}^{\lambda_*}$ is the associated $1$-periodic principal eigenfunction, and $M>0$ is a large constant satisfying 
	\begin{equation*}
		\vect{u}_0(x)\leq M\vect{\varphi}^{\lambda_*}(x)e^{-\lambda_*x}.\qedhere
	\end{equation*}
	
\end{proof}

\begin{proof}[Proof of Theorem \ref{thm:lin-det}]
Let $c^*$ be the number defined as 
	\begin{equation*}
		c^*:=\inf_{\lambda>0} \frac{-k(\lambda)}{\lambda},
	\end{equation*}
	where $k(\lambda)$ is defined in Lemma \ref{prop:k(lambda)} Statement \ref{item:eigenpairk(lambda)}  
	with $A(x)=Df(x,0 )$. Recall that $c^*$ is well-defined by Proposition \ref{prop:minspeed}.

	Let $u_0\in BUC(\mathbb R, \mathbb R^d_+)$ be given and $u(t, x)$ be the solution to \eqref{eq:gen-syst} satisfying $u(0, x)=u_0(x)$. We assume that $u_0(x) = 0 $ for all $x\geq 0$ and that
	\begin{equation*}
		\liminf_{x\to-\infty} \Vert u_0(x)\Vert>0.
	\end{equation*}

	It has been shown in Lemma \ref{lem:upper-spreading} that
	\begin{equation*}
		\limsup_{t\to+\infty}\sup_{x\geq ct} \Vert u(t,x)\Vert_\infty=\limsup_{t\to+\infty}\sup_{x\geq ct} \max_{1\leq i\leq d}u_i(t,x)=0,
	\end{equation*}
	therefore Statement (ii) in Definition \ref{def:prop-speed} holds. 

	Let $\eta>0$ and $f^-$ be as in Definition \ref{def:super-monotone}. Recall that, by Proposition \ref{prop:sub-sub-barrier},  $f^-$ can be chosen so that the equation $-\mathcal Lp=f^-(x,p)$ admits a positive periodic fixed point $p(x)$ which attracts any periodic initial condition $0\leq u_0(x)\leq p(x)$, and $\Vert p\Vert_{L^\infty(\mathbb R)^d}\leq \eta$.
	
	We define  $\underline{u}(t,x)$ as the unique solution to 
	\begin{equation}\label{eq:ubeta}
		\left\{\begin{aligned}\relax
			&\underline{u}_t(t,x)-\mathcal L\underline{u}(t,x)=f^{-}(x, \underline{u}(t,x)), \\
			&\underline{u}(0, x)=\min\big({u}_0(x), p(x)\big). 
		\end{aligned}\right.
	\end{equation}
	By Proposition \ref{prop:sub-sub-barrier}, the interval 
	\begin{equation*}
		[0, p(x)]:=\{v(x)\in BUC(\mathbb R)^d\,|\,0\leq v(x)\leq p(x)\}, 
	\end{equation*}
	is positively invariant by the semiflow generated by \eqref{eq:ubeta}. Therefore $\underline{u}(t, x)\leq p(x)$ for all $t\geq 0 $ and $x\in\mathbb R$. By Theorem \ref{thm:super-monotone}, we have then 
	\begin{equation*}
		u(t,x)\geq \underline{u}(t,x) \text{ for all } t\geq 0 \text{ and } x\in\mathbb R.
	\end{equation*}
	Applying Lemma \ref{lem:lower-spreading} to $u$ and  $\underline{u}$, we find that 
	\begin{equation*}
		\lim_{t\to +\infty} \inf_{x\leq ct} u(t,x) \geq \lim_{t\to +\infty} \inf_{x\leq ct} \underline{u}(t,x) \geq p(x)\geq \delta \mathbf{1}\gg 0, 
	\end{equation*}
	for all $c<c^*$ and $\delta>0$ sufficiently small, hence we have shown Item (i) in Definition \ref{def:prop-speed}.
\end{proof}

\subsection{Traveling waves: proof of  Theorem \ref{thm:TW}}
\label{sec:lin-det}

\label{sec:proofs-gen}

We now prove Theorem \ref{thm:TW} and the existence of traveling waves. The proof is done by constructing an upper barrier and a lower barrier. The construction of the upper barrier is rather simple as the following Lemma shows: 
\begin{lem}[Upper barrier]\label{lem:upper-barrier}
	Let $\lambda>0$ and $c>0$ be such that $\lambda c+k(\lambda)\geq 0$. Define
	\begin{equation}\label{eq:upperu}
		\overline{u}(t,x):= e^{-\lambda(x-ct)}\vect{\varphi}_\lambda(x)
	\end{equation}
	where $\vect{\varphi}_\lambda$ is the solution to \eqref{eq:defk(lambda)} associated with $\lambda$, and satisfies $\Vert \vect{\varphi}_\lambda\Vert_{L^\infty}=1$. Any solution $u(t, x)$ of \eqref{eq:gen-syst} starting from below $\overline{u}(t,x)$ at $t=0$ stays below $\overline{u}(t,x)$ at later times. More precisely, if the inequality
	\begin{equation*}
		0\leq u(0, x)\leq \overline{u}(0,x)  
	\end{equation*}
	holds componentwise for all $x\in\mathbb R$, then for all $t>0$ the inequality
	\begin{equation*}
		0\leq u(t,x)\leq \overline{u}(t,x)
	\end{equation*}
	holds componentwise for all $x\in\mathbb R$.
\end{lem}
\begin{proof}
	We remark that	
	\begin{equation}\label{eq:supersol-u}
		\overline{u}_t(t,x)-\mathcal L\overline{u}(t,x)=\big(\lambda c +A(x)+k(\lambda)\big)\overline{u}(t,x)
	\end{equation}
	Since $\lambda c+ k(\lambda) \geq 0$, we have
	\begin{align*}
		{u}_t(t,x)-\mathcal L{u}(t,x)=f(x, u(t,x))\leq  A(x)u(t,x)\leq \big(\lambda c+k(\lambda)+A(x)\big)u(t,x).
	\end{align*}
	Therefore $u(t,x)$ is a subsolution to {\eqref{eq:supersol-u}}. By the comparison principle for cooperative  parabolic systems, we have
	\begin{equation*}
		u(t, x)\leq \overline{u}(t,x)
	\end{equation*}
	for all $t>0$ whenever $u(0, x)\leq \overline{u}(0,x)$.
\end{proof}
Next we construct a lower barrier. The function $\xi$ in the following Lemma will play an important role in this construction for the case $c>c^*$.
\begin{lem}\label{lem:lower-barrier}
	Under the assumptions of Theorem \ref{thm:TW}, let $c>c^*$ and $\lambda>0$ be such that $\frac{-k(\lambda)}{\lambda}=c$. 
	Define 
	\begin{equation}\label{eq:xi}
		\xi(t, x)=e^{-\lambda(x-ct)}\varphi_\lambda(x)-\omega e^{-\mu(x-ct)}\varphi_\mu(x), 
	\end{equation}
	where $\mu>0$ satisfies $k(\mu)+\mu c>0$ and $\lambda<\mu<\lambda(1+\beta)$ where $\beta>0 $ is the constant defined in the assumptions of Theorem \ref{thm:TW}.
	There exists $\omega^*>0$ such that, for all $\omega\geq \omega^*$, the function $\xi(t,x)$ satisfies the differential inequality
	\begin{equation}\label{eq:subsol-xi}
		(\xi_i)_t(t,x)-\mathcal L_i \xi_i(t,x)\leq f^-_i(x, \xi(t,x)) 
	\end{equation}
	as well as $\Vert \xi(t,x)\Vert_\infty \leq \eta$,  whenever there is $i\in\{1,\ldots, d\}$ such that $\xi_i(t,x)>0$,  where $f^-$ and $\eta$ are as in Definition \ref{def:super-monotone}. 

	In particular, if $\omega>\omega^*$, any solution $u(t,x)$ of \eqref{eq:gen-syst} satisfying the inequality $u(0, x)\geq \max\big(\xi(0,x), 0\big)$ also satisfies 
	\begin{equation}\label{eq:super-xi}
		u(t,x)\geq \max\big(\xi(t,x), 0\big) \text{ for all }t>0 \text{ and } x\in\mathbb R.
	\end{equation}
\end{lem}
\begin{proof}
	The existence of $\mu$ as defined in the statement of the Lemma is a consequence of $c>c^*$ and the properties of the principal eigenvalue $k$, see Proposition \ref{prop:minspeed}. Our goal is to find $\omega>0$ such that 
	\begin{equation}\label{eq:sub-sol}
		\xi_t(t,x)-\mathcal L \xi(t,x)\leq f^-(x, \xi(t,x)) \text{ whenever }\xi(t, x)>0.
	\end{equation}
	Let us select $(t,x)$ such that $\xi(t, x)>0$. Recall that, for all $\nu>0$, we have the equation $-\mathcal L(\varphi_\nu(x)e^{-\nu x}) = (A(x)+k(\nu))e^{-\nu x}\varphi_\nu(x)$ by definition of $k(\nu)$. We compute
	\begin{align*}
		\xi_t-\mathcal L\xi &= \left(A(x)+k(\lambda)+\lambda c\right)e^{-\lambda(x-ct)}\varphi_\lambda(x)-\big(A(x)+k(\mu) + \mu c\big)e^{-\mu(x-ct)}\omega\varphi_\mu(x) \\
		&=A(x)\xi(t,x) - (k(\mu)+\mu c)\omega e^{-\mu(x-ct)}\varphi_\mu(x).
	\end{align*}
	It follows from our assumption in the statement of Theorem \ref{thm:TW} that 
	\begin{equation*}
		\Vert f^-(x, u)-A(x)u\Vert_\infty \leq M\Vert u\Vert_\infty^{1+\beta}, 
	\end{equation*}
	for some constant $M>0$. 
	\begin{align*}
		\Vert f^-(x, \xi(t,x))-A(x) \xi(t,x)\Vert_\infty &\leq M \Vert \xi(t,x)\Vert_\infty^{1+\beta} \\ 
		&\leq M\left(\sup_{y\in\mathbb R}\max_{1\leq i\leq d}(\varphi_\lambda)_i(y)\right)^{1+\beta} e^{-(1+\beta)\lambda(x-ct)},
	\end{align*}
	and it follows that 
	\begin{equation*}
		(k(\mu)+\mu c)\omega e^{-\mu(x-ct)}\varphi_\mu(x)\geq A(x)\xi(t,x)- f^-(x, \xi(t,x))
	\end{equation*}
	for all $x-ct\geq \frac{1}{(1+\beta)\lambda-\mu}\left[\ln\left(\frac{M\sup_{y\in\mathbb R}\max_{1\leq i\leq d}(\varphi_{\lambda})_i^{1+\beta}(y)}{(k(\mu)+\mu c)\inf_{y\in\mathbb R}\min_{1\leq i\leq d}(\varphi_\mu)_i(y)}\right)-\ln \omega\right]$. On the other hand, because of the specific form of $\xi$, we have $\xi(t,x)\ll 0$ for all $x-ct\leq \frac{-1}{\mu-\lambda}\ln\left(\frac{\sup_{y\in\mathbb R}\max_{1\leq i\leq d}(\varphi_{\lambda})_i(y)}{\inf_{y\in\mathbb R}\min_{1\leq i\leq d}(\varphi_\mu)_i(y)}\right) + \frac{1}{\mu-\lambda}\ln\omega$. Therefore if $\omega $ is sufficiently large, namely
	\begin{equation*}
		\ln\omega>\frac{\mu-\lambda}{\beta \lambda}\displaystyle \ln\left(\frac{M\sup_{y\in\mathbb R}\max_{1\leq i\leq d}(\varphi_{\lambda})_i^{1+\beta}(y)}{(k(\mu)+\mu c)\inf_{y\in\mathbb R}\min_{1\leq i\leq d}(\varphi_\mu)_i(y)}\right) - \frac{(1+\beta)\lambda-\mu}{\beta\lambda}\ln\left(\frac{\sup_{y\in\mathbb R}\max_{1\leq i\leq d}(\varphi_{\lambda})_i(y)}{\inf_{y\in\mathbb R}\min_{1\leq i\leq d}(\varphi_\mu)_i(y)}\right),
	\end{equation*}
	then $\xi(t,x)>0$ implies that  $(k(\mu)+\mu c)\omega e^{-\mu(x-ct)}\varphi_\mu(x)\geq A(x) \xi(t,x)-f^-(x, \xi(t,x))$ and therefore 
	\begin{align*}
		\xi_t - \mathcal L\xi(t,x)& \leq A(x)\xi(t, x) - (k(\mu)+\mu c)\omega e^{-\mu(x-ct)}\varphi_\mu(x) \\ 
		&\leq A(x)\xi(t,x) - \big(A(x)\xi(t,x) - f^-(x, \xi(t,x))\big)\\
		&=f^-(x, \xi(t,x)).
	\end{align*}
	We have shown that \eqref{eq:sub-sol} holds for $\omega>0$ sufficiently large. Finally 
	\begin{align*}
		\sup_{x\in\mathbb R} \xi(t,x)&\leq \sup_{x\in\mathbb R} \left(e^{-\lambda x}\sup_{y\in\mathbb R}\varphi_\lambda(y)-\omega e^{-\mu x}\inf_{y\in\mathbb R} \varphi_{\mu}(y)\right)\\
		&\leq \sup_{y\in\mathbb R}\varphi_\lambda(y)\left(\frac{\lambda\sup_{y\in\mathbb R}\varphi_\lambda(y)}{\mu \inf_{y\in\mathbb R} \varphi_{\mu}(y)}\right)^{\frac{\lambda}{\mu-\lambda}}\omega^{-\frac{\lambda}{\mu-\lambda}}, 
	\end{align*}
	therefore the supremum of $\xi(t,x)$ is arbitrarily small for $\omega $ sufficiently large.

	To finish our argument we remark that (if $\omega\geq \omega^*$) $\xi(t,x)$ and $u(t,x)$ are respectively a sub-solution and a super-solution to the cooperative system
	\begin{equation*}
		v_t(t,x)-\mathcal Lv(t,x) = f^-(x,v(t,x)),
	\end{equation*} 
	which admits a comparison principle. Therefore if $u(0,x)\geq \xi(0,x)$ for all $x\in\mathbb R$, then $u(t,x)\geq \xi(t,x)$ for all $t>0$ and $x\in\mathbb R$. The Lemma is proved.
\end{proof}

In the critical case $c=c^*$, we need to define $\xi$ differently.
Recall that, by Proposition \ref{prop:minspeed}, there exists a unique $\lambda^*$ such that $c^*=\frac{k(\lambda^*)}{\lambda^*}$.
By lemma \ref{prop:k(lambda)}, $k(\lambda)$ is strictly concave and analytic, therefore there exists a nonnegative integer $m\geq 0$ such that the multiplicity of $k(\lambda)+c^*\lambda$ is $2m+2$, in the sense that 
\begin{equation*}
	\lambda^* c^*+k(\lambda^*) = 0, \, c^*+ k'(\lambda^*)=0, \, k^{(i)}(\lambda^*)=0 \text{ for } 2\leq i\leq 2m+1,  \text{ and } k^{(2m+2)}(\lambda^*)<0. 
\end{equation*}
\begin{lem}\label{lem:lower-barrier-critical}
	Let  the assumptions of Theorem \ref{thm:TW} hold. Define 
	\begin{equation}\label{eq:xi^*}
		\xi(t, x)=\begin{cases} 
			\max\left(\frac{\partial^{2m+2}}{\partial \lambda^{2m+2}}\left(\left.e^{-\lambda(x-c^*t)}\varphi_\lambda(x)\right)\right|_{\lambda=\lambda^*}-\omega e^{-\lambda^*(x-c^*t)}\varphi_{\lambda^*}(x), 0\right) & \text{ if } x-c^*t\geq \left(\frac{\omega}{2}\right)^{\frac{1}{2m+2}},  \\ 
			0 & \text{ if }   x-c^*t< \left(\frac{\omega}{2}\right)^{\frac{1}{2m+2}}
		\end{cases}
	\end{equation}
	where the maximum is taken componentwise, then there exists $\omega^*>0$ such that, for all $\omega\geq \omega^*$, the function $\xi(t,x)$ satisfies the differential inequality
	\begin{equation}\label{eq:subsol-xi*}
		(\xi_i)_t(t,x)-\mathcal L_i \xi_i(t,x)\leq f^-_i(x, \xi(t,x)) 
	\end{equation}
	as well as $\Vert \xi(t,x)\Vert_\infty \leq \eta$ whenever $\xi_i(t,x)>0$ for some $i\in\{1, \ldots, d\}$, where $f^-$ and $\eta$ are as in Definition \ref{def:super-monotone}. 
	In particular, if $\omega\geq\omega^*$, any solution $u(t,x)$ of \eqref{eq:gen-syst} satisfying the inequality $u(0, x)\geq \max\big(\xi(0,x), 0\big)$ also satisfies 
	\begin{equation}\label{eq:super-xi*}
		u(t,x)\geq \max(\xi(t,x), 0)\; \text{ for all }t>0 \text{ and } x\in\mathbb R.
	\end{equation}
\end{lem}
\begin{proof}
	Let us define  the function $\Xi(t, x):=e^{-\lambda(x-c^*t)}\varphi_\lambda(x)$ for $\lambda>0$, $\omega>0$  and $(t, x)\in\mathbb R$. We have
	\begin{equation}\label{eq:eigen-lambda*}
		\Xi_t(t, x)-\mathcal L\hspace{2pt}\Xi(t,x)=A(x)\Xi(t, x)+\big(\lambda c^*+k(\lambda)\big)\Xi(t,x), 
	\end{equation}
	then by the analyticity of $k(\lambda)$ and $\varphi_\lambda$ with respect to $\lambda$ we have, taking $(2m+2) $ times the derivative in the above expression: 
	\begin{align*}
		\partial_t\left(\frac{\partial^{2m+2}}{\partial \lambda^{2m+2}}\Xi(t,x)\right)-\mathcal L\left(\frac{\partial^{2m+2}}{\partial \lambda^{2m+2}}\Xi(t,x)\right)& = A(x)\left(\frac{\partial^{2m+2}}{\partial \lambda^{2m+2}}\Xi(t,x)\right)\\ 
		&\quad +\sum_{j=0}^{2m+2}\binom{2m+2}{j}(\lambda c^*+k(\lambda))^{(j)}(\lambda)(e^{-\lambda(x-c^*t)}\varphi_\lambda(x))^{(2m+2-j)} .
	\end{align*}
	If $\lambda=\lambda^*$ we have 
	\begin{equation}\label{eq:derivLambda}
		\partial_t\left(\Xi^{(2m+2)}(t,x)\right)-\mathcal L\left(\Xi^{(2m+2)}(t,x)\right) = A(x)\left(\Xi^{(2m+2)}(t,x)\right)+k^{(2m+2)}(\lambda^*)e^{-\lambda(x-c^*t)}\varphi_\lambda(x),
	\end{equation}
	where $k^{(2m+2)}(\lambda^*)<0$ (because of the concavity of $k$) and $\Xi^{(2m+2)}(t,x):=\frac{\partial^{2m+2}}{\partial \lambda^{2m+2}}\Xi(t,x)$. Next the leading term in $\Xi^{(2m+2)}(t,x)$ when $x-c^*t\to+\infty$ is $(x-c^*t)^{2m+2}e^{-\lambda(x-c^*t)}\varphi_\lambda(x)$, therefore 
	\begin{equation*}
		\Xi^{(2m+2)}(t, x) \sim (x-c^*t)^{2m+2}e^{-\lambda (x-c^*t)}\varphi_\lambda(x)\text{ when }x-c^*t\to+\infty,
	\end{equation*}
	and there is $s_0\in\mathbb R$ such that 
	\begin{equation}\label{eq:ineq-equiv}
		\frac{1}{2}(x-c^*t)^{2m+2}e^{-\lambda (x-c^*t)}\varphi_\lambda(x)\leq \Xi^{(2m+2)}(t, x)\leq 2(x-c^*t)^{2m+2}e^{-\lambda (x-c^*t)}\varphi_\lambda(x)\leq\eta\quad\text{ if }\; x-c^*t\geq s_0.
	\end{equation}
	 Now, we define $\xi(t,x):=\Xi^{(2m+2)}(t,x)-\omega e^{-\lambda^*(x-c^*t)}\varphi_{\lambda^*}(x)$.
	Since $e^{-\lambda^*(x-c^*t)}\varphi_{\lambda^*}(x)$ is a solution of \eqref{eq:eigen-lambda*} with $\lambda=\lambda^*$, \eqref{eq:derivLambda} implies the following:
	\begin{equation}\label{eq:derivxi}
		\partial_t\left(\xi(t,x)\right)-\mathcal L\left(\xi(t,x)\right) = A(x)\left(\xi(t,x)\right)+k^{(2m+2)}(\lambda^*)e^{-\lambda(x-c^*t)}\varphi_\lambda(x).
	\end{equation}
	Next, since $\varphi_{\lambda^*} (x) $ is uniformly positive on $\mathbb R$,  for $\omega $ sufficiently large, we have $\xi(t, x)<0$ for all $x-c^*t\in \big(0,  s_0\big]$, and \eqref{eq:ineq-equiv} implies 
	\begin{equation*}
		\left[\frac{1}{2}(x-c^*t)^{2m+2}-\omega\right]e^{-\lambda^*(x-c^*t)}\varphi_{\lambda^*}(x)\leq \xi(t,x)\leq \left[{2}(x-c^*t)^{2m+2}-\omega\right]e^{-\lambda^*(x-c^*t)}\varphi_{\lambda^*}(x)
	\end{equation*}
	if $x-c^*t>s_0$. In particular, if $x-c^*t\leq \max\left(\left(\frac{\omega}{2}\right)^{\frac{1}{2m+2}}, s_0\right)$, then $\xi(t, x)\leq 0$.  If $x-c^*t \geq \max\left(\left(\frac{\omega}{2}\right)^{\frac{1}{2m+2}}, s_0\right)$, we have 
	\begin{align*}
		\Vert \xi(t,x)\Vert &\leq \max\left[\omega-\frac{(x-c^*t)^{2m+2}}{2}, 2(x-c^*t)^{2m+2}-\omega\right]e^{-\lambda^*(x-c^*t)}\Vert \varphi_{\lambda^*}(x)\Vert 
	\end{align*}
	In order to estimate the right-hand side of the  above inequality, we first consider the case when $\left(\frac{\omega}{2}\right)^{\frac{1}{2m+2}} \leq (x-c^*t)\leq (2\omega)^{\frac{1}{2m+2}}$. Then we have
	\begin{align*}
		\Vert \xi(t,x)\Vert &\leq 3\omega e^{-\lambda^*(x-c^*t)}\Vert \varphi_{\lambda^*}(x)\Vert \\
			&\leq 3\omega e^{-\frac{\beta}{1+\beta}\lambda^*(x-c^*t)}\Vert\varphi_{\lambda^*}(x)\Vert^{\frac{\beta}{1+\beta}}\times e^{-\frac{1}{1+\beta}\lambda^*(x-c^*t)}\Vert\varphi_{\lambda^*}(x)\Vert^{\frac{1}{1+\beta}} \\
			&\leq 3\omega e^{-\lambda^*\frac{\beta}{1+\beta}\left(\frac{\omega}{2}\right)^{\frac{1}{2m+2}}}\Vert\varphi_{\lambda^*}(x)\Vert^{\frac{\beta}{1+\beta}}\times e^{-\frac{1}{1+\beta}\lambda^*(x-c^*t)}\Vert\varphi_{\lambda^*}(x)\Vert^{\frac{1}{1+\beta}}\\
			&\leq K_1(\omega)e^{-\frac{1}{1+\beta}\lambda^*(x-c^*t)}\Vert\varphi_{\lambda^*}(x)\Vert^{\frac{1}{1+\beta}},
	\end{align*}
	where 
	\begin{equation*}
		K_1(\omega):=3\omega e^{-\lambda^*\frac{\beta}{1+\beta}\left(\frac{\omega}{2}\right)^{\frac{1}{2m+2}}}\max_{x\in\mathbb R}\Vert\varphi_{\lambda^*}(x)\Vert^{\frac{\beta}{1+\beta}}
	\end{equation*}
	and $\beta>0$ is the constant from  \eqref{eq:conv-f}. Next we consider the case when $ (x-c^*t)\geq (2\omega)^{\frac{1}{2m+2}}$. We have
	\begin{align*}
		\Vert \xi(t,x)\Vert &\leq 2(x-c^*t)^{2m+2}e^{-\lambda^*(x-c^*t)}\Vert\varphi_{\lambda^*}(x)\Vert \\
		&\leq 2(x-c^*t)^{2m+2}e^{-\frac{\beta}{1+\beta}\lambda^*(x-c^*t)}\Vert\varphi_{\lambda^*}(x)\Vert^{\frac{\beta}{1+\beta}}\times e^{-\frac{1}{1+\beta}\lambda^*(x-c^*t)}\Vert\varphi_{\lambda^*}(x)\Vert^{\frac{1}{1+\beta}}  \\
			&\leq 2(x-c^*t)^{2m+2}e^{-\frac{\beta}{1+\beta}\lambda^*(x-c^*t)}\Vert\varphi_{\lambda^*}(x)\Vert^{\frac{\beta}{1+\beta}}\times e^{-\frac{1}{1+\beta}\lambda^*(x-c^*t)}\Vert\varphi_{\lambda^*}(x)\Vert^{\frac{1}{1+\beta}} \\
			&\leq K_2(\omega)e^{-\frac{1}{1+\beta}\lambda^*(x-c^*t)}\Vert\varphi_{\lambda^*}(x)\Vert^{\frac{1}{1+\beta}}, 
	\end{align*}
	where 
	\begin{equation*}
		K_2(\omega):=2\sup_{s\geq (2\omega)^{\frac{1}{2m+2}}}s^{2m+2}e^{-\frac{\beta}{1+\beta}\lambda^*s}\max_{x\in\mathbb R}\Vert\varphi_{\lambda^*}(x)\Vert^{\frac{\beta}{1+\beta}}.
	\end{equation*}
	Let $\kappa>0$ be a constant such that  $\varphi_{\lambda^*}(x)\geq \kappa \Vert \varphi_{\lambda^*}(x)\Vert_\infty\mathbf{1}$ for any $x\in\mathbb R$ and let  $M>0$ be the constant that appears in \eqref{eq:conv-f}.
	Since $K_1(\omega)$ and $K_2(\omega)$ converge to $0$ as $\omega\to +\infty$, the following holds if $\omega $ is chosen sufficiently large:
	\begin{equation*}
		\Vert \xi(t,x)\Vert \leq \left(\frac{\kappa\big(-k^{(2m+2)}(\lambda^*)\big)}{M}\right)^{\frac{1}{1+\beta}} e^{-\frac{\lambda^*}{1+\beta}(x-c^*t)}\Vert\varphi_{\lambda^*}(x)\Vert^{\frac{1}{1+\beta}}
	\end{equation*}
	for all $x-c^*t\geq \left(\frac{\omega}{2}\right)^{\frac{1}{2m+2}}$.
	Combining this inequality with $\varphi_{\lambda^*}(x)\geq \kappa \Vert \varphi_{\lambda^*}(x)\Vert_\infty\mathbf{1}$ and \eqref{eq:conv-f}, we obtain
	\begin{multline*}
		-k^{(2m+2)}(\lambda^*)e^{-\lambda^*(x-c^*t)}(\varphi_{\lambda^*})_i(x)\geq \kappa\big(-k^{(2m+2)}(\lambda^*)\big)e^{-\lambda^*(x-c^*t)}\Vert\varphi_{\lambda^*}(x)\Vert  \\ 
		\geq M\Vert \xi(t,x)\Vert^{1+\beta}\geq \Vert f^-(t,\xi(t,x))- A(x)\xi(t,x)\Vert_\infty \geq \big(A(x)\xi(t,x)\big)_i-f^-_i(x,\xi(t,x)) 
	\end{multline*}
	for all $i\in\{1, \ldots, d\}$ whenever $\xi_j(t,x)>0$ for some $j\in\{1, \ldots, d\}$. 
	Recalling \eqref{eq:derivxi}, we have shown that 
	\begin{equation*}
		\xi_t(t,x)-\mathcal L\xi(t,x)\leq f^-(t,\xi(t,x))
	\end{equation*}
	whenever $\xi_i(t,x)>0$ for some $i\in\{1, \ldots, d\}$. This completes the proof of Lemma \ref{lem:lower-barrier-critical}.
\end{proof}
Now we are ready to construct a lower barrier.
If $c>c^*$, we define $\underline{u}(t,x)$ as
\begin{equation*}
	\underline{u}(t, x)=\max_{n\in\mathbb N}\xi(t, x+nL),
\end{equation*}
where $\xi$ is defined by \eqref{eq:xi},   $\omega>\omega^*$, and the maximum is taken componentwise. It follows immediately from the periodicity of the equation \eqref{eq:gen-syst} that the functions $(t, x)\mapsto \xi(t, x+nL)$ for $n\in\mathbb N$ are lower barriers for \eqref{eq:gen-syst}; therefore $\underline{u}(t, x) $ is also a lower barrier. 
Since $\xi(t,x)>0$ for $x>0$ sufficiently large,  $\underline{u}$ is uniformly positive when $x\to -\infty$.
Moreover, $\xi(t,x)<\overline{u}(t,x)$ and
\begin{equation*}
	\xi(t, x+nL)\leq \overline{u}(t, x+nL)=e^{-\lambda(x+nL-ct)}\varphi_\lambda (x+nL) = e^{-\lambda nL} \overline{u}(t, x)\leq e^{-\lambda L}\overline{u}(t,x)< \overline{u}(t,x),
\end{equation*}
for all $n\in\mathbb N$ and $n\geq 1$, therefore  
\begin{equation}\label{eq:ulu}
	\underline{u}(t, x)<\overline{u}(t,x) \text{ for all }(t, x)\in \mathbb R^2. 
\end{equation}

If $c=c^*$ we need to modify the process slightly. We let $\xi(t, x) $ be defined by \eqref{eq:xi^*} and pick $\lambda<\lambda^*$, so that $\lambda c^*+k(\lambda)>0$. Since the leading term in \eqref{eq:xi^*} is controlled by $(x-c^*t)^{2m+2}e^{-\lambda^*(x-c^*t)}$, there is $s^*\geq 0$ such that 
\begin{equation*}
	\xi(t, x+s^*)< \overline{u}(t,x):=e^{-\lambda(x-c^*t)}\varphi_{\lambda^*}(x) \text{ for all } x\in\mathbb R \text{ and }t\geq 0, 
\end{equation*}
therefore we define 
\begin{equation}\label{eq:ulu*}
	\underline{u}(t, x)=\max_{n\in\mathbb N}\xi(t, x+s^*+nL).
\end{equation}
Reasoning as above, we obtain that $\underline{u}(t,x)$ is a lower barrier and satisfies \eqref{eq:ulu}.

We are now in a position to prove Theorem \ref{thm:TW}.
\begin{proof}[Proof of Theorem \ref{thm:TW}]
	The fact that there exists no traveling wave for $c<c^*$ is a direct consequence of the spreading property (Theorem \ref{thm:lin-det}). In order to construct a traveling wave for $c\geq c^*$, we first deal with the case $c>c^*$ and apply the Schauder fixed-point Theorem to construct the traveling wave. We finally send $c$ to $c^*$ the minimal speed in order to construct the minimal speed traveling wave.

	Let us select $c> c^*$, and let $\lambda$ be the smallest positive solution to $\lambda c=-k(\lambda)$ (which exists by Proposition \ref{prop:minspeed}). We let $\overline{u}(t,x) $ be the function defined in Lemma \ref{lem:upper-barrier} and $\underline{u}(t,x)$ be the function defined in equation \eqref{eq:ulu}.

	For $M>0$, we define the (convex) space
	\begin{equation*}
		E_M:=\left\{\vect{v}\in BUC([-M, +\infty), \mathbb R^d)\,|\, \underline{u}(0, x)\leq \vect{v}(x)\leq \overline{u}(0,x)\right\},
	\end{equation*}
	and the operator $Q^M:E_M\to BUC([-M, +\infty), \mathbb R^d)$ by $Q^M(\vect{v})(x)=\tilde{\vect{v}}(x+L)$, where $\tilde v$ is the solution at time $t=\frac{L}{c}$ to 
	\begin{equation*}
		\begin{system}
			\relax&
			\tilde{\vect{v}}_t-\mathcal L\tilde{\vect{v}} = f(x, \tilde{\vect{v}}), & x&\in\mathbb R, \\
			\relax&\tilde{\vect{v}}(t=0, x)=\begin{cases} \vect{v}(x) & \text{if }x\geq -M \\ \max\Big(\min\big(\vect{v}(-M), \overline{u}(t,x)\big), \underline{u}(t,x)\Big) &\text{if } x\leq -M. \end{cases}
		\end{system}
	\end{equation*}
	Then, it follows from Lemma \ref{lem:upper-barrier} that $Q^M(\vect{v}) \leq \overline{u}(t+\frac{L}{c}, x+L)=\overline{u}(t,x)$ for each $\vect{v}\in E_M$, and from Lemma \ref{lem:lower-barrier} and Theorem \ref{thm:super-monotone} (recall that $\underline{u}(t, x)\leq \eta$) that  $Q^M(\vect{v})\geq \underline{u}$ for each $v\in E_M$. Thus $E_M$ is left stable by $Q^M$. Moreover, by the regularizing properties of parabolic operators, $Q^M$ is compact. Thus, the Schauder fixed-point Theorem implies the existence of a fixed-point $\vect{u}^M\in E_M$ such that $Q^M(\vect{u}^M)=\vect{u}^M$. By the classical elliptic regularity, there exists a sequence $M_n\to+\infty$ such that $u^{M_n} $ converges locally uniformly to a solution $\vect{u}^\infty$ to $ Q^\infty(\vect{u^\infty})=\vect{u}^\infty$, and which belongs to 
	\begin{equation*}
		E_\infty:=\left\{\vect{v}\in BUC((-\infty, +\infty), \mathbb R^d)\,|\, \underline{u}(0, x)\leq \vect{v}(x)\leq \overline{u}(0,x)\right\}.
	\end{equation*}

	$u_c:=\vect{u}^\infty$ is  the expected traveling wave. 

	If $c=c^*$ we can repeat the same procedure, but  by replacing the above $\overline{u}(t,x)$ by $\overline{u}(t,x):=e^{-\lambda(x-c^*t)}\varphi_\lambda(x)$ for some $\lambda\in (0, \lambda^*)$ (where $\lambda^*$ is the unique solution of $\lambda^*c^* + k(\lambda^*)=0$)  and $\underline{u}(t,x)$ by \eqref{eq:ulu*}. This leads to the existence of the minimal speed traveling wave $u_{c^*}(t,x)$, which then satisfies 
	\begin{equation*}
		\underline{u}(t,x)\leq u_{c^*}(t,x)\leq \overline{u}(t,x).
	\end{equation*}
	The theorem is proved.
\end{proof}

\begin{rem}[Exponential behavior of traveling waves]  
	Since the function $\omega e^{-\mu(x-ct)}\varphi_\mu(x)$ in the definition of $\xi(t,x)$ above is dominated by the term $e^{-\lambda(x-ct)}\varphi_\lambda(x)$ as $x\to +\infty$, we have $\overline{u}(t,x)\approx \underline{u}(t,x)$ for large $x$, in the sense that $\underline{u}_i/\overline{u}_i\to 1$ as $x\to+\infty$. Consequently, for each $c>c^*$, the traveling wave $u_c$ which we constructed above satisfies $u_c(t,x)\approx \overline{u}(t,x)$ for large $x$.  This implies, in particular, 
	\begin{equation*}
		0<\liminf_{x\to+\infty}\min_{1\leq i\leq d}\frac{u_i(t,x)}{e^{-\lambda_c x}}\leq \limsup_{x\to +\infty}\max_{1\leq i\leq d} \frac{u_i(t, x)}{e^{-\lambda_c x}}<+\infty, 
	\end{equation*}
	where $\lambda_c>0 $ is the minimal root of $\lambda c+k(\lambda)=0$. Thus the asymptotic profile of the traveling wave $u_{c}$ along the leading edge is well approximated by that of $\overline{u}$, which is a solution of the linearized equation around $u=0$. However, from the analogy of the scalar KPP type equations (see, e.g., \cite{Ham-08}) it is likely that the minimal speed traveling wave $u_{c^*}$ does not have the same asymptotics; more precisely we suspect that $u_{c^*}(t,x)\ne {\mathcal O}(e^{-\lambda_c x})$ as $x\to+\infty$ because of the degeneracy of the characteristic equation $\lambda c+k(\lambda)=0$ for $c=c^*$. 
\end{rem}

\subsection{Hair-Trigger effect} 

In this Section we prove the hair-trigger effect (Theorem \ref{thm:hair-trigger}) when the Dirichlet principal eigenvalue is negative, $\lambda_1^{\infty}<0$.
\begin{proof}[Proof of Theorem \ref{thm:hair-trigger}]
	First we note that, by Proposition \ref{prop:sub-sub-barrier}, we may assume without loss of generality that the constant function $x\mapsto \eta\mathbf{1}$ is a super-solution of the equation $-\mathcal Lu\geq f^-(x, u)$.  
	In the following proof we will work under this assumption. 

	Let $R>0 $ be sufficiently large, so that the Dirichlet principal eigenvalue $\lambda_1^R$ is negative (recall the definition of $(\lambda_1^R, \varphi^R(x))$ in Definition \ref{def:princ-eig}) and let $\varphi^R$ be the associated principal eigenfunction, normalized with $\Vert \varphi^R\Vert_{L^\infty(-R, R)^d}=1$. Define 
	\begin{equation*}
		\kappa:=\inf_{x\in (-R, R)}\min_{1\leq i\leq d}\frac{\varphi^R_i(x)}{\Vert\varphi(x)\Vert_{\infty}}, 
	\end{equation*}
	which is finite and positive by the elliptic strong maximum principle and Hopf's Lemma. Then, because of the differentiability of $u\mapsto f^-(x, u)$, there exists $\varepsilon_0>0$ such that for each $u\geq 0$ with $\Vert u\Vert\leq \varepsilon_0$, we have
	\begin{equation*}
		\Vert f^-(x, u)-A(x)u\Vert_{\infty} \leq -\lambda_1^R\kappa\Vert u\Vert_{\infty}. 
	\end{equation*}
	Reducing $\varepsilon_0$ if necessary so that $\varepsilon_0\leq \eta$ (where $\eta$ is as in Definition \ref{def:super-monotone}), this shows that, for $0<\varepsilon\leq \varepsilon_0$, $\varepsilon\varphi^R(x)$ is a lower barrier for \eqref{eq:gen-syst}. Indeed, 
	\begin{align*}
		-\mathcal L\varepsilon\varphi^R(x) &= A(x)\varepsilon \varphi^R(x)+\lambda_1^R\varepsilon\varphi^R(x)\leq A(x)\varphi^R(x)+\lambda_1^R\kappa \Vert\varepsilon \varphi^R\Vert_{\infty} \\
		&\leq A(x)\varepsilon\varphi^R + f^-\big(x, \varepsilon\varphi^R(x)\big)-A(x)\varepsilon\varphi^R(x)=f^-\big(x, \varepsilon\varphi^R(x)\big).
	\end{align*}
	Let $\underline{u}^{R, \varepsilon}(t,x)$ be the solution to the initial-value problem 
	\begin{equation*}
		\left\{\begin{aligned}\relax
			& \underline{u}^{R, \varepsilon}_t-\mathcal L\underline{u}^{R, \varepsilon}=f^-(x, \underline{u}^{R, \varepsilon}), \\
			& \underline{u}^{R, \varepsilon}(0, x)=\underline{u}^{R, \varepsilon}_0(x), 
		\end{aligned}\right. 
	\end{equation*}
	where $\underline{u}^{R, \varepsilon}_0(x)=\varepsilon\varphi^R(x)$ if $x\in(-R, R)$, and $\underline{u}^{R, \varepsilon}_0(x)=0$ otherwise. It follows from the parabolic strong maximum principle that $\underline{u}^{R, \varepsilon}(t,x)>0$ for all $t>0$ and $x\in\mathbb R$, so that in particular $\underline{u}^{R, \varepsilon}(t,\pm R)>0$. Then, we deduce from the parabolic comparison principle that 
	\begin{equation}\label{eq:HTE-super}
		\underline{u}^{R, \varepsilon}(t,x) > \varepsilon\varphi^R(x), \text{ for all } t>0 \text{ and } x\in\mathbb R.
	\end{equation}
	
	Next, fix $\tau>0$. Then it follows from \eqref{eq:HTE-super} that $\underline{u}^{R, \varepsilon}(\tau, x)> \varepsilon\varphi^R(x)=\underline{u}^{R, \varepsilon}_0(x)$. We deduce from the parabolic comparison principle that 
	\begin{equation*}
		\underline{u}^{R, \varepsilon}(t+\tau ,x)>\underline{u}^{R, \varepsilon}(t,x), 
	\end{equation*}
	in other words, $\underline{u}^{R, \varepsilon}(t,x) $ is strictly increasing in time. Thus the limit
	\begin{equation*}
		V^{R, \varepsilon}:=\lim_{t\to+\infty} \underline{u}^{R, \varepsilon}(t,x)
	\end{equation*}
	exists and is an equilibrium of the equation involving $f^-$. It is not difficult to show, by using Serrin's sweeping method, that 
	\begin{equation*}
		V^{R, \varepsilon}(x)\geq \varepsilon_0\varphi^R(x).
	\end{equation*}
	Indeed, define
	\begin{equation*}
		\varepsilon_1:=\sup\{\varepsilon'\geq 0\, |\, \varepsilon'\varphi^R(x)\leq V^R(x)\}.
	\end{equation*}
	Then clearly $\varepsilon_1\geq \varepsilon$ (by the parabolic comparison principle). If $\varepsilon_1<\varepsilon_0$, then there exists a contact point $x_0$ such that $\varepsilon_1\varphi^R(x_0)\leq V^{R, \varepsilon}(x_0)$ and $\varepsilon_1\varphi^R_i(x_0)=V^{R, \varepsilon}_i(x_0)$ for some $i\in\{1, \ldots, d\}$. We find a contradiction by applying the elliptic strong maximum principle in the $i$-th equation of the system. 

	Let $u_0$  be any nontrivial initial data and fix  $t_0 > 0$ and $u(t,x)$ be the solution to \eqref{eq:gen-syst} satisfying $u(0, x)=u_0(x)$. Then $u(t_0,x)>0$ for all $x\in\mathbb R$.  Therefore, for any  $k\in\mathbb Z$, we can find $\varepsilon_k\in(0,\varepsilon_0]$ such that
	\begin{equation*}
		u(t_0,x) \geq \varepsilon_k\varphi^R(x+kL).
	\end{equation*}
	Now we compare $u(t,x)$ and the solution $\underline{u}(t,x)$ with initial data $ \underline{u}_0(x):=\varepsilon_k\varphi^R(x+kL)$ inside $(-kL-R, -kL+R)$  and $0$ outside.
Then by the result (4), we see that
	\begin{equation*}
		\liminf_{t\to\infty} u(t,x) \geq \varepsilon_0\varphi^R(x+kL) \text{ in }(-kL-R, -kL+R)
	\end{equation*}
	for any  $k\in\mathbb Z$.  This implies that there exists  $\delta>0$ (independent of $u_0$) such that  
	\begin{equation*}
		\liminf_{t\to\infty} u(t,x) \geq \delta\mathbf{1}
	\end{equation*}
	for all $x\in\mathbb R$.
\end{proof}

{
\section{Singular limits}
\label{sec:singular-limits}

\subsection{Spreading speed for rapidly oscillating coefficients}
\label{sec:large-diff}
}

In this section we prove Theorem \ref{thm:large-diff}. 

\paragraph{Formal computations to get the formula for the speed.} Here we present the classical computations that allow to retrieve the correct result, though without the correct mathematical  justification. The basic idea is to apply known results from homogenization theory to the eigenvalue problem involved in the definition of the minimal speed \eqref{eq:defk(lambda)}, i.e. 
\begin{equation}\label{eq:hom-eigen}
	-L_\lambda^\varepsilon\vect{\varphi}=-(\vect{\sigma}^\varepsilon(x)\vect{\varphi}_x)_x+(2\lambda\vect{\sigma}^\varepsilon(x)-\vect{q}^\varepsilon(x))\vect{\varphi}_x+(\lambda \vect{\sigma}^\varepsilon_x(x)+\lambda \vect{q}^\varepsilon(x)-\lambda^2\vect{\sigma}^\varepsilon(x)-A^\varepsilon(x))\vect{\varphi}=k^\varepsilon(\lambda)\vect{\varphi}.
\end{equation}
We follow the approach of Bensoussan, Lions and Papanicolaou \cite{Ben-Lio-Pap-11} and introduce an asymptotic expansion  in $\varepsilon$:
\begin{equation}\label{eq:formal-devt}
	\vect{\varphi}(x)=\vect{\phi}\left(x, \frac{x}{\varepsilon}\right)=\vect{\phi}^0\left(x, \frac{x}{\varepsilon}\right)+\varepsilon \vect{\phi}^1\left(x, \frac{x}{\varepsilon}\right)+\varepsilon^2\vect{\phi}^2\left(x, \frac{x}{\varepsilon}\right)+\cdots
\end{equation}
where the functions $\phi(x, y)$,  $\phi^0(x y)$, $\phi^1(x, y)$ and $\phi^2(x, y)$ and are $1$-periodic in $y$. We substitute \eqref{eq:formal-devt} into  \eqref{eq:hom-eigen} and rewrite it in terms of the variables $x$ and $y$.
\begin{multline}\label{eq:hom-formal}
	\varepsilon^{-2}\left[-(\vect{\sigma}(y)\vect{\phi}_y)_y\right]+\varepsilon^{-1}\left[-(\vect{\sigma}(y)\vect{\phi}_x)_y-(\vect{\sigma}(y)\vect{\phi}_y)_x+(2\lambda \vect{\sigma}(y)-\vect{q}(y))\vect{\phi}_y + \lambda\vect{\sigma}_y(y)\vect{\phi} \right] \\
	+\varepsilon^0\left[-(\vect{\sigma}(y)\vect{\phi}_x)_x+(2\lambda \vect{\sigma}(y)-\vect{q}(y))\vect{\phi}_x+(\lambda\vect{q}(y)-\lambda^2\vect{\sigma}(y)-A(y)-k^\varepsilon(\lambda){I})\vect{\phi} \right]\\
	\begin{aligned}
		\relax&=\varepsilon^{-2}\left[-(\vect{\sigma}(y)\vect{\phi}^0_y)_y\right]\\ 
		&\quad + \varepsilon^{-1}\left[-(\vect{\sigma}(y)\vect{\phi}^1_y)_y-(\vect{\sigma}(y)\vect{\phi}^0_y)_x-(\vect{\sigma}(y)\vect{\phi}^0_x)_y+(2\lambda\vect{\sigma}(y)-\vect{q}(y))\vect{\phi}^0_y +\lambda \vect{\sigma}_y(y)\vect{\phi}^0 \right]\\
		&\quad +\varepsilon^0\left[ -(\vect{\sigma}(y)\vect{\phi}^2_y)_y-(\vect{\sigma}(y)\vect{\phi}^1_y)_x-(\vect{\sigma}(y)\vect{\phi}^1_x)_y+(2\lambda\vect{\sigma}(y)-\vect{q}(y))\vect{\phi}^1_y +\lambda \vect{\sigma}_y(y)\vect{\phi}^1\right.\\ 
		&\quad \quad\left.-(\vect{\sigma}(y)\vect{\phi}^0_x)_x+(2\lambda \vect{\sigma}(y)-\vect{q}(y))\vect{\phi}^0_x+(\lambda\vect{q}(y)-\lambda^2\vect{\sigma}(y) -A(y)-k^\varepsilon(\lambda)I)\vect{\phi}^0 \right]\\
		&\quad+\mathcal O(\varepsilon)\\
		&=\vect{0}. 
	\end{aligned}
\end{multline}
In \eqref{eq:hom-formal}, the coefficients of $\varepsilon^{-2}$ and $\varepsilon^{-1}$ must be zero.
In particular, we have:
\begin{equation*}
	-(\vect{\sigma}(y)\vect{\phi}^0_y(x,y))_y=0, 
\end{equation*}
which yields $\vect{\phi}^0_y(x,y)=\vect{\sigma}^{-1}(y)\tilde{\vect{\phi}}^0(x)$, however since $\vect{\phi}$ is $1$-periodic in $y$ we have $\int_0^1\vect{\phi}^0_y(x,y)\dd y=\int_0^1\vect{\sigma}^{-1}(y)\dd y\tilde{\vect{\phi}}^0(x) =\vect{0}$ and therefore $\tilde{\vect{\phi}}^0(x)=\vect{0} $ for all $x\in\mathbb R$. Thus, integrating again, we get:
\begin{equation}\label{eq:hom-phi0}
	\vect{\phi}^0(x, y)=\vect{\varphi}(x), \text{ for all } x\in\mathbb R.
\end{equation}

Next we focus on the coefficient in $\varepsilon^{-1}$ term in \eqref{eq:hom-formal}. Using \eqref{eq:hom-phi0}, we rewrite the $\varepsilon^{-1}$ coefficient as:
\begin{equation*}
	-(\vect{\sigma}(y)\vect{\phi}^1_y)_y=\vect{\sigma}_y(y)(\vect{\varphi}_x(x)-\lambda\vect{\varphi}(x)).
\end{equation*}
We remark that $\vect{\chi}(y):=\int_0^y\left(\int_0^1\vect{\sigma}^{-1}(z')\dd z'\right)^{-1}\vect{\sigma}^{-1}(z)-\vect{1}\dd z$ is a particular solution to:
\begin{equation*}
	-(\vect{\sigma}(y)\vect{\chi}_y(y))_y=-\vect{\sigma}_y(y).
\end{equation*}
Therefore we can write:
\begin{equation}\label{eq:hom-phi1}
	\vect{\phi}^1(x,y)=\vect{\chi}(y)(\vect{\varphi}_x(x)-\lambda\vect{\varphi}(x))+\tilde{\vect{\phi}}^1(x).
\end{equation}

Last, integrating the coefficient in $\varepsilon^0$ in \eqref{eq:hom-formal} gives us the homogenization limit for \eqref{eq:hom-eigen}. We get: 
\begin{align*}
	&0-\overline{\vect{\sigma}\vect{\chi}_y}(\vect{\varphi}_{xx}-\lambda \vect{\varphi}_x)-0+(2\lambda \overline{\vect{\sigma}\vect{\chi}_y}-\overline{\vect{q}\vect{\chi}_y})(\vect{\varphi}_{x}-\lambda \vect{\varphi})+\lambda \overline{\vect{\sigma}_y\vect{\chi}}(\vect{\varphi}_{x}-\lambda \vect{\varphi}) \\ 
	&-(\overline{\vect{\sigma}}\vect{\varphi}_x)_x+(2\lambda \overline{\vect{\sigma}}-\overline{\vect{q}})\vect{\varphi}_x+(\lambda\overline{\vect{q}}-\lambda^2\overline{\vect{\sigma}}-\overline{A}-k^\varepsilon(\lambda)I)\vect{\varphi}=\vect{0},
\end{align*}
where $\overline{u}$ denotes the average of a function $u$ over one period.
Using the fact that 
\begin{equation*}
	\vect{\chi}_y(y)=\left(\int_0^1\vect{\sigma}^{-1}(z)\dd z\right)^{-1}\vect{\sigma}^{-1}(y)-\vect{1}
\end{equation*}
and integrating by parts, we get:
\begin{equation*}
	-\overline{\vect{\sigma}^{-1}}^{-1}\vect{\varphi}_{xx} + \left(2\lambda\overline{\vect{\sigma}^{-1}}^{-1}-\overline{\vect{\sigma}^{-1}}^{-1}\overline{\vect{\sigma}^{-1}\vect{q}}\right)\vect{\varphi}_x + \left(\lambda\overline{\vect{\sigma}^{-1}}^{-1}\overline{\vect{\sigma}^{-1}\vect{q}}-\lambda^2\overline{\vect{\sigma}^{-1}}^{-1}-\overline{A}-k^0(\lambda)I\right)\vect{\varphi}=\vect{0}.
\end{equation*}
By the uniqueness of the periodic principal eigenvalue, $\varphi$ is equal to the Perron-Frobenius eigenvector of the matrix $\lambda\overline{\vect{\sigma}^{-1}}^{-1}\overline{\vect{\sigma}^{-1}\vect{q}}-\lambda^2\overline{\vect{\sigma}^{-1}}^{-1}-\overline{A}-k^0(\lambda)I$.  
We retrieve \eqref{eq:lim-hom} indeed.

\begin{proof}[Proof of Theorem \ref{thm:large-diff}]
	We divide the proof in two steps.

	\begin{stepping}
		\step We show that $k^\varepsilon(\lambda)$ converges locally uniformly to 
		\begin{equation*}
			k^0(\lambda):=\lambda_{PF}\big(\lambda\overline{\vect{q}}^H-\lambda^2\overline{\vect{\sigma}}^H-\overline{A}\big)
		\end{equation*}
		in $(0, +\infty)$, where we recall that $\overline{\vect{\sigma}}^H:=\overline{\vect{\sigma}^{-1}}^{-1}$, $\overline{\vect{q}}^H:=\overline{\vect{\sigma}^{-1}}^{-1}\overline{\vect{\sigma}^{-1}\vect{q}}$, 
		\begin{equation}\label{eq:homogenisation-eigen}
			L^\varepsilon_\lambda\vect{\phi}^\varepsilon_\lambda-A^{\varepsilon}(x)\vect{\phi}^{\varepsilon}_\lambda=k^\varepsilon(\lambda)\vect{\phi}^\varepsilon_\lambda
		\end{equation}
		and $L^\varepsilon_\lambda\vect{\varphi}:= e^{\lambda x}\mathcal L^\varepsilon(e^{-\lambda x}\vect{\varphi})$ for all $\vect{\varphi}$. \medskip

		We argue by contradiction and assume that there exists a bounded interval $\left[\frac{1}{R}, R\right]$, $\delta>0$ and  $\varepsilon_n>0$ such that $\sup_{\frac{1}{R}\leq\lambda\leq R}|k^{\varepsilon_n}(\lambda)- k^0(\lambda)|\geq \delta$. Since $\left[\frac{1}{R}, R\right] $ is bounded, for each $n\in\mathbb N$ there exists $\lambda^{\varepsilon_n}\in \left[\frac{1}{R}, R\right]$ such that $\sup_{\frac{1}{R}\leq\lambda\leq R}|k^{\varepsilon_n}(\lambda)- k^0(\lambda)|=|k^{\varepsilon_n}(\lambda^{\varepsilon_n})- k^0(\lambda^{\varepsilon_n})|$. Up to the extraction of a subsequence we  assume that $\lambda^{\varepsilon}\to\lambda_0\in \left[\frac{1}{R}, R\right]$. For simplicity in the rest of the proof we will omit the subscript $n$ and write $\varepsilon$ instead of $\varepsilon_n$.

		Let $\vect{\phi}^\varepsilon:=\vect{\phi}^{\varepsilon}_{\lambda^\varepsilon}(x)>0 $ be a sequence of solutions to $L^{\varepsilon}_{\lambda^\varepsilon}\vect{\phi}^\varepsilon=k^{\varepsilon}(\lambda)\vect{\phi}^\varepsilon$ with $\varepsilon\to 0$ and $\lambda^\varepsilon\to \lambda_0$,  which satisfies $\Vert \vect{\phi}^\varepsilon\Vert_{L^2(0,1)^d}^2=\int_0^1\sum_{i=1}^d(\phi^\varepsilon_i)^2(x)dx=1$. Testing \eqref{eq:homogenisation-eigen} at a maximum and minimum point of $\vect{\phi}^\varepsilon$, respectively, we find that 
		\begin{equation*}
			-(\lambda^\varepsilon)^2-\underset{x\in \mathbb R}{\underset{1\leq i\leq d}{\sup}}\,\underset{j=1}{\overset{d}{\sum}}a_{ij}(x) \leq k^{\varepsilon}(\lambda^\varepsilon)\leq -(\lambda^\varepsilon)^2-\underset{x\in \mathbb R}{\underset{1\leq i\leq d}{\inf}}\,\underset{j=1}{\overset{d}{\sum}}a_{ij}(x) .
		\end{equation*}
		In particular, $k^{\varepsilon}(\lambda^\varepsilon)$ is uniformly bounded in $n$ and we may extract a subsequence such that $k^{\varepsilon}(\lambda^\varepsilon)\to k^0$. 

		Let us show that $\vect{\phi}^{\varepsilon}$ is uniformly bounded in  $H^1(0,1)^d$. Indeed, we have 
		\begin{align*}
			\underline{\sigma}\Vert \vect{\phi}^{\varepsilon}_x\Vert_{L^2(0,\varepsilon)^d}^2&\leq \sum_{i=1}^d \int_0^\varepsilon \sigma_i\left(\varepsilon^{-1}x\right)(\partial_x{\phi}^{\varepsilon}_i(x))^2\dd x \\
			&=\sum_{i=1}^d \int_0^\varepsilon q_i^\varepsilon(x)(\partial_x\phi^\varepsilon_i(x))\phi^\varepsilon_i(x)\dd x+\int_0^\varepsilon(\lambda^\varepsilon q^\varepsilon_i(x)+(\lambda^\varepsilon)^2\sigma^\varepsilon_i(x) + k^\varepsilon(\lambda_\varepsilon))\phi^\varepsilon_i(x)^2\dd x\\
			&\quad + \sum_{i=1}^d\sum_{j=1}^d \int_0^\varepsilon a^\varepsilon_{ij}(x)\phi^\varepsilon_i(x)\phi^\varepsilon_j(x)\dd x \\
			&\leq \Vert \vect{q}\Vert_{L^\infty} \Vert \vect{\phi}^{\varepsilon}_x\Vert_{L^2(0,\varepsilon)^d}\Vert\vect{\phi}^{\varepsilon}\Vert_{L^2(0,\varepsilon)^d}  \\ 
			&\quad + C\left(\lambda^\varepsilon\Vert \vect{q}\Vert_{L^\infty}+(\lambda^\varepsilon)^2\Vert\vect{\sigma}\Vert_{L^\infty}+k^\varepsilon(\lambda^\varepsilon)+\Vert A\Vert_{L^\infty}\right)\Vert \vect{\phi}^{\varepsilon}\Vert_{L^2(0,\varepsilon)^d}^2, 
		\end{align*}
		and by periodicity
		\begin{equation*}
			\lfloor\varepsilon^{-1}\rfloor\Vert \phi_x^\varepsilon\Vert_{L^2(0,\varepsilon)^d}\leq \Vert \phi_x^\varepsilon\Vert_{L^2(0,1)^d}\leq (\lfloor\varepsilon^{-1}\rfloor+1)\Vert \phi_x^\varepsilon\Vert_{L^2(0,\varepsilon)^d},
		\end{equation*}
		where $\lfloor\varepsilon^{-1}\rfloor$ is the lower integer part of $\varepsilon^{-1}$, and therefore
		\begin{equation*}
			\underline{\sigma}\Vert \vect{\phi}^{\varepsilon}_x\Vert_{L^2(0,1)^d}^2\leq \frac{\lfloor\varepsilon^{-1}\rfloor+1}{\lfloor\varepsilon^{-1}\rfloor}C\Vert \phi^{\varepsilon}\Vert_{L^2(0,1)^d}^2
		\end{equation*}
		where $C$ is independent of $\varepsilon$ and  $\Vert \vect{\phi}^{\varepsilon}\Vert_{L^2(0,1)^d}=1$. Therefore, (up to the extraction of a subsequence) there is $\vect{\phi}\in H^1(0,1)^d$ such that $\vect{\phi}^{\varepsilon}\to \vect{\phi} $ {\em strongly} in $L^2(0,1)^d$ and $\vect{\phi}^{\varepsilon}\rightharpoonup \vect{\phi}$ {\em weakly} in $H^1(0,1)^d$. Next we remark that, rewriting \eqref{eq:hom-eigen} as:
		\begin{equation}\label{eq:hom-xi}
			-(\vect{\sigma}^\varepsilon(\vect{\phi}^\varepsilon_x-\lambda^\varepsilon\vect{\phi}^\varepsilon))_x + (\lambda^\varepsilon\vect{\sigma}^\varepsilon-\vect{q}^\varepsilon)(\vect{\phi}^\varepsilon_x-\lambda^\varepsilon\vect{\phi}^\varepsilon)=A^\varepsilon\vect{\phi}^\varepsilon+k^\varepsilon(\lambda^\varepsilon)\vect{\phi}^\varepsilon,
		\end{equation}
		the function $\vect{\xi}^\varepsilon:=\vect{\sigma}^\varepsilon\vect{\phi}^{\varepsilon}_x-\lambda^\varepsilon\phi^\varepsilon$ is uniformly bounded in $H^1(0,1)^d$. Indeed multiplying \eqref{eq:hom-xi} by $\xi^\varepsilon_x$ and integrating, we get 
		\begin{multline*}
			\int_0^\varepsilon\sigma^\varepsilon \big|\xi_x^\varepsilon\big|^2 = \int_0^\varepsilon\big[(q^\varepsilon-\lambda^\varepsilon\sigma^\varepsilon)(\phi_x^\varepsilon-\lambda^\varepsilon\phi^\varepsilon)\big]\cdot \xi_x^\varepsilon +\int_0^\varepsilon \Big[\big(A^\varepsilon+k^\varepsilon(\lambda^\varepsilon)I\big)\phi^\varepsilon\Big]\cdot\xi^\varepsilon_x   \\
			\leq \Vert q^\varepsilon-\lambda^\varepsilon\sigma^\varepsilon\Vert_{L^\infty(0,\varepsilon)^d}\Vert \phi^\varepsilon_x-\lambda^\varepsilon\phi^\varepsilon\Vert_{L^2(0,\varepsilon)^d}\Vert\xi_x^\varepsilon\Vert_{L^2(0,\varepsilon)^d}+\Vert A^\varepsilon+k^\varepsilon(\lambda^\varepsilon)I\Vert_{L^\infty(0,\varepsilon)^d}\Vert\phi^\varepsilon\Vert_{L^2(0,\varepsilon)^d}\Vert \xi_x\Vert_{L^2(0, \varepsilon)^d}.
		\end{multline*}
		By the periodicity of $\xi^\varepsilon$ we easily conclude that $\xi^\varepsilon_x$ is uniformly bounded in $L^2(0,1)^d$.
		Therefore (up to the extraction of a subsequence), there is $\vect{\xi}\in H^1(0,1)^d$ such that $\vect{\xi}^\varepsilon\to\vect{\xi}$ strongly in $L^2(0,1)^d $ and $\vect{\xi}^\varepsilon\rightharpoonup\vect{\xi}$ weakly in $H^1(0,1)^d$. In particular $\vect{\phi}^\varepsilon_x-\lambda^\varepsilon\vect{\phi}^\varepsilon=\vect{\sigma}^{-1}\vect{\xi}^\varepsilon\rightharpoonup \overline{\vect{\sigma}^{-1}}\vect{\xi}$ in $L^2(0,1)^d$ weakly and therefore $\vect{\xi}=\overline{\vect{\sigma}}^H(\vect{\phi}_x-\lambda_0\vect{\phi})$. This allows us to determine the limits of each term in \eqref{eq:hom-xi} by using the convergence of $\vect{\xi}^\varepsilon$:
		\begin{align*}
			\vect{\sigma}^\varepsilon(\vect{\phi}^\varepsilon_x-\lambda^\varepsilon\vect{\phi}^\varepsilon)=\vect{\xi}^\varepsilon&\to \vect{\xi}=\overline{\vect{\sigma}}^H(\vect{\phi}_x-\lambda_0\vect{\phi}),& \text{ in }&L^2(0,1)^d \text{ strong}, \\
			(\vect{\sigma}^\varepsilon(\vect{\phi}^\varepsilon_x-\lambda^\varepsilon\vect{\phi}^\varepsilon))_x&\rightharpoonup (\overline{\vect{\sigma}}^H(\vect{\phi}_x-\lambda_0\vect{\phi}))_x,& \text{ in }&H^1(0,1)^d \text{ weak}, \\
			\vect{q}^\varepsilon(\vect{\phi}^\varepsilon_x-\lambda^\varepsilon\vect{\phi}^\varepsilon)=(\vect{\sigma}^\varepsilon)^{-1}\vect{q}^\varepsilon\vect{\xi}^\varepsilon&\rightharpoonup \overline{\vect{\sigma}^{-1}\vect{q}}\vect{\xi}=\overline{\vect{q}}^H(\vect{\phi}_x-\lambda_0\vect{\phi}),& \text{ in }&L^2(0,1)^d \text{ weak}, \\
		\end{align*}
		and \eqref{eq:hom-xi} becomes:
		\begin{equation*}
			-\overline{\vect{\sigma}}^H((\vect{\phi}_x-\lambda_0\vect{\phi}))_x + (\lambda_0\overline{\vect{\sigma}}^H-\overline{\vect{q}}^H)(\vect{\phi}_x-\lambda_0\vect{\phi})=\overline{A}\vect{\phi}+k^0\vect{\phi}.
		\end{equation*}
		Note that because of the periodicity of $\phi^\varepsilon$, the convergence in $L^2(0,1)^d$ or $H^1(0, 1)^d$, weak or strong, implies the same convergence for the $L^2$ or $H^1$ local uniform topology on $\mathbb R$. 
		Thus $\phi$ satisfies \eqref{eq:hom-eigen} with $\vect{\sigma}^\varepsilon$ replaced by $\overline{\vect{\sigma}}^H$ and $\vect{q}^\varepsilon$ replaced by $\overline{\vect{q}}^H$. By the uniqueness of the principal eigenvector, we find that $\vect{\phi}$ is a positive constant vector satisfying: 
		\begin{equation*}
			(-\lambda_0^2\overline{\vect{\sigma}}^H+\lambda_0\overline{\vect{q}}^H-\overline{A})\vect{\phi}=k^0\vect{\phi},
		\end{equation*}
		therefore $k^0=\lambda_{PF}(-\lambda_0^2\overline{\vect{\sigma}}^H+\lambda_0\overline{\vect{q}}^H-\overline{A})=k^0(\lambda_0)$. This is a contradiction.\medskip

		\step We show that the minimum of $\frac{k^\varepsilon(\lambda)}{\lambda}$ converges to the one of $\frac{k^0(\lambda)}{\lambda}$.\medskip

		It is well-known, in the scalar matrix case, that $\lambda\mapsto \lambda_{PF}\big(\lambda^2 \langle {\sigma}\rangle_{A} + \langle A\rangle_A\big)=k^0(\lambda)$ is a strictly concave function, and that $ \frac{-k^0(\lambda)}{\lambda} $ has a unique minimum for $\lambda>0$. We have extended this property to the case of systems in Proposition \ref{prop:minspeed}. From the local uniform convergence of  $k^\varepsilon( \lambda) $  to $k^0(\lambda)$, we conclude the local uniform convergence of  $-\frac{k^\varepsilon(\lambda)}{\lambda} $ to $\frac{k^0(\lambda)}{\lambda} $ and $\lambda^*_\varepsilon:=\arg\min\left(-\frac{k^\varepsilon(\lambda)}{\lambda}\right)$ to $\lambda^*_0:=\arg\min\left(-\frac{k^0(\lambda)}{\lambda}\right)$. This finishes the proof of Theorem \ref{thm:large-diff}.
	\end{stepping}
\end{proof}

\subsection{Strong coupling}
\label{sec:strong-coupling}
{
In this subsection we study the singular limit of the following system,
which is a modified version of system \eqref{eq:main-sys} with strong coupling
\begin{equation}\label{eq:strong-coupling}
	\begin{system}
	    \relax &u_t=(\sigma_u(x) u_{x})_x+\big(r_u(x)-\kappa_u(x)(u+v)\big)u -\frac{1}{\varepsilon}\big(p(x)u - \big(1-p(x)\big)v\big),\\
		\relax &v_t=(\sigma_v(x) v_{x})_x+\big(r_v(x)-\kappa_v(x)(u+v)\big)v+\frac{1}{\varepsilon}\big(p(x)u - \big(1-p(x)\big)v\big),
	\end{system}
\end{equation}
where $p(x)\in(0,1)$ is smooth (at leat $C^2$) and $\varepsilon>0$, with a particular interest in the limit $\varepsilon\to 0$. A formal way to compute the limit is to consider asymptotic expansions of $u$ and $v$
\begin{align*}
	u(t,x)&=u^0(t, x)+\varepsilon u^1(t, x)+\varepsilon^2u^2(t,x)+\cdots , \\
	v(t,x)&=v^0(t, x)+\varepsilon v^1(t, x)+\varepsilon^2v^2(t,x)+\cdots ,
\end{align*}
and this method has the advantage of allowing an arbitrary degree of precision in the asymptotic behavior of the solution when $\varepsilon\to0$. However, since we are presently concerned with the zero-order term only, we present an easier way to compute the limit. We let $P(t,x):=p(x)u(t,x)-\big(1-p(x)\big)v(t,x)$ and remark that, for a limit to exist, one must have $P(t, x)\to 0$ as $\varepsilon\to 0$. Therefore the limit $(u^0, v^0)$ satisfies
\begin{equation*}
    p(x)u^0=\big(1-p(x)\big)v^0
\end{equation*}
and the sum $S:=u^0+v^0$ is the solution of a closed scalar reaction-diffusion equation which can be determined explicitly by the relations $u^0=\big(1-p(x)\big)S$, $v^0=p(x)S$,
\begin{align*}
	u_x &= -p_x S+(1-p)S_x,   \\
	v_x &= p_xS+pS_x,  
\end{align*}
and therefore
\begin{align*}
	S_t&=\big((1-p(x))\sigma_u(x)+p(x)\sigma_v(x)S_x\big)_x+\left((\sigma_v(x)-\sigma_u(x))p_xS\right)_x+\big(r(x)-\kappa(x)S\big)S \\ 
	&=\big(\sigma(x)S_x\big)_x + q(x)S_x + \big(r(x)+q_x(x)-\kappa(x)S\big)S, 
\end{align*}
where 
\begin{align*}
	\sigma(x)&=(1-p(x))\sigma_u(x)+p(x)\sigma_v(x), \\
	r(x)&=(1-p(x))r_u(x)+p(x)r_v(x), \\
	\kappa(x)&=(1-p(x))\kappa_u(x)+p(x)\kappa_v(x) , \\
	q(x)&=(\sigma_v(x)-\sigma_u(x))p_x.
\end{align*}
In particular, $\sigma_u(x)$, $\sigma_v(x)$ and $p(x)$ can be chosen so that the sign of $\sigma_v(x)-\sigma_u(x) $ is the same as the sign of $p_x(x)$, in which case $q(x)>0$ and 
\begin{equation*}
	\int_0^1 \frac{q(x)}{2\sigma(x)}\dd x>0.
\end{equation*}
In this case it is known (see \eqref{eq:orderspeed} in Appendix \ref{app:scalar}) that the leftward and rightward speeds are different. Since there is a strict sign between $c^*_{\text{left}}$ and $c^*_{\text{right}}$, the same holds for the original system \eqref{eq:strong-coupling} with $\varepsilon>0$ sufficiently small.

For the sake of concision, we will not make this entire argument rigorous but focus on the limit of the principal eigenproblem, which implies the convergence of the minimal speed. 
\begin{prop}\label{prop:strong-coupling}
    Let $\lambda\in \mathbb R$ and $\varepsilon >0$ be given. Denote $k^\varepsilon(\lambda)$, $\varphi_\lambda^\varepsilon(x)$, $\psi_\lambda^\varepsilon(x)$ the principal solution to the eigenproblem
    \begin{equation}\label{eq:strong-coupling-eps}
	\left\{\begin{aligned}\relax
	    &-\big(\sigma_u(x)(\varphi^\varepsilon_\lambda)_x\big)_x + 2\lambda(\varphi^\varepsilon_\lambda)_x +\big(\lambda(\sigma_u)_x(x)-\lambda^2\sigma_u(x)-r_u(x)\big)\varphi^\varepsilon_\lambda +\frac{1}{\varepsilon}\big(p(x)\varphi^\varepsilon_\lambda - \big(1-p(x)\big)\psi^\varepsilon_\lambda\big) = k^\varepsilon(\lambda) \varphi^\varepsilon_\lambda, \\ 
	    &-\big(\sigma_v(x)(\psi^\varepsilon_\lambda)_x\big)_x + 2\lambda(\psi^\varepsilon_\lambda)_x +\big(\lambda(\sigma_v)_x(x)-\lambda^2\sigma_v(x)-r_v(x)\big)\psi^\varepsilon_\lambda -\frac{1}{\varepsilon}\big(p(x)\varphi^\varepsilon_\lambda - \big(1-p(x)\big)\psi^\varepsilon_\lambda\big) = k^\varepsilon(\lambda) \psi^\varepsilon_\lambda ,
	\end{aligned}\right.
    \end{equation}
    with $L$-periodic boundary conditions and normalized in $L^2_{per}$.
    Then, as $\varepsilon\to 0$, the function $k^\varepsilon(\lambda)$ converges locally uniformly to $k^0(\lambda)$, the principal eigenvalue of the problem
    \begin{equation*}
	-\big(\sigma(x)(\varphi^0_\lambda)_x\big)_x +\big(2\lambda - q(x)\big)(\varphi^0_\lambda)_x + \big(\lambda\sigma_x(x)+\lambda q(x)-\lambda^2\sigma(x)-r(x)-q_x(x)\big)\varphi^\varepsilon_\lambda  = k^0(\lambda)\varphi^0_\lambda, 
    \end{equation*}
    for a $L$-periodic positive scalar function $\varphi$, where   
    \begin{align*}
	\sigma(x)&=(1-p(x))\sigma_u(x)+p(x)\sigma_v(x), & 
	r(x)&=(1-p(x))r_u(x)+p(x)r_v(x), &
	q(x)&=(\sigma_v(x)-\sigma_u(x))p_x.
    \end{align*}
\end{prop}
\begin{proof}
    We argue by contradiction and assume that there exists a bounded interval $\left[-R, R\right]$, $\delta>0$ and  $\varepsilon_n>0$ such that $\sup_{-R\leq\lambda\leq R}|k^{\varepsilon_n}(\lambda)- k^0(\lambda)|\geq \delta$. Since $\left[-{R}, R\right] $ is bounded, for each $n\in\mathbb N$ there exists $\lambda^{\varepsilon_n}\in \left[-{R}, R\right]$ such that $\sup_{-{R}\leq\lambda\leq R}|k^{\varepsilon_n}(\lambda)- k^0(\lambda)|=|k^{\varepsilon_n}(\lambda^{\varepsilon_n})- k^0(\lambda^{\varepsilon_n})|$. Up to the extraction of a subsequence we  assume that $\lambda^{\varepsilon}\to\lambda_0\in \left[-{R}, R\right]$. For simplicity in the rest of the proof we will omit the subscript $n$ and write $\varepsilon$ instead of $\varepsilon_n$. We will also omit the subscripts and superscripts, when there is no ambiguity, for the solutions $(\varphi, \psi)$ of  \eqref{eq:strong-coupling-eps}.

    Let us show that $\varphi$ and $\psi$ are bounded in $H^1_{per}$ when $\varepsilon\to 0$. Indeed, multiplying the first line of \eqref{eq:strong-coupling-eps} by $p(x)\varphi(x)$, we get
    \begin{align*}
	\int \sigma_u(x)p(x)\varphi_x^2+\int\sigma_u(x)p_x(x)\varphi^2+2\lambda \int p(x)\varphi\varphi_x&=\int\big(-\lambda(\sigma_u)_x(x)+\lambda^2\sigma_u(x)+r_u(x) + k^\varepsilon(\lambda)\big)p(x)\varphi^2 \\ 
	&\quad -\frac{1}{\varepsilon}\left(\int p(x)^2\varphi^2 - \int p(x)\big(1-p(x)\big)\varphi\psi\right), 
    \end{align*}
    and multiplying the second line by $\big(1-p(x)\big)\psi(x)$ we get 
    \begin{multline*}
	\int \sigma_v(x)\big(1-p(x)\big)\psi^2-\int\sigma_v(x)p_x(x)\psi^2+2\lambda \int \big(1-p(x)\big)\psi\psi_x \\ 
	=\int\big(-\lambda(\sigma_v)_x(x)+\lambda^2\sigma_v(x)+r_v(x) + k^\varepsilon(\lambda)\big)\big(1-p(x)\big)\psi^2 \\ 
	\quad +\frac{1}{\varepsilon}\left(\int p(x)\big(1-p(x)\big)\varphi\psi - \int \big(1-p(x)\big)^2\psi^2\right), 
    \end{multline*}
    and finally the sum of the two equations above yields
    \begin{equation*}
	\int p(x)\varphi_x^2 + \int \big(1-p(x)\big)\psi_x^2 \leq C_1 \left(\Vert \varphi\Vert_{L^2_{per}}^2 + \Vert \psi\Vert_{L^2_{per}}^2\right), 
    \end{equation*}
    where $C_1$ is independent of $\varepsilon$. Since $p(x)$ and $(1-p(x))$ are bounded below, we conclude that $\varphi$ and $\psi$ are indeed bounded in $H^1_{per}$ uniformly when $\varepsilon\to 0$. 

    Therefore, up to the extraction of a subsequence, $\varphi$ and $\psi$ converge respectively to $\bar\varphi$ and $\bar\psi$ as $\varepsilon\to 0$, weakly in $H^1_{per}$ and strongly in $L^2_{per}$. Let $f\in C^2_{per}$ be a smooth test function, then multiplying the first line of \eqref{eq:strong-coupling-eps} by  $f$ leads to 
    \begin{align*}
	\int f(x)\left(p(x)\psi-\big(1-p(x)\big)\psi\right) &= \varepsilon\Big[-\int \sigma_u(x)\varphi_xf_x(x) - 2\lambda\int \varphi_xf(x) \\
	&\quad+\int\big(-\lambda(\sigma_u)_x(x)+\lambda^2\sigma_u(x)+r_u(x)+k^\varepsilon(\lambda)\big)\varphi f(x) \Big]\\
	&\xrightarrow[\varepsilon\to 0]{}0, 
    \end{align*}
    which shows that, since $f(x)$ is arbitrary,  
    \begin{equation*}
	p(x)\bar\varphi(x) = \big(1-p(x)\big)\bar\psi(x). 
    \end{equation*}
    By elementary computations, we find that $S(x):=\varphi(x)+\psi(x)$ converges weakly to a function  $\bar S\in H^1_{per}$ which solves  
    \begin{equation*}
	-\big(\sigma(x)(\bar{S})_x\big)_x +\big(2\lambda_0 - q(x)\big)(\bar{S})_x + \big(\lambda_0\sigma_x(x)+\lambda_0 q(x)-\lambda_0^2\sigma(x)-r(x)-q_x(x)\big)\bar S  = \bar k\bar{S}, 
    \end{equation*}
    where $\bar k=\lim k^\varepsilon(\lambda)$ and $\sigma(x)$, $q(x)$, $r(x)$ are as in the statement of the proposition. Therefore $\bar k=k^0(\lambda_0)$, which is a contradiction. Proposition \ref{prop:strong-coupling} is proved. 
\end{proof}
\begin{rem}\label{rem:strong-coupling-speed}
    In particular, with the notations of Proposition \ref{prop:strong-coupling}, assume that 
    \begin{equation*}
	0>\lambda_1^\infty=\max_{\lambda\in\mathbb R}k(\lambda).
    \end{equation*}
    Then the left- and rightward propagation speeds for the limit problem and the corresponding notions for the approximating problem,  
    \begin{align*}
	c_{\text{left}}^0&:= \inf_{\lambda>0} \dfrac{-k^0(-\lambda)}{\lambda}, & c_{\text{right}}^0&:= \inf_{\lambda>0} \dfrac{-k^0(\lambda)}{\lambda}, \\ 
	c_{\text{left}}^\varepsilon&:= \inf_{\lambda>0} \dfrac{-k^\varepsilon(-\lambda)}{\lambda}, & c_{\text{right}}^\varepsilon&:= \inf_{\lambda>0} \dfrac{-k^\varepsilon(\lambda)}{\lambda}, 
    \end{align*}
    are well-defined for $\varepsilon>0$ and 
    \begin{align*}
	\lim_{\varepsilon\to 0}c_{\text{left}}^\varepsilon &=c_{\text{left}}^0, & \lim_{\varepsilon\to 0}c_{\text{right}}^\varepsilon&=c_{\text{right}}^0.
    \end{align*}
    In particular, under the framework described at the beginning of the subsection (see also the Appendix \ref{app:scalar}), it is not difficult achieve  $c^0_{\text{left}}\neq c^0_{\text{right}}$ by a careful selection of the coefficients of the problem. 
\end{rem}
}


\section{Long-time behavior of the original model}
\label{sec:long-time}

In this section we focus on the  original problem \eqref{eq:main-sys}. 
In Section \ref{sec:statsol-ode} we study a related ODE problem and show local asymptotic stability and the uniqueness of stationary solutions. We then extend those results to problems with homogeneous coefficients, in Section \ref{sec:hom-rd}, provided Assumption \ref{as:cond-instab-0} is satisfied. In the same Section we show Theorem \ref{thm:ltb}. Finally we prove Theorem \ref{thm:rapidosc} in Section \ref{sec:hom}.

\subsection{A complete study of the ODE problem}
\label{sec:statsol-ode}

Let us look into the stationary states for the ODE system \eqref{eq:syst-ode}:
\begin{equation*}
	\begin{system}
		u_t=(r_u-\kappa_u(u+v))u+\mu_vv-\mu_uu&=:f^u(u,v), \\
		v_t=(r_v-\kappa_v(u+v))v+\mu_uu-\mu_vv&=:f^v(u,v).
	\end{system}
\end{equation*}
In this Section we work under the assumption that every coefficient in the above equation is positive. Surprisingly, it is possible to show that the solution converges to a unique equilibrium in all cases. To achieve this goal, two different methods are to be employed, depending on the sign of $r_u-\mu_u$ and $r_v-\mu_v$. If one is positive, the system admits a Lyapunov functional, which will be our main tool to study the long-time behavior of the system; whereas in the case where both are nonpositive, the system is ultimately cooperative and the long-time behavior can be determined by monotonicity arguments (here the method of super- and subsolutions). Note that both arguments were inspired by the paper of Cantrell, Cosner and Yu \cite{Can-Cos-Yu-18}. We still include the proofs for the sake of completeness. 
\begin{lem}[Stability of stationary states]\label{lem:stat-ode-stability}
	Let $ r_u, r_v\in\mathbb R$, $\kappa_u>0$, $\kappa_v>0$, and $\mu>0$. Let $(u^*\geq 0, v^*\geq 0)$ be a nontrivial stationary state for \eqref{eq:syst-ode}. Then $(u^*, v^*)$ is locally asymptotically stable.

	More precisely, the Jacobian matrix of the nonlinearity at $(u^*, v^*)$ is 
	\begin{equation}\label{eq:abcd}
		D_{(u^*, v^*)}f=\left(\begin{matrix} r_u-\mu_u-\kappa_u(2u+v) & \mu_v - \kappa_uu \\
		\mu_u-\kappa_vv & r_v-\mu_v-\kappa_v(u+2v) \end{matrix}\right)=:
		\begin{pmatrix}
			a & b \\ c & d
		\end{pmatrix}
	\end{equation}
	and we have $a=-\left(\kappa_u u^*+\mu_v\frac{v^*}{u^*}\right)<0$, $d=-\left(\kappa_v v^*+\mu_u\frac{u^*}{v^*}\right)<0$ and 
	\begin{align*}
		\text{\rm tr} (D_{(u^*, v^*)} f)&=a+d < 0,\\
		\det (D_{(u^*, v^*)} f)&=ad-bc >0.
	\end{align*}
\end{lem}
\begin{proof}
	We divide the proof in three steps.

	{\bf Step 1:} We show that $u^*>0$ and $v^*>0$. 

	Assume by contradiction that $u^*=0$. Then, by our assumption that $(u^*, v^*)$ is non-trivial, we have $v^*>0$. Evaluating the first line of \eqref{eq:syst-ode}, we find $ 0=\mu v^*>0 $, which is a contradiction. 

	The assumption that $v^*>0$ leads to a similar contradiction. We conclude that $u^*>0$ and $v^*>0$.

	Before resuming the proof, let us remark the following formula, which is a consequence of \eqref{eq:syst-ode}:
	\begin{equation*}
		\begin{system}
			r_u-\mu_u-\kappa_u(u^*+v^*)=-\mu_v\frac{v^*}{u^*},\\
			r_v-\mu_v-\kappa_v(u^*+v^*)=-\mu_u\frac{u^*}{v^*}.
		\end{system}
	\end{equation*}

	{\bf Step 2:} We show that $ \text{tr} (D_{(u^*, v^*)} f) < 0$.

	Using the fact that $(u^*, v^*) $ is a stationary state for \eqref{eq:syst-ode}, we have
	\begin{equation*}
		\text{tr} (D_{(u^*, v^*)} f)= r_u-\mu_u-\kappa_u(2u^*+v^*) +  r_v-\mu_v-\kappa_v(u^*+2v^*) = -\mu_v \frac{v^*}{u^*}-\mu_u\frac{u^*}{v^*} -\kappa_u u^*-\kappa_vv^*<0.
	\end{equation*}

	{\bf Step 3:} We show that $\det (D_{(u^*, v^*)} f) >0$.

	We compute:
	\begin{align*}
		\det  (D_{(u^*, v^*)} f) &= (r_u-\mu_u-\kappa_u(2u^*+v^*))(r_v-\mu_v-\kappa_v(u^*+2v^*)) - (\mu_v - \kappa_uu^*)(\mu_u-\kappa_vv^*) \\
		&= \left(\mu_v\frac{v^*}{u^*}+\kappa_uu^*\right)\left( \mu_u\frac{u^*}{v^*}+\kappa_vv^*\right) - (\mu_v - \kappa_uu^*)(\mu_u-\kappa_vv^*) \\
		&= \mu_u\mu_v+\mu_v\kappa_v\frac{(v^*)^2}{u^*}+\mu_u\kappa_u\frac{(u^*)^2}{v^*}+\kappa_u\kappa_vu^*v^* - \mu_u\mu_v+\mu_v\kappa_vv^* + \mu_u\kappa_uu^* -\kappa_u\kappa_vu^*v^* \\
		&=\mu_v\kappa_v\frac{(v^*)^2}{u^*} + \mu_u\kappa_u\frac{(u^*)^2}{v^*}+\mu_v\kappa_uu^*+\mu_u\kappa_vv^*>0.
	\end{align*}

	This finishes the proof of Lemma \ref{lem:stat-ode-stability}
\end{proof}

\begin{lem}[Existence and uniqueness of stationary state]\label{lem:stat-ode-uniqueness}
	Let $ r_u, r_v\in\mathbb R$, $\kappa_u>0$, $\kappa_v>0$, and $\mu_u, \mu_v>0$. Then, there exists at most one nonnegative nontrivial stationary state for \eqref{eq:syst-ode}. If Assumption \ref{as:cond-instab-0} is met, then there is a positive stationary state $(u^*, v^*)$, which satisfies:
	\begin{enumerate}[label={\rm(\roman*)}]
		\item \label{item:lemestu^*-smallmu}
			if $r_u-\mu_u>0$ (resp. $r_v-\mu_v>0$), then 
			\begin{align*}
				\frac{\min\left(\mu_v, r_u-\mu_u\right)}{\kappa_u}&\leq u^*\leq \frac{\max\left(\mu_v,r_u-\mu_u\right)}{\kappa_u} \\
				\text{resp. }\frac{\min\left(\mu_u, r_v-\mu_v\right)}{\kappa_v}&\leq v^*\leq \frac{\max\left( \mu_u,r_v-\mu_v\right)}{\kappa_v}. 
			\end{align*}
			Moreover,  equality happens in the above inequalities if, and only if $\mu_v=r_u-\mu_u$ (resp. $r_v-\mu_v=\mu_u$).
		\item \label{item:lemestu^*-bigmu}
			if $r_u-\mu_u\leq 0$ (resp. $r_v-\mu_v\leq0$), then $0<u^*< \frac{\mu_v}{\kappa_u}$ (resp. $0<v^*<\frac{\mu_u}{\kappa_v}$)
	\end{enumerate}
\end{lem}
\begin{proof}
	Let $(u,v)$ be a nonnegative nontrivial stationary state for \eqref{eq:syst-ode}. Then $(u,v)$ satisfies
	\begin{equation*}
		\begin{system}
			\relax &u(r_u-\kappa_u(u+v))+\mu_vv-\mu_uu=0, \\
			\relax &v(r_v-\kappa_v(u+v))+\mu_uu-\mu_vv=0. 
		\end{system}
	\end{equation*}
	As remarked in the proof of Lemma \ref{lem:stat-ode-stability}, since $(u,v)$ is nonegative and nontrivial, we have in fact $u>0$ and $v>0$. Let us change the variables: 
	\begin{equation*}
		\begin{system}
			\relax &S=u+v, \\
			\relax &Q=\frac{u}{v}.
		\end{system}
	\end{equation*}
	Then the new variables $(S, Q)$ satisfy the system:
	\begin{align*}
		& \begin{system}
			\relax &Q(r_u-\kappa_uS)+\mu_v-\mu_uQ=0, \\
			\relax &r_v-\kappa_vS+\mu_uQ-\mu_v=0,
		\end{system} \\
		\Leftrightarrow & 
		\begin{system}
			\relax &Q(r_u-\kappa_uS)+\mu_v-\mu_uQ=0, \\
			\relax &S=\frac{r_v+\mu_uQ-\mu_v}{\kappa_v},
		\end{system}
		\\
		\Leftrightarrow &
		\begin{system}
			\relax &-\mu_u\frac{\kappa_u}{\kappa_v}Q^2 + \left(r_u-\mu_u\frac{\kappa_u}{\kappa_v}(r_v-\mu_v)\right)Q+\mu_v=0, \\
			\relax &S=\frac{r_v+\mu_uQ-\mu_v}{\kappa_v}.
		\end{system}
	\end{align*}
	The first line of the latter system has a unique positive solution: 
	\begin{equation*}
		Q=\frac{\kappa_v}{2\mu_u\kappa_u}\left(r_u-\mu_u-\frac{\kappa_u}{\kappa_v}(r_v-\mu_v)+\sqrt{\left(r_u-\mu_u-\frac{\kappa_u}{\kappa_v}(r_v-\mu_v)\right)^2+4\frac{\kappa_u}{\kappa_v}\mu_u\mu_v}\right).
	\end{equation*}
	Since the change of variables is reversible, there cannot exist two nonnegative nontrivial solutions for the original system.

	The proof of existence of a stationary solution is quite straightforward by using a global bifurcation argument; we refer to an earlier work \cite[Theorem 2.3]{Alf-Gri-18} for a proof in a periodic setting.

	Next we focus on the estimates on nontrivial stationary states. Since the statement is symmetric with respect to the variable $u$ or $v$, we only prove the result for $u^*$. 
	Assume first that $r_u-\mu_u>\mu_v>0$. Then $u^*$ satisfies:
	\begin{equation}\label{eq:estonu^*}
		0=u^*(r_u-\mu_u-\kappa_uu^*)+v^*(\mu_v-\kappa_uu^*).
	\end{equation}
	If $u^*<\frac{\mu_v}{\kappa_u}$, then both terms in the right-hand side of \eqref{eq:estonu^*} are positive, which is a contradiction. Similarly, if $u^*>\frac{r_u-\mu_u}{\kappa_u}$, then both terms are negative, which is also a contradiction. We conclude that $\frac{\mu_v}{\kappa_u}\leq u^*\leq \frac{r_u-\mu_u}{\kappa_u}$. Finally, if equality is achieved in the latter inequality, then one of the terms in \eqref{eq:estonu^*} is 0 and the other is positive, which is a contradiction. Thus
	\begin{equation*}
		\frac{\mu_v}{\kappa_u}<u^*<\frac{r_u-\mu_u}{\kappa_u}.
	\end{equation*}
	In the case $0<r_u-\mu_u<\mu_v$, a similar argument shows that 
	\begin{equation*}
		\frac{r_u-\mu_u}{\kappa_u}<u^*<\frac{\mu_v}{\kappa_u}.
	\end{equation*}
	Finally, if $r_u-\mu_u=\mu_v$, then both terms in the right-hand side of \eqref{eq:estonu^*} have the same sign independently of $u^*$, hence the only possibility is
	\begin{equation*}
		u^*=\frac{r_u-\mu_u}{\kappa_u}=\frac{\mu_v}{\kappa_u}.
	\end{equation*}
	Statement \ref{item:lemestu^*-smallmu} is proved. To show Statement \ref{item:lemestu^*-bigmu}, since $r_u-\mu_u\leq 0$, we simply rewrite \eqref{eq:estonu^*} as:
	\begin{equation*}
		u^*=\frac{\mu_v}{\kappa_u}+\frac{u^*}{\kappa_uv^*}(r_u-\mu_u-\kappa_uu^*)<\frac{\mu_v}{\kappa_u}.
	\end{equation*}

	Lemma \ref{lem:stat-ode-uniqueness} is proved. 
\end{proof}
Next we study the stability of the trivial steady state: 
\begin{lem}[Stability of 0]\label{lem:stab-0-ode}
	Let $r_u, r_v\in\mathbb R$ and $\mu_u>0$, $\mu_v>0$. Define
	\begin{equation*}
		A:=\begin{pmatrix}r_u-\mu_u & \mu_v \\ \mu_u & r_v-\mu_v\end{pmatrix}.
	\end{equation*}
	Then, the matrix $A$ has two simple real eigenvalues: 
	\begin{align*}
		\lambda_1&=\dfrac{r_u-\mu_u+r_v-\mu_v+\sqrt{(r_u-\mu_u-r_v+\mu_v)^2+4\mu_u\mu_v}}{2} , \\
		\lambda_2&=\dfrac{r_u-\mu_u+r_v-\mu_v-\sqrt{(r_u-\mu_u-r_v+\mu_v)^2+4\mu_u\mu_v}}{2}.
	\end{align*}
	Moreover the eigenvector associated with $\lambda_1$ lies in the first quadrant:
	\begin{equation*}
		\varphi_1:=\begin{pmatrix}r_u-\mu_u-(r_v-\mu_v)+\sqrt{(r_u-\mu_u-(r_v-\mu_v))^2+4\mu_u\mu_v} \\ 2\mu_u\end{pmatrix}.
	\end{equation*}
	Finally, if $\frac{\mu_v}{\mu_u+\mu_v}r_u+\frac{\mu_u}{\mu_u+\mu_v}r_v>0$, then $\lambda_1>0$.
\end{lem}
\begin{proof}
	The characteristic polynomial associated with $A$ is: 
	\begin{equation*}
		\chi_A(X)=X^2-(r_u-\mu_u+r_v-\mu_v)X+r_ur_v-\mu_ur_v-\mu_vr_u.
	\end{equation*}
	The roots of this second-order polynomial can be computed thanks to its discriminant $\Delta$:
	\begin{align*}
		\Delta&=(r_u+r_v-\mu_u-\mu_v)^2-4(r_ur_v-\mu_vr_u-\mu_ur_v) \\
		&=r_u^2+r_v^2+\mu_u^2+\mu_v^2+2r_ur_v-2r_u\mu_u-2r_u\mu_v-2r_v\mu_u-2r_v\mu_v+2\mu_u\mu_v\\
		&\quad -4r_ur_v+4\mu_vr_u+4\mu_ur_v \\
		&=r_u^2+r_v^2+\mu_u^2+\mu_v^2-2r_ur_v-2r_u\mu_u+2r_u\mu_v+2r_v\mu_u-2r_v\mu_v+2\mu_u\mu_v\\
		&=\big(r_u-\mu_u-(r_v-\mu_v)\big)^2+4\mu_u\mu_v.
	\end{align*}
	In particular, $\Delta>0$ and thus $\chi_A$ always has two real roots:
	\begin{align*}
		\lambda_1&:=\frac{r_u-\mu_u+r_v-\mu_v+\sqrt{\big(r_u-\mu_u-(r_v-\mu_v)\big)^2+4\mu_u\mu_v}}{2},\\
		\lambda_2&:=\frac{r_u-\mu_u+r_v-\mu_v-\sqrt{\big(r_u-\mu_u-(r_v-\mu_v)\big)^2+4\mu_u\mu_v}}{2}.
	\end{align*}
	The eigenvector associated with $\lambda_1$ can be easily computed and is always positive: 
	\begin{equation*}
		\varphi_1:=\begin{pmatrix}r_u-\mu_u-(r_v-\mu_v)+\sqrt{(r_u-\mu_u-(r_v-\mu_v))^2+4\mu_u\mu_v} \\ 2\mu_u\end{pmatrix}.
	\end{equation*}
	It follows easily from the Perron-Frobenius Theorem that $\varphi_1$ is the unique positive eigenvector of the matrix $A$.

	Next we investigate the sign of $\lambda_1$. To this end we introduce for $\alpha\geq 0$: 
	\begin{equation*}
		\lambda_1(\alpha):=\frac{r_u-\alpha\mu_u+r_v-\alpha\mu_v+\sqrt{\big(r_u-\alpha\mu_u-(r_v-\alpha\mu_v)\big)^2+4\alpha^2\mu_u\mu_v}}{2}.
	\end{equation*}
	Notice that $\lambda_1=\lambda_1(\alpha=1)$, the mapping $\alpha\mapsto \lambda_1(\alpha)$ is convex (as we will show below) and $\lambda_1(0)=\max(r_u, r_v)>0$. To catch the behavior of the function as $\alpha\to+\infty$, we rewrite $\lambda_1(\alpha)$ as
	\begin{align*}
		\lambda_1(\alpha)&=\frac{1}{2}\Big(r_u-\alpha\mu_u+r_v-\alpha\mu_v \\
		&\quad+\sqrt{r_u^2+r_v^2-2r_ur_v+2\alpha(-r_u\mu_r+r_u\mu_v+r_v\mu_u-r_v\mu_v)+\alpha^2(\mu_u^2+\mu_v^2+2\mu_u\mu_v)}\Big)\\
		&=\frac{1}{2}\left(r_u-\alpha\mu_u+r_v-\alpha\mu_v+\sqrt{(r_u-r_v)^2+2\alpha(r_v-r_u)(\mu_u-\mu_v)+\alpha^2(\mu_u+\mu_v)^2}\right)\\
		&=\frac{1}{2}\Bigg(r_u+r_v-\alpha(\mu_u+\mu_v) +\alpha(\mu_u+\mu_v)\sqrt{1+\frac{2(r_v-r_u)(\mu_u-\mu_v)}{\alpha(\mu_u+\mu_v)^2}+o_{\alpha\to+\infty}\left(\frac{1}{\alpha}\right)}\Bigg)\\
		&=\frac{1}{2}\left(r_u+r_v+\frac{(r_v-r_u)(\mu_u-\mu_v)}{\mu_u+\mu_v}+o_{\alpha\to+\infty}(1)\right)\\
		&=\frac{\mu_v}{\mu_u+\mu_v}r_u+\frac{\mu_u}{\mu_u+\mu_v}r_v+o_{\alpha\to+\infty}(1).
	\end{align*}
	Thus, $\lambda_1(\alpha)>0$ for all $\alpha>0$, and in particular $\lambda_1=\lambda_1(\alpha=1)>0$.

	To show the convexity of $\alpha\mapsto\lambda_1(\alpha)$, we simply notice that
	\begin{equation*}
		\dfrac{\dd^2}{\dd x^2}\lambda_1(\alpha)=\dfrac{4A^2\alpha^2+4AB\alpha+8AC-B^2}{16\left(A\alpha^2+B\alpha+C\right)^{\frac{3}{2}}},
	\end{equation*}
	with
	\begin{align*}
		A:&=(\mu_u+\mu_v)^2, \\
		B:&=2(r_v-r_u)(\mu_u-\mu_v), \\
		C:&=(r_u-r_v)^2.
	\end{align*}
	Hence the roots of $\dfrac{\dd^2}{\dd x^2}\lambda_1(\alpha)$ are determined by the quantity
	\begin{equation*}
		\Delta:=32A^2B^2-128A^3C=32A^2(B^2-4AC).
	\end{equation*}
	Since:
	\begin{align*}
		B^2-4AC&=4(r_u-r_v)^2(\mu_u-\mu_v)^2-4(r_u-r_v)(\mu_u+\mu_v)\\
		&= 4(r_u-r_v)^2\left((\mu_u-\mu_v)^2-(\mu_u+\mu_v)^2\right) \\
		&=-16(r_u-r_v)^2\mu_u\mu_v\leq 0,
	\end{align*}
	then for all $\alpha>0$ we have  $\dfrac{\dd^2}{\dd x^2}\lambda_1(\alpha)\geq 0$, hence $\lambda_1(\alpha)$ is convex. This finishes the proof of Lemma \ref{lem:stab-0-ode}.
\end{proof}
\begin{rem}[Stability of $0$] 
	The previous computation can be carried out the same way independently of the sign of $r_u$ and $r_v$. This gives a criterion for the instability of $0$ in the case $\min(r_u, r_v)<0$: if $\frac{\mu_v}{\mu_u+\mu_v}r_u+\frac{\mu_u}{\mu_u+\mu_v}r_v\geq0$ then $0$ is always unstable, whereas if $\frac{\mu_v}{\mu_u+\mu_v}r_u+\frac{\mu_u}{\mu_u+\mu_v}r_v<0$ the stability of $0$ depends on the size of the mutation rate. In the latter case, the ratio $\frac{\mu_u}{\mu_v}$ being fixed, $0$ is always unstable if $\mu_u, \mu_v$ are sufficiently small, and always stable if $\mu_u, \mu_v$ are sufficiently large.
\end{rem}

We are now in a position to give our arguments for the long-time behavior of the ODE problem. We begin with the case where there exists a Lyapunov functional for the system. We introduce the functionals:
\begin{equation}\label{eq:Lyapunov-func}
	\mathcal F_u(u):=u-u^*-u^*\ln\left(\frac{u}{u^*}\right), \qquad \mathcal F_v(v):=v-v^*-v^*\ln\left(\frac{v}{v^*}\right).
\end{equation}
Note that this Lyapunov functional is rather classical and has been used for instance by Hsu \cite{Hsu-78} in a competitive context. The present argument was inspired by the more recent article of Cantrell, Cosner and Yu \cite{Can-Cos-Yu-18}.
\begin{lem}[Lyapunov functional]\label{lem:Lyapunov-ode}
	Let Assumption \ref{as:cond-instab-0} hold and assume $\max(r_u-\mu_u, r_v-\mu_v)>0$. Then, there is $K>0$ such that the functional $\mathcal F^K(u, v):=\mathcal F_u(u)+K\mathcal F_v(v)$ is a Lyapunov functional for \eqref{eq:syst-ode}, i.e. 
	\begin{equation*}
		\frac{\dd}{\dd t}\mathcal F^K(u(t),v(t))\leq 0, 
	\end{equation*}
	along any positive trajectory $(u(t), v(t))$. Moreover, the inequality is strict unless $(u(t), v(t))=(u^*, v^*)$. 
\end{lem}
\begin{proof}
	Since it is clear that $\mathcal F(u^*, v^*)=0$, we will focus on the case of an orbit starting from a positive initial condition $(u_0, v_0)$. We first compute: 
	\begin{align*}
		\frac{\dd}{\dd t}\mathcal F_u(u(t))&= u_t\left(1-\frac{u^*}{u}\right)= (u-u^*)\frac{u_t}{u}\\
		&=(u-u^*)\left(r_u-\mu_u-\kappa_u(u+v)+\mu_v\frac{v}{u}\right) \\
		&=(u-u^*)\left(\kappa_u(u^*+v^*)-\mu_v\frac{v^*}{u^*}-\kappa_u(u+v)+\mu_v\frac{v}{u}\right) \\ 
		&=-\kappa_u(u-u^*)^2 -\kappa_u(u-u^*)(v-v^*)+\mu_v(u-u^*)\left(\frac{u^*v-uv^*}{uu^*}\right) \\
		&=-\left(\kappa_u+\mu_v\frac{v^*}{uu^*}\right)(u-u^*)^2 -\left(\kappa_u-\frac{\mu_v}{u^*}\right)(u-u^*)(v-v^*) \\
		&\leq-\kappa_u(u-u^*)^2-\left(\kappa_u-\frac{\mu_v}{u^*}\right)(u-u^*)(v-v^*),
	\end{align*}
	and the  inequality is strict unless $u=u^*$. Similarly, 
	\begin{equation*}
		\frac{\dd}{\dd t}\mathcal F_v(y)\leq-\kappa_v(v-v^*)^2-\left(\kappa_v-\frac{\mu_u}{v^*}\right)(u-u^*)(v-v^*),
	\end{equation*}
	and the inequality is strict unless $v=v^*$. Since $(u, v)\neq (u^*, v^*)$, we have for all $K>0$: 
	\begin{multline*}
		\frac{\dd}{\dd t}\mathcal F^K(u,v)<-\kappa_u(u-u^*)^2 -\left(\kappa_u-\frac{\mu_v}{u^*} + K\left(\kappa_v-\frac{\mu_u}{u^*}\right)\right)(u-u^*)(v-v^*)-K\kappa_v(v-v^*)^2.
	\end{multline*}
	Next we prove that the right-hand side can be made nonpositive for all $(u,v)>(0,0)$ for a well-chosen value of $K$. We remark that the right-hand side is a quadratic form in $(U:=u-u^*, V:=v-v^*)$, which can be written as $-Q(U,V)$ where:
	\begin{equation*}
		Q(U,V):=AU^2+(B+KC)UV+KDV^2, 
	\end{equation*}
	and $U=u-u^*$, $V=v-v^*$, $A=\kappa_u$, $B=\kappa_u-\frac{\mu_v}{u^*}$, $C=\kappa_v-\frac{\mu_u}{v^*}$ and $D=\kappa_v$. We claim that $Q(U,V)$ can be made positive definite by a proper choice of $K>0$. Indeed, algebraic computations lead to 
	\begin{align*}
		Q(U,V)&=A\left(U+\frac{B+KC}{2A}V\right)^2+\left(KD-\frac{(B+KC)^2}{4A}\right)V^2, 
	\end{align*}
	and therefore it suffices to find $K>0$ such that 
	\begin{equation*}
		0<KD-\frac{(B+KC)^2}{4A}=\dfrac{-C^2K^2+(4AD-2BC)K-B^2}{4A}=:-\frac{P(K)}{4A}.
	\end{equation*}
	Now $P(K)$ is a second-order polynomial and its number of roots is determined by the sign of the quantity
	\begin{equation*}
		\Delta=(4AD-2BC)^2-4B^2C^2=16AD(AD-BC)>0. 
	\end{equation*}
	If $BC<AD$, the polynomial $P$ has two roots, and those roots have to be nonnegative since $P(K)<0$ for all $K<0$. This shows that there exists $K>0$ with $P(K)<0$, which proves our claim and consequently finishes the proof of Lemma \ref{lem:Lyapunov-ode}.

	Our last task is therefore to check that $BC<AD$. Assume first that $r_u-\mu_u>0$ and $r_v-\mu_v>0$, then $B=\kappa_u-\frac{\mu_v}{u^*}>0$ and $C=\kappa_v-\frac{\mu_u}{v^*}>0$ are both positive by Lemma \ref{lem:stat-ode-uniqueness}. Thus,
	\begin{equation*}
		BC=\left(\kappa_u-\frac{\mu_v}{u^*}\right)\left(\kappa_v-\frac{\mu_u}{v^*}\right)\leq \kappa_u\kappa_v=AD.
	\end{equation*}
	Next assume that $r_u-\mu_u\leq 0$ and $r_v-\mu_v>0$ (the case $r_v-\mu_v\leq 0$ and $r_u-\mu_u>0$ can be treated similarly). In this case, $\kappa_u-\frac{\mu_v}{u^*}\leq 0$ and $\kappa_v-\frac{\mu_u}{v^*}>0$ and thus
	\begin{equation*}
		BC=\left(\kappa_u-\frac{\mu_v}{u^*}\right)\left(\kappa_v-\frac{\mu_u}{v^*}\right)\leq0< \kappa_u\kappa_v=AD.
	\end{equation*}
	Hence $BC<AD $ always holds under our hypotheses. Lemma \ref{lem:Lyapunov-ode} is proved.
\end{proof}
Notice in particular that Proposition \ref{prop:longtime-ode} follows directly from Lemma \ref{lem:Lyapunov-ode} in the case $\max(r_u-\mu_u, r_v-\mu_v)>0$. Next we consider the case $\max(r_u-\mu_u, r_v-\mu_v)\leq 0$. In this case, we show that the dynamics is eventually cooperative and we use the method of monotone iteration to conclude. 

\begin{lem}[Ultimately cooperative dynamics]\label{lem:coop}
	Let Assumption \ref{as:cond-instab-0} hold and assume $\max(r_u-\mu_u, r_v-\mu_v)\leq 0$. Then, we have 
	\begin{align*}
		\lim_{t\to+\infty}(u(t), v(t))=(u^*, v^*).
	\end{align*}
\end{lem}
\begin{proof}
	Let $(u(t), v(t))$ be a positive solution to \eqref{eq:syst-ode}. Then $(u(t), v(t))$ is a subsolution to the cooperative system:
	\begin{equation*}
		\left\{\begin{aligned}
			\bar u_t&=\bar u(r_u-\mu_u-\kappa_u\bar u)+\bar v\max(\mu_v-\kappa_u\bar u, 0),\\
			\bar v_t&=\bar v(r_v-\mu_v-\kappa_v\bar v)+\bar u\max(\mu_u-\kappa_v\bar v, 0),
		\end{aligned}\right.
	\end{equation*}
	and in particular $u(t)\leq \bar u(t)$ and $v(t)\leq \bar v(t)$. Since $\bar u, \bar v$ eventually enters the cooperative region $0<\bar u< \frac{\mu_v}{\kappa_u}$ and $0<\bar v< \frac{\mu_u}{\kappa_v}$, so does $(u,v)$. Next we use the method of sub- and supersolutions to show the convergence of $(u(t),v(t))$ starting from $(u_0, v_0)\in \left(0, \frac{\mu_v}{\kappa_u}\right)\times \left(0, \frac{\mu_u}{\kappa_v}\right)$. We remark that $(\bar u, \bar v):=(\frac{\mu_v}{\kappa_u}, \frac{\mu_u}{\kappa_v})$ is a strict supersolution: 
	\begin{equation*}
		\left\{\begin{aligned}
			\bar u(r_u-\mu_u-\kappa_u\bar u)+\bar v(\mu_v-\kappa_u\bar u)&<0, \\ 
			\bar v(r_v-\mu_v-\kappa_v\bar v)+\bar u(\mu_u-\kappa_v\bar v)&<0,
		\end{aligned}\right.
	\end{equation*}
	while for $\alpha>0$ sufficiently small the vector $(\underline u, \underline v):=\alpha\varphi_1$ (where $\varphi_1$ is defined in Lemma \ref{lem:stab-0-ode}) is a strict subsolution: 
	\begin{equation*}
		\left\{\begin{aligned}
			\underline u(r_u-\mu_u-\kappa_u\underline u)+\underline v(\mu_v-\kappa_u\underline u)&=\lambda_1\underline u+o(\alpha)>0, \\ 
			\underline v(r_v-\mu_v-\kappa_v\underline v)+\underline u(\mu_u-\kappa_v\underline v)&=\lambda_1\underline v+o(\alpha)>0.
		\end{aligned}\right.
	\end{equation*}
	Then, the technique of monotone iterations gives us a maximal stationary solution $(\overline u^*, \overline v^*)<(\bar u, \bar v)$ and a minimal stationary solution $(\underline u^*, \underline v^*)>(\underline u, \underline v)$ such that: 
	\begin{equation*}
		(\underline u^*, \underline v^*)\leq \liminf_{t\to+\infty}(u(t), v(t))\leq \limsup_{t\to+\infty}(u(t), v(t))\leq (\overline u^*, \overline v^*).
	\end{equation*}
	Finally since $(u^*, v^*)$ is the unique stationary solution to \eqref{eq:syst-ode}, we have indeed: 
	\begin{equation*}
		( u^*,  v^*)\leq \liminf_{t\to+\infty}(u(t), v(t))\leq \limsup_{t\to+\infty}(u(t), v(t))\leq (u^*, v^*),
	\end{equation*}
	hence $(u(t), v(t))$ converges to the stationary solution as $t\to+\infty$. Lemma \ref{lem:coop} is proved. 
\end{proof}

We are now in a position to prove Proposition \ref{prop:longtime-ode} and conclude this Section:
\begin{proof}[Proof of Proposition \ref{prop:longtime-ode}]
	If $\lambda_A>0$, the existence and uniqueness of a stationary solution $(u^*, v^*)$ has been shown in Lemma \ref{lem:stat-ode-uniqueness}. The convergence of $(u(t), v(t))$ when $t\to+\infty$ has been shown in Lemma \ref{lem:Lyapunov-ode} if $\max(r_u-\mu_u, r_v-\mu_v)>0$ by the means of a Lyapunov argument, and in Lemma \ref{lem:coop} if $\max(r_u-\mu_u, r_v-\mu_v)\leq 0$ by the means of a monotone iteration sequence. This covers all the possibilities and hence finishes the proof of the statement in Proposition \ref{prop:longtime-ode} when $\lambda_A>0$.

	Let $\lambda_A<0$, and let $(u(t), v(t))$ be a solution to \eqref{eq:syst-ode}. Then $(u(t), v(t))$ is a sub-solution for the cooperative system 
	\begin{equation*}
		\begin{system}
			\relax &\bar u_t=(r_u-\mu_u)\bar u+\mu_v\bar v, \\
			\relax &\bar v_t=\mu_u\bar u+(r_v-\mu_v)\bar v,
		\end{system}
	\end{equation*}
	hence for $M>\max(u(0), v(0))$ and  $(\bar u(t), \bar v(t)):=Me^{\lambda_At}(\varphi_A^u, \varphi_A^v)$  we have
	\begin{equation*}
		(u(t), v(t))\leq (\bar u(t), \bar v(t))
	\end{equation*}
	for all $t>0$, and in particular $\lim_{t\to\infty} \max (u(t),v(t))\leq \lim_{t\to\infty} \max (\bar u(t),\bar v(t))=0$.

	Let  $\lambda_A=0$, and let $(u(t), v(t))$ be a solution to \eqref{eq:syst-ode}. Then $(u(t), v(t))$ is a sub-solution for the cooperative system 
	\begin{equation*}
		\begin{system}
			\relax &\bar u_t=(r_u-\mu_u-\min(\kappa_u, \kappa_v)\bar u)\bar u+\mu_v\bar v, \\
			\relax &\bar v_t=\mu_u\bar u+(r_v-\mu_v-\min(\kappa_u, \kappa_v)\bar v)\bar v.
		\end{system}
	\end{equation*}
	Let $(\bar u(0), \bar v(0)):=M_0(\varphi^u_A, \varphi^v_A)$ where $(\varphi^u_A, \varphi^v_A) $ is the principal eigenvector of the matrix $A$ as defined in Assumption \ref{as:cond-instab-0}. Then $(\bar u(t), \bar v(t))=M(t)(\varphi_A^u, \varphi_A^v)$ and the function $M(t)$ satisfies
	\begin{equation*}
		\frac{\dd }{\dd t}M(t)=-\min(\kappa_u, \kappa_v)M(t)^2.
	\end{equation*}
	In particular, $\lim_{t\to\infty}M(t)=0$. Since $M_0$ can be chosen so that $(u(0), v(0))\leq (\bar u(0), \bar v(0))$, we deduce that $\lim_{t\to\infty} (u(t), v(t))=(0,0)$. Proposition \ref{prop:longtime-ode} is proved.
\end{proof}

\subsection{Long-time behavior for the solutions to the  homogeneous problem}
\label{sec:hom-rd}
We aim at showing that the $\omega$-limit set of a positively bounded from below initial condition 
is reduced to a single element $\{(u^*, v^*)\}$, where $(u^*, v^*)$ is the unique stationary state for \eqref{eq:syst-ode}. As shown below, we can prove such a result only for a subset of the set of parameters.

\begin{thm}[Entire solutions]\label{thm:entire-sol-hom}
	Let Assumption \ref{as:cond-instab-0} and \ref{as:LTPDE} hold. Let $(u(t, x), v(t,x))$ be an nonnegative bounded entire solution to \eqref{eq:syst-hom-rd}. Assume that $(u,v)$ is bounded from below, i.e. that there exists $\delta>0$ such that for all $t\in\mathbb R$ and $x\in\mathbb R$ we have:
	\begin{equation*}
		u(t, x)\geq \delta>0\text{ and } v(t,x)\geq \delta>0,
	\end{equation*}
	then $(u, v)\equiv (u^*, v^*)$.
\end{thm}
\begin{proof}
	We divide the proof in three steps.

	\begin{stepping}
		\step The ultimately cooperative case: $\max(r_u-\mu_u, r_v-\mu_v)\leq 0$, with $(\sigma_1, \sigma_2)> (0,0)$.

		In this case, our argument is very similar to the one in Lemma \ref{lem:coop}. We first notice that $(u(t, x), v(t,x))$ is a sub-solution to the cooperative ODE system:
		\begin{equation}\label{eq:entiresol-upest}
			\begin{system}
				\relax &\bar u_t=(r_u-\mu_u-\kappa_u\bar u)\bar u+\bar v\max(\mu_v-\kappa_u\bar u, 0), \\
				\relax &\bar v_t=(r_v-\mu_v-\kappa_v\bar v)\bar v+\bar u\max(\mu_u-\kappa_v\bar v, 0),
			\end{system}
		\end{equation}
		and in particular $u(t, x)\leq \bar u(t)$ and $v(t, x)\leq \bar v(t)$. Since $\bar u, \bar v$ eventually enters the cooperative region $0<\bar u< \frac{\mu_v}{\kappa_u}$ and $0<\bar v< \frac{\mu_u}{\kappa_v}$, so does $(u,v)$. Moreover, since $(u(t,x), v(t,x))$ is defined for all $t\in\mathbb R$, we deduce that
		\begin{align*}
			\sup_{(t,x)\in\mathbb R^2}u(t,x)&<\frac{\mu_v}{\kappa_u},\\
			\sup_{(t,x)\in\mathbb R^2}v(t,x)&<\frac{\mu_u}{\kappa_v}.
		\end{align*}
		Hence, the entire  solution $(u(t,x), v(t,x))$ of the reaction-diffusion system \eqref{eq:syst-hom-rd} stays in the cooperative region and  can thus be compared with the solution to the ODE system \eqref{eq:syst-ode}. More precisely, for all $t\in\mathbb R$ and $x\in\mathbb R$, we have:
		\begin{align*}
			\underline u(t_0, t)&\leq u(t,x)\leq \bar u(t_0, t), \\
			\underline v(t_0, t)&\leq v(t,x)\leq \bar v(t_0, t),
		\end{align*}
		where $(\underline u(t_0, t), \underline v(t_0, t))$ is the solution to \eqref{eq:syst-ode} at time $t$ starting from the initial condition $(\underline u(t_0), \underline v(t_0))=(\delta, \delta)$, and $(\bar u(t_0, t)\bar v(t_0, t)) $ is the solution to \eqref{eq:syst-ode} at time $t$ starting from the initial condition $(\bar u(t_0), \bar v(t_0))=\left(\frac{\mu_v}{\kappa_u}, \frac{\mu_u}{\kappa_v}\right)$. Since:
		\begin{equation*}
			\lim_{t_0\to-\infty} (\underline u(t_0, t), \underline v(t_0, t))=\lim_{t_0\to-\infty}(\bar u(t_0, t), v(t_0, t))=(u^*, v^*), 
		\end{equation*}
		we have indeed $(u(t,x), v(t,x))\equiv (u^*, v^*)$ and Theorem \ref{thm:entire-sol-hom} is proved in this case. \medskip

		\step The ultimately competitive case: $\min(r_u-\mu_u-\mu_v, r_v-\mu_v-\mu_u)> 0$, with $(\sigma_1, \sigma_2)> (0,0)$.

		First notice that, as in Step 1, the solution $(u(t,x), v(t,x))$ can be controlled from above by the solution to the ODE \eqref{eq:entiresol-upest}, and hence we have the upper estimate:
		\begin{equation*}
			u(t,x)<\frac{r_u-\mu_u}{\kappa_u}, \qquad v(t,x)<\frac{r_v-\mu}{\kappa_v}.
		\end{equation*}
		Next we remark that $(u(t,x), v(t,x))$ is a supersolution to the cooperative system:
		\begin{equation*}
			\begin{system}
				\relax &\underline u_t=\underline u\left(r_u-\mu_u-\kappa_u\left(\frac{r_u-\mu_u}{\mu_v}\right)\underline u\right)+\delta(\mu_v-\kappa_u\underline u), \\
				\relax &\underline v_t=\underline v\left(r_v-\mu_v-\kappa_v\left(\frac{r_v-\mu_v}{\mu_u}\right)\underline v\right)+\delta(\mu_u-\kappa_v\underline v).
			\end{system}
		\end{equation*}
		In particular, we have for all $t\in\mathbb R$ and $x\in\mathbb R$:
		\begin{equation*}
			\frac{\mu_v}{\kappa_u}\leq u(t, x), \qquad \frac{\mu_u}{\kappa_v}\leq v(t,x).
		\end{equation*}
		Hence $(u(t, x), v(t,x))$ stays in the invariant rectangle $\left[\frac{\mu_v}{\kappa_u}, \frac{r_u-\mu_u}{\kappa_u}\right]\times\left[\frac{\mu_u}{\kappa_v}, \frac{r_v-\mu_v}{\kappa_v}\right]$, where system \eqref{eq:syst-hom-rd} is competitive. In particular, system \eqref{eq:syst-hom-rd} is order-preserving for the non-classical order $\leq_c $ on $\mathbb R^2$ (see e.g. \cite[Proposition 5.1]{Smi-95}): 
		\begin{equation*}
			(u, v)\leq_c (\tilde u, \tilde v) \Longleftrightarrow u\leq \tilde u \text{ and } v\geq \tilde v,
		\end{equation*}
		in this rectangle. Thus, 
		\begin{align*}
			\underline u(t_0, t)&\leq_c u(t,x)\leq_c \bar u(t_0, t), \\
			\underline v(t_0, t)&\leq_c v(t,x)\leq_c \bar v(t_0, t),
		\end{align*}
		where $(\underline u(t_0, t), \underline v(t_0, t))$ is the solution to \eqref{eq:syst-ode} at time $t$ starting from the initial condition $(\underline u(t_0), \underline v(t_0))=\left(\frac{\mu_u}{\kappa_u}, \frac{r_v-\mu_v}{\kappa_v}\right)$, and $(\bar u(t_0, t)\bar v(t_0, t)) $ is the solution to \eqref{eq:syst-ode} at time $t$ starting from the initial condition $(\bar u(t_0), \bar v(t_0))=\left(\frac{r_u-\mu}{\kappa_u}, \frac{\mu_u}{\kappa_v}\right)$. Since:
		\begin{equation*}
			\lim_{t_0\to-\infty} (\underline u(t_0, t), \underline v(t_0, t))=\lim_{t_0\to-\infty}(\bar u(t_0, t), v(t_0, t))=(u^*, v^*), 
		\end{equation*}
		we have indeed $(u(t,x), v(t,x))\equiv (u^*, v^*)$ and Theorem \ref{thm:entire-sol-hom} is proved in this case. \medskip

		\step The Lyapunov case: $\sigma_u=\sigma_v=:\sigma$ and $\max(r_u-\mu_u, r_v-\mu_v)\geq 0$.

		In this case the system is mixed quasimonotone and, to the extent of our knowledge, monotonicity arguments cannot be employed. We therefore turn to a generalisation of the Lyapunov argument which was used in Lemma \ref{lem:Lyapunov-ode}. Let $\mathcal F_u$, $\mathcal F_v$ be the functions defined in \eqref{eq:Lyapunov-func} and $K$ be the constant given by Lemma \ref{lem:Lyapunov-ode}, so that $\mathcal F^K(u, v):=\mathcal F_u(u)+K\mathcal F_v(v)$ is a Lyapunov functional for the flow of the ODE \eqref{eq:syst-ode}. Define $w(t, x)=\mathcal F^K(u(t,x), v(t,x))$. Then $w$ satisfies:
		\begin{align*}
			w_t-\sigma w_{xx}&=(u_t-\sigma u_{xx})\mathcal F_u'(u)+K(v_t-\sigma v_{xx})\mathcal F_v'(v)-\sigma(u_x^2\mathcal F_u''(u) + Kv_x^2\mathcal F_v''(v))\\
			&\leq -\kappa_u(u-u^*)^2-\left(\kappa_u-\frac{\mu_v}{u^*}+K\left(\kappa_v-\frac{\mu_u}{v^*}\right)\right)-K\kappa_v(v-v^*)^2=:Q(u,v).
		\end{align*}
		Indeed $\mathcal F_u$ and $\mathcal F_v$ are convex functions, hence $\mathcal F_u''(u)\geq 0$ and $\mathcal F_v''(v)\geq 0$ for all $u,v$. Since $Q(u,v)\leq 0$, $w$ is a bounded entire subsolution to the heat equation, therefore has to be a constant. Since $Q(u,v)<0$ whenever $(u, v)\neq (u^*, v^*)$, the only possibility is $w(t, x)\equiv 0$ and therefore $(u(t, x), v(t,x))\equiv (u^*, v^*)$. This finishes the proof of Theorem \ref{thm:entire-sol-hom} in the case $\max(r_u-\mu_u, r_v-\mu_v)\geq 0$ and $\sigma_u=\sigma_v$.
	\end{stepping}

	Since all the possible cases have been covered, Theorem \ref{thm:entire-sol-hom} is proved.
\end{proof}

\begin{proof}[Proof of Theorem \ref{thm:ltb}]
	Let $(u_0(x)\geq 0, v_0(x)\geq 0)$ be a nontrivial initial condition and $(u(t,x), v(t,x))$ be the corresponding solution to \eqref{eq:main-sys}. We argue by contradiction and assume that there exists $\varepsilon>0$ and  a sequences $t_n\to+\infty$ and $x_n\in\mathbb R$ such that $|x_n|\leq ct_n$ and 
	\begin{equation*}
		|u(t_n, x_n)-u^*|\geq\varepsilon>0.
	\end{equation*}
	Then, due to the classical parabolic estimates, the sequence $(u(t+t_n, x),v(t+t_n, x)$ converges locally uniformly and up to an extraction to an entire solution $(u^\infty(t, x), v^\infty(t, x))$ which satisfies $|u^\infty(0, 0)-u^*|\geq \varepsilon$. By Theorem  \ref{thm:main-lindet}, there exists $\delta>0$ such that $(u^\infty(t, x), v^\infty(t,x))\geq (\delta, \delta).$ Hence Theorem \ref{thm:entire-sol-hom} applies and we have $(u^\infty(t, x), v^\infty(t, x))\equiv (u^*, v^*)$. This is a contradiction. If $|v(t_n, x_n)-v^*|\geq \varepsilon$, we easily derive a similar contradiction.  Therefore, we have
	\begin{equation*}
		\lim_{t\to\infty}\sup_{|x|\leq ct}\max(|u(t,x)-u^*|, |v(t, x)-v^*|)=0,
	\end{equation*}
	and Theorem \ref{thm:ltb} holds.
\end{proof}

In the case when Assumption \ref{as:LTPDE} fails to hold, we can no longer prove the {\em global} stability of the stationary solution, however, since the spectrum of the linearized operator is included in the nonpositive complex plane, we can still prove {\em local} stability by studying the semigroup associated with the system (in particular, no Turing bifurcation is occurring with our system). This is the content of Theorem \ref{thm:hom-local-stab}.
\begin{proof}[Proof of Theorem \ref{thm:hom-local-stab}]
	We divide the proof in two steps. Our strategy is as follows: in the first step we show that the constant stationary solution is linearly stable for the elliptic PDE, meaning that the spectrum of the linearized operator lies in the complex half-plane of negative real parts. In the  second step we show how this linear stability leads to nonlinear stability, by using semigroup theory. \medskip

	\begin{stepping}
		\step We show that the spectrum of the linearized operator is included in the negative complex plane.

		In this Step we investigate the operator:
		\begin{equation*}
			\mathcal{A}
			\begin{pmatrix} 
				g \\ h 
			\end{pmatrix}
			:=
			\begin{pmatrix}
				\sigma_u g_{xx} \\ \sigma_v h_{xx} 
			\end{pmatrix}
			+
			\begin{pmatrix}
				(r_u-\mu_u-2\kappa_u u^* - \kappa_u v^*)g + (\mu_v-\kappa_u v^*)h \\
				(\mu_u-\kappa_vu^*)h+(r_v-\mu_v-\kappa_vu^* - 2\kappa_vv^*)g
			\end{pmatrix},
		\end{equation*}
		considered as an unbounded operator acting on $(g, h)\in BUC(\mathbb R)^2$, $BUC(\mathbb R)$ being the space of bounded and uniformly continuous functions on $\mathbb R$ equipped with the supremum norm (this is classically a Banach space),   with domain $D(\mathcal A)=C^2_b(\mathbb R)^2$.

		Let $\lambda\in \mathbb C$ and  $(\varphi, \psi)\in BUC(\mathbb R)^2$ be given and consider the resolvent equation
		\begin{equation}\label{eq:resolvent}
			(\lambda I-\mathcal A)\parmatrix{g \\ h} = \parmatrix{\varphi \\ \psi}. 
		\end{equation}
		The set of solutions of the latter equation can be computed explicitly by the variation of constants formula. More precisely, we let $Y(x)=(g, g_x, h, h_x)^T$  and rewrite \eqref{eq:resolvent} as an ODE on $\mathbb R^4$:
		\begin{align*}
			\frac{\dd}{\dd x}Y(x) &= 
			\begin{pmatrix}
				0 & 1 & 0 & 0 \\ 
				\sigma_u^{-1}(\lambda - (r_u-\mu_u-2\kappa_u u^* - \kappa_u v^*)) & 0 & -\sigma_u^{-1}(\mu_v-\kappa_u u^*) & 0\\
				0 & 0 & 0 & 1 \\
				-\sigma_v^{-1}(\mu_u-\kappa_vv^*) & 0 & \sigma_v^{-1}(\lambda-(r_v-\mu_v-\kappa_vu^* - 2\kappa_vv^*)) & 0
			\end{pmatrix}Y - \parmatrix{0 \\ \varphi \\ 0 \\ \psi} \\
			&=:B_\lambda Y(x)+Z(x).
		\end{align*}
		We first investigate the bounded eigenvectors  of the ODE  $Y'=B_\lambda Y$. These correspond to the imaginary eigenvalues of the matrix $B_\lambda$, {\it i.e.} the imaginary roots of the polynomial 
		\begin{equation*}
			\chi_\lambda(X):=X^4+\big(\sigma_u^{-1}a + \sigma_v^{-1}d-\lambda(\sigma_u^{-1}+\sigma_v^{-1})\big)X^2+\sigma_u^{-1}\sigma_v^{-1}\big(\lambda^2-(a+d)\lambda +ad - bc\big), 
		\end{equation*}
		where it is convenient to use the notation $a$, $b$, $c$, $d$ introduced to denote the coefficients of the Jacobian matrix of the nonlinearity at the equilibrium point: 
		\begin{align*}
			a&:=r_u-\mu_u-2\kappa_uu^*-\kappa_u v^*=-\left(\kappa_u u^*+\mu_v\frac{v^*}{u^*}\right), & b &:=\mu_v-\kappa_u u^*, \\
			d&:=r_v-\mu_v-\kappa_v u^*-2\kappa_vv^*=-\left(\kappa_vv^*+\mu_u\frac{u^*}{v^*}\right), & c &:=\mu_u-\kappa_v v^*.
		\end{align*}
		We show that there exists a curve $\mathcal C\subset \mathbb C$, which is contained in the half-plane $\Re(z)\leq -\omega$ for $z\in\mathbb C$, where
		\begin{equation}\label{eq:omega}
			\omega:= \min\left(-\frac{a+d}{2}, -\frac{\sigma_u^{-1}a+\sigma_v^{-1}d}{\sigma_u^{-1}+\sigma_v^{-1}}\right),
		\end{equation}  
		and such that $B_\lambda $ has no imaginary eigenvalue if $\lambda $ is in the connected compound $\mathcal C^+$ of $\mathbb C\backslash \mathcal C$ which contains the positive real axis. Moreover $\mathcal C$ asymptotically looks like  straight lines:
		\begin{equation*}
			\Im(z)\sim \pm\left|2\frac{\sigma_u^{-1}\sigma_v^{-1}}{\sigma_u^{-1}+\sigma_v^{-1}}-\sqrt{\sigma_u^{-1}\sigma_v^{-1}}\right|\Re(z), \text{ for } z\in \mathcal C \text{ with }\Re(z)\to -\infty.
		\end{equation*}
		Indeed, investigating the values taken by $\chi_\lambda(iX)$ for real values of $X$, we find that 
		\begin{equation*}
			\chi_\lambda(iX)=X^4+\big(-(\sigma_u^{-1}a + \sigma_v^{-1}d)+\lambda(\sigma_u^{-1}+\sigma_v^{-1})\big)X^2+\sigma_u^{-1}\sigma_v^{-1}\big(\lambda^2-(a+d)\lambda +ad - bc\big).
		\end{equation*}
		Since $a<0$, $d<0$ and $ad-bc>0$ (see Lemma \ref{lem:stat-ode-stability} and note that our notation coincides with \eqref{eq:abcd}), we immediately see that $\chi_\lambda(iX)>0$ if $\lambda $ is real and $\lambda \geq \max\left(a+d, \frac{\sigma_u^{-1}a+\sigma_v^{-1}d}{\sigma_u^{-1}+\sigma_v^{-1}}\right)$. If $\Im(\lambda)\neq 0$, we remark that 
		\begin{equation*}
			\Im(\chi_\lambda (iX))=\Im(\lambda)\left[(\sigma_u^{-1}+\sigma_v^{-1})X^2+\sigma_u^{-1}\sigma_v^{-1}(2\Re(\lambda)-(a+d))\right], 
		\end{equation*}
		therefore if $\Re(\lambda)>\frac{a+d}{2}$ the polynomial $\chi_\lambda(iX)$ cannot have a real root. If $\Re(\lambda)\leq \frac{a+d}{2}$ there are two candidates 
		\begin{equation*}
			X_\pm := \pm\sqrt{\frac{\sigma_u^{-1}\sigma_v^{-1}}{\sigma_u^{-1}+\sigma_v^{-1}}(a+d-2\Re(\lambda))}, 
		\end{equation*}
		and for those values of $X$ we have
		\begin{align*}
			\Re(\chi_\lambda(X))&=\left(\frac{\sigma_u^{-1}\sigma_v^{-1}}{\sigma_u^{-1}+\sigma_v^{-1}}(a+d-2\Re(\lambda))\right)^2-(\sigma_u^{-1}a+\sigma_v^{-1}d)\frac{\sigma_u^{-1}\sigma_v^{-1}}{\sigma_u^{-1}+\sigma_v^{-1}}(a+d-2\Re(\lambda)) \\
			&\quad + ad-bc+\Re(\lambda)\sigma_u^{-1}\sigma_v^{-1}(a+d-2\Re(\lambda))+\sigma_u^{-1}\sigma_v^{-1}\big(\Re(\lambda)^2-\Im(\lambda)^2-(a+d)\Re(\lambda)\big).
		\end{align*}
		We conclude that $\chi_\lambda(iX)$ cannot have a real root in this case either, provided $\Im(\lambda)^2$ is bounded from below by a polynomial of degree two in $\Re(\lambda)$. Hence we have found our curve $\mathcal C$.

		When $\lambda \in\mathcal C^+$ ({\it i.e.} lies in the connected component of $\mathbb C\backslash \mathcal C$ containing $\mathbb R^+$) we show that $Y$ is uniquely determined and depends continuously on $Z$. Indeed,  the set of solutions to the equation $Y'=B_\lambda Y+Z$ can be determined by the {\it variation of constants formula} 
		\begin{equation}\label{eq:hom-res}
			Y(x)=e^{xB_\lambda} Y_0+\int_0^xe^{(x-s)B_\lambda}Z(s)\dd z, 
		\end{equation}
		for arbitrary $Y_0\in \mathbb R^4$. We show that there exists a unique choice of $Y_0$ such that $Y(x)$ remains bounded on $\mathbb R$. Indeed, because of the specific form of $\chi_\lambda(X)$, the matrix $B_\lambda $ has either  four or two distinct eigenvalue.  The latter case occurs exactly when the discriminant of the characteristic polynomial $\chi_\lambda(X)$ is null, namely
		\begin{equation*}
			\big(\sigma_u^{-1}a + \sigma_v^{-1}d -\lambda (\sigma_u^{-1}+\sigma_v^{-1})\big)^2 - 4\sigma_u^{-1}\sigma_v^{-1}\big(\lambda^2-(a+d)\lambda + ad-bc\big)=0.
		\end{equation*}
		The left-hand polynomial $D(\lambda)$ of the former equation can be written as
		\begin{align*}
			D(\lambda)&=\sigma_u^{-2} a^2 + \sigma_v^{-2} d^{2} +\lambda^2(\sigma_u^{-1}+\sigma_v^{-1})^2 +2\sigma_u^{-1}\sigma_v^{-1}ad -2\sigma_u^{-1}a\lambda(\sigma_u^{-1}+\sigma_v^{-1})-2\sigma_v^{-1}d\lambda(\sigma_u^{-1}+\sigma_v^{-1})\\ 
			&\quad -4\sigma_u^{-1}\sigma_v^{-1}\lambda^2 + 4\sigma_u^{-1}\sigma_v^{-1}(a+d)\lambda - 4\sigma_u^{-1}\sigma_v^{-1}(ad-bc) \\
			&=\lambda^2\big(\sigma_u^{-1}-\sigma_v^{-1}\big)^2 +\lambda\big( -2(\sigma_u^{-1}a+\sigma_v^{-1}d)(\sigma_u^{-1}+\sigma_v^{-1})+4\sigma_u^{-1}\sigma_v^{-1}(a+d)\big)+(\sigma_u^{-1}a-\sigma_v^{-1}d)^2+4bc\\
			&=\lambda^2\big(\sigma_u^{-1}-\sigma_v^{-1}\big)^2 +\lambda\big( 2\sigma_u^{-1}\sigma_v^{-1}(a+d)-2(\sigma_u^{-2}a+\sigma_v^{-2}d)\big)+(\sigma_u^{-1}a-\sigma_v^{-1}d)^2+4bc\\
			&=\lambda^2\big(\sigma_u^{-1}-\sigma_v^{-1}\big)^2 +2\lambda(\sigma_v^{-1}-\sigma_u^{-1})(\sigma_u^{-1} a-\sigma_v^{-1} d)+(\sigma_u^{-1}a-\sigma_v^{-1}d)^2+4bc \\
			&=\big(\lambda(\sigma_u^{-1}-\sigma_v^{-1})+\sigma_u^{-1}a-\sigma_v^{-1}d\big)^2 + 4bc.
		\end{align*}
		The roots are determined by the sign of $bc$:
		\begin{equation}\label{eq:exceptional_point}
			\lambda^{\pm}=\frac{\sigma_v^{-1}d-\sigma_u^{-1}a\pm2\sqrt{-bc}}{\sigma_u^{-1}-\sigma_v^{-1}}\quad\text{ if }bc<0,\quad  \lambda^{\pm}=\frac{\sigma_v^{-1}d-\sigma_u^{-1}a\pm2i\sqrt{bc}}{\sigma_u^{-1}-\sigma_v^{-1}}\quad \text{ if } bc>0.
		\end{equation}
		Note that, since 
		\begin{equation*}
			B_\lambda^2=\begin{pmatrix}
				\sigma_u^{-1} (\lambda - a) & 0 & -\sigma_u^{-1}b & 0 \\
				0 &\sigma_u^{-1}(\lambda - a) & 0 & -\sigma_u^{-1} b \\ 
				-\sigma_v^{-1} c & 0 & \sigma_v^{-1} (\lambda-d) & 0 \\
				0 & -\sigma_v^{-1} c & 0 & \sigma_v^{-1}(\lambda-d) 
			\end{pmatrix},
		\end{equation*}
		there is no hope that the matrix $B_\lambda$ is diagonalizable when the characteristic polynomial has only two roots (because the minimal polynomial has degree $>2$; see also the Motzkin-Taussky Theorem \cite[Theorem 2.6 p.85]{Kat-95}).\medskip

		Therefore we distinguish two cases.

		\noindent{\sc Case 1.} The matrix $B_\lambda$ is diagonalizable. 

		\noindent In this case there exists $\lambda_0, \lambda_1\in\mathbb C$ such that $0<\Re(\lambda_0)\leq\Re(\lambda_1)$ and an invertible matrix $P\in M_4(\mathbb R)$ such that 
		\begin{equation*}
			B_\lambda = P\,\diag(\lambda_1, \lambda_0,- \lambda_0, -\lambda_1)P^{-1}.
		\end{equation*}
		In this case solving equation \eqref{eq:hom-res} on each eigenspace yields
		\begin{equation*}
			Y_0=P\begin{pmatrix}
				-\int_{0}^{+\infty}e^{-\lambda_1s}\tilde Z_+^1(s)\dd s\\ 
				-\int_{0}^{+\infty}e^{-\lambda_0s}\tilde Z_+^0(s)\dd s\\ 
				\int_{-\infty}^{0}e^{\lambda_0s}\tilde Z_-^0(s)\dd s\\ 
				\int_{-\infty}^{0}e^{\lambda_1s}\tilde Z_-^1(s)\dd s
			\end{pmatrix}, 
			\text{ where } 
			\tilde Z(x):=\begin{pmatrix} 
				\tilde Z_+^1(x) \\ 
				\tilde Z_+^0(x) \\ 
				\tilde Z_-^0(x) \\ 
				\tilde Z_-^1(x)
			\end{pmatrix}
			= P^{-1}Z(x).
		\end{equation*}
		Therefore $Y_0$ is a continuous function of $Z$ and \eqref{eq:hom-res} is recast 
		\begin{equation*}
			Y(x)=P\begin{pmatrix}
				-\int_{x}^{+\infty}e^{\lambda_1(x-s)}\tilde Z_+^1(s)\dd s\\ 
				-\int_{x}^{+\infty}e^{\lambda_0(x-s)}\tilde Z_+^0(s)\dd s\\ 
				\int_{-\infty}^{x}e^{-\lambda_0(x-s)}\tilde Z_-^0(s)\dd s\\ 
				\int_{-\infty}^{x}e^{-\lambda_1(x-s)}\tilde Z_-^1(s)\dd s
			\end{pmatrix}.
		\end{equation*}
		We have found that $\lambda-\mathcal A $ admits a bounded inverse in $BUC(\mathbb R)^2 $.\medskip

		\noindent{\sc Case 2.} The matrix $B_\lambda$ is not diagonalizable (i.e. $\lambda=\lambda^{\pm}$ given by \eqref{eq:exceptional_point}). 

		\noindent In this case, there is $\lambda_0\in\mathbb C$ with $\Re(\lambda_0)>0$ and an invertible matrix $P\in M_4(\mathbb R)$ such that $B_\lambda$ is equivalent to its Jordan normal form:
		\begin{equation*}
			B_\lambda = P\begin{pmatrix} 
				\lambda_0 & 1 & 0 & 0 \\
				0 & \lambda_0 & 0 & 0 \\
				0 & 0 & -\lambda_0 & 1 \\
				0 & 0 & 0 & -\lambda_0
			\end{pmatrix}
			P^{-1}, 
		\end{equation*}
		and therefore 
		\begin{equation*}
			e^{xB_\lambda} = P\begin{pmatrix} 
				e^{\lambda_0x} & xe^{\lambda_0x} & 0 & 0 \\
				0 & e^{\lambda_0x} & 0 & 0 \\
				0 & 0 & e^{-\lambda_0x} & xe^{-\lambda_0x} \\
				0 & 0 & 0 & e^{-\lambda_0x}
			\end{pmatrix}
			P^{-1}. 
		\end{equation*}
		In this case solving equation \eqref{eq:hom-res} on each eigenspace yields
		\begin{equation*}
			Y_0=P\begin{pmatrix}
				-\int_{0}^{+\infty}e^{-\lambda_0s}\big(\tilde Z_+^1(s)-s\tilde Z_{+}^0(s)\big)\dd s\\ 
				-\int_{0}^{+\infty}e^{-\lambda_0s}\tilde Z_{+}^0(s)\dd s\\ 
				\int_{-\infty}^{0}e^{\lambda_0s}\big(\tilde Z_-^0(s) - s\tilde Z_-^1(s)\big)\dd s\\ 
				\int_{-\infty}^{0}e^{\lambda_1s}\tilde Z_-^1(s)\dd s
			\end{pmatrix}, 
			\text{ where } 
			\tilde Z(x):=\begin{pmatrix} 
				\tilde Z_+^1(x) \\ 
				\tilde Z_+^0(x) \\ 
				\tilde Z_-^0(x) \\ 
				\tilde Z_-^1(x)
			\end{pmatrix}
			= P^{-1}Z(x).
		\end{equation*}
		Once again we have found that $Y_0$ depends continuously on $Z$ and therefore $\lambda-\mathcal A$ admits a continuous inverse on $BUC(\mathbb R)^2$ given by the formula 
		\begin{equation*}
			Y(x)=P\begin{pmatrix}
				-\int_{x}^{+\infty}e^{\lambda_0(x-s)}\big(\tilde Z_+^1(s)+(x-s)\tilde Z_+^0(s)\big)\dd s\\ 
				-\int_{x}^{+\infty}e^{\lambda_0(x-s)}\tilde Z_+^0(s)\dd s\\ 
				\int_{-\infty}^{x}e^{-\lambda_0(x-s)}\big(\tilde Z_-^0(s)+(x-s)\tilde Z_-^1(s)\big)\dd s\\ 
				\int_{-\infty}^{x}e^{-\lambda_0(x-s)}\tilde Z_-^1(s)\dd s
			\end{pmatrix}.
		\end{equation*}

		To finish our first Step we remark that the operator $\mathcal A$ is sectorial and generates an analytic semigroup on $BUC(\mathbb R)^2$. Indeed, $\mathcal A $ is a bounded perturbation of the unbounded operator $(\sigma_u \partial_{xx}, \sigma_v\partial_{xx})^T$ (acting on $D(A)=C^2_{UC}(\mathbb R)^2$ in $BUC(\mathbb R)^2$), which is sectorial and  generates an analytic semigroup on $BUC(\mathbb R)^2$ \cite[Corollary 3.1.9 p. 81]{Lun-95}. In particular, $e^{t\mathcal A}$ can be computed by the Dunford-Taylor integral
		\begin{equation*}
			e^{t\mathcal A}=\frac{1}{2i\pi}\int_{\Gamma}e^{\lambda t}(\lambda I-\mathcal A)^{-1}\dd \lambda ,
		\end{equation*}
		where $\Gamma$ is a curve in $\mathcal C^+$ joining a straight line $\{\rho e^{-i\theta}, \rho>0\}$ for some $\theta\in\left[\frac{\pi}{2}, \pi\right)$ to the straight line $\{\rho e^{-i\theta}:\rho>0\}$ and oriented so that $\Im (\lambda) $ increases on $\Gamma$. From the above computations it is clear that $\Gamma $ can be chosen so that $\Re(\lambda)\leq -\frac{\omega}{2}$ (where $\omega$ is given by \eqref{eq:omega}) for all $\lambda \in \Gamma$, in which case
		\begin{equation*}
			e^{t\mathcal A} = e^{-\frac{\omega}{2} t}\cdot\frac{1}{2i\pi}\int_{\Gamma} e^{\left(\lambda + \frac{\omega}{2}\right) t}(\lambda-\mathcal A)^{-1}\dd \lambda
		\end{equation*}
		therefore
		\begin{align*}
			\Vert e^{t\mathcal A}\Vert_{BUC(\mathbb R)^2}&\leq e^{-\frac{\omega}{2}t}\cdot\frac{1}{2\pi} \int_{\Gamma} e^{-(\Re(\lambda)+\frac{\omega}{2})t}\Vert (\lambda-\mathcal A)^{-1}\Vert_{\mathcal L(BUC(\mathbb R)^2)}\dd\lambda\\ 
			&\leq Ce^{-\frac{\omega}{2}t},
		\end{align*}
		for all $t>0$, where $C$ depends only on $\mathcal A$ and $\omega$.

		\medskip

		\step We show the nonlinear stability. 

		Let $(u(t,x), v(t,x))$ be the  solution of \eqref{eq:syst-hom-rd} starting from $(u_0, v_0)\in BUC(\mathbb R)^2$. We remark that 
		\begin{equation*}
			\begin{pmatrix} u-u^* \\ v-v^*\end{pmatrix}_t-\mathcal A\begin{pmatrix} u-u^* \\ v-v^*\end{pmatrix}=o\left(\left\Vert \begin{pmatrix} u-u^* \\ v-v^*\end{pmatrix} \right\Vert_{BUC(\mathbb R)^2}\right), 
		\end{equation*}
		that our original equation \eqref{eq:syst-hom-rd} is a Lipschitz perturbation of the semigroup $T(t)$ generated by $\mathcal A$, and that it has been shown in Step 1  that $e^{\frac{\omega}{2} t}T(t)$ is bounded, with $\omega>0$ defined by \eqref{eq:omega}. In this context, it has been shown in  \cite[Theorem 10.2.2 p.157]{Caz-Har-98} (as a consequence of Gronwall's inequality) that  there exists a $\varepsilon_0>0$ and a constant $M>0$ such that 
		\begin{equation*}
			\left\Vert \begin{pmatrix} u(t,\cdot) \\ v(t, \cdot)\end{pmatrix} - \begin{pmatrix} u^* \\ v^*\end{pmatrix}\right\Vert_{BUC(\mathbb R)^2}\leq M\left\Vert \begin{pmatrix} u_0 \\ v_0\end{pmatrix} - \begin{pmatrix} u^* \\ v^*\end{pmatrix}\right\Vert_{BUC(\mathbb R)^2}e^{-\frac{\omega}{2}t}, \text{ if } \left\Vert \begin{pmatrix} u_0-u^* \\ v_0-v^*\end{pmatrix} \right\Vert_{BUC(\mathbb R)^2}\leq \varepsilon_0.
		\end{equation*}
	\end{stepping}

	This finishes the proof of Theorem \ref{thm:hom-local-stab}.
\end{proof}

\subsection{Homogenization}
\label{sec:hom}

In this section we extend the results obtained for the homogeneous systems to the class of systems with rapidly oscillating coefficients. 

Recall that we are concerned with system \eqref{eq:rapidosc}:
\begin{equation*}
	\begin{system}
		\relax &u_t=(\sigma_u^\varepsilon(x) u_{x})_x+(r_u^\varepsilon(x)-\kappa_u^\varepsilon(x)(u+v))u+\mu_v^\varepsilon(x)v-\mu_u^\varepsilon(x)u \\
		\relax &v_t=(\sigma_v^\varepsilon(x) v_{x})_x+(r_v^\varepsilon(x)-\kappa_v^\varepsilon(x)(u+v))v+\mu_u^\varepsilon(x)u-\mu_v^\varepsilon(x)v,
	\end{system}
\end{equation*}
where $\sigma_u^\varepsilon(x):=\sigma_u\left(\frac{x}{\varepsilon}\right)$, $\sigma_v^\varepsilon(x):=\sigma_v\left(\frac{x}{\varepsilon}\right)$, $r_u^\varepsilon(x):=r_u\left(\frac{x}{\varepsilon}\right)$, $r_v^\varepsilon(x):=r_v\left(\frac{x}{\varepsilon}\right)$, $\kappa_u^\varepsilon(x):=\kappa_u\left(\frac{x}{\varepsilon}\right)$, $\kappa_v^{\varepsilon}(x):=\kappa_v\left(\frac{x}{\varepsilon}\right)$, $\mu_u^\varepsilon(x):=\mu_u\left(\frac{x}{\varepsilon}\right)$, $\mu_v^\varepsilon(x):=\mu_v\left(\frac{x}{\varepsilon}\right)$ and $r_u$, $r_v$, $\kappa_u$, $\kappa_v$ are periodic with period 1. We also recall the definitions of the mean coefficients as in \eqref{eq:mean-coeffs}:
\begin{align*}
	\overline{r_u}&:=\int_0^1r_u(x)\dd x , & \overline{\kappa_u}&:=\int_0^1\kappa_u(x)\dd x , & \overline{\mu_u}&:=\int_0^1\mu_u(x)\dd x, \\
	\overline{r_v}&:=\int_0^1r_v(x)\dd x, &\overline{\kappa_v}&:=\int_0^1\kappa_v(x)\dd x, & \overline{\mu_v}&:=\int_0^1\mu_v(x)\dd x, \notag 
\end{align*}
and finally:
\begin{align*}
	\overline{\sigma_u}^H&:=\left(\int_0^1\frac{1}{\sigma_u(x)}\dd x\right)^{-1}, & \overline{\sigma_v}^H&:=\left(\int_0^1\frac{1}{\sigma_v(x)}\dd x\right)^{-1}.
\end{align*}

\begin{lem}[The homogenisation limit of entire solutions]\label{lem:rapid-osc-lim}
	Let $\overline{\sigma_u}^H$, $\overline{\sigma_v}^H$, $\overline{r_u}$, $\overline{r_v} $, $\overline{\kappa_u}$, $\overline{\kappa_v}$,  $\overline{\mu_u}$ and $\overline{\mu_v}$ satisfy Assumption \ref{as:cond-instab-0} and \ref{as:LTPDE}. Let $\varepsilon>0$ and $(u^\varepsilon(t, x), v^\varepsilon(t, x)) $ be a nonnegative nontrivial  entire solution to \eqref{eq:rapidosc} which is bounded from above and from below by positive constants. Then, as $\varepsilon\to 0$, the functions $(u^\varepsilon(t, x), v^\varepsilon(t, x))$ converge locally uniformly to the unique nonnegative nontrivial  stationary state $(u^*, v^*)$ of the homogenised problem \eqref{eq:syst-hom-rd} with $\sigma_u$, $\sigma_v$, $r_u$, $r_v$, $\kappa_u$, $\kappa_v$, $\mu_u$, $\mu_v$ replaced by $\overline{\sigma_u}^H$, $\overline{\sigma_v}^H$, $\overline{r_u}$, $\overline{r_v}$, $\overline{\kappa_u}$, $\overline{\kappa_v}$, $\overline{\mu_u}$, $\overline{\mu_v}$.
\end{lem}
\begin{proof}
	We divide the proof in three steps.

	\begin{stepping}
		\step  We show that $(u^\varepsilon(t, x), v^\varepsilon(t, x))$ does not vanish. 

		Let $\big(\lambda_1^\varepsilon, (\varphi^\varepsilon(x)>0,\psi^\varepsilon(x)>0)\big) $ be the principal eigenpair associated with the eigenproblem:
		\begin{equation*}
			\begin{system}
				\relax &-(\sigma_u^\varepsilon(x)\varphi^\varepsilon_{x})_x=(r_u^\varepsilon(x)-\mu_u^\varepsilon(x))\varphi^\varepsilon(x)+\mu_v^\varepsilon(x)\psi^\varepsilon(x)+\lambda_1^\varepsilon\varphi^\varepsilon(x)\\
				\relax &-(\sigma_v^\varepsilon(x)\psi^\varepsilon_{x})_x= \mu_u^\varepsilon(x)\varphi^\varepsilon(x)+(r^\varepsilon_v(x)-\mu_v^\varepsilon(x))\psi^\varepsilon(x) + \lambda_1^\varepsilon\psi^\varepsilon(x),
			\end{system}
		\end{equation*}
		with $\varepsilon$-periodic boundary conditions, and normalised with $\max_{x\in\mathbb R}\sup\max(\varphi^\varepsilon(x), \psi^\varepsilon(x))=1$. Since $(u^\varepsilon, v^\varepsilon)$ is bounded from below,   there exists $\alpha>0$ such that $\alpha(\varphi^\varepsilon(x), \psi^\varepsilon(x))\leq (u^\varepsilon(t, x), v^\varepsilon(t, x))$. Let us define 
		\begin{equation*}
			A:=\sup\left\{\alpha>0, \forall x\in\mathbb R, \alpha(\varphi^\varepsilon(x), \psi^\varepsilon(x))\leq (u^\varepsilon(t, x), v^\varepsilon(t, x))\right\}.
		\end{equation*}
		Then by definition of $A>0$ (and up to a shift and limiting process), there exists $t, x\in\mathbb R$ such that either $u^\varepsilon(t, x)=A\varphi^\varepsilon(x)$ or $v^\varepsilon(t, x)=A\psi^\varepsilon(x)$. Let us assume that the former holds. Then, we have 
		\begin{align*}
			0&\geq -\kappa_u^\varepsilon(x)u(t,x)(u(t,x)+v(t,x))-\lambda_1^\varepsilon A\varphi^\varepsilon(x) 
			=-A^2\kappa_u^\varepsilon(x)\varphi(x)(\varphi(x)+\psi(x))-\lambda_1^\varepsilon A\varphi^\varepsilon(x) \\
			&\geq A\varphi^\varepsilon(x)(-2A\sup_{y\in\mathbb R}\kappa_u(y) - \lambda_1^\varepsilon),
		\end{align*}
		which implies that $A\geq \frac{-\lambda_1^\varepsilon}{2\sup_{y\in\mathbb R}\kappa_u(y)}$. We get a similar estimate in the case $v^\varepsilon(x)=A\psi^\varepsilon(x)$, which shows the inequality
		\begin{equation*}
			A\geq \frac{-\lambda_1^\varepsilon}{2\max_{x\in\mathbb R}\big(\max(\kappa_u(x), \kappa_v(x))\big)}.
		\end{equation*}
		Then, it is classical (and has been proved in the proof of Theorem \ref{thm:large-diff}) that $(\lambda_1^\varepsilon, (\varphi^\varepsilon(x), \psi^\varepsilon(x)))\to(\lambda_1^0<0, (\varphi^0, \psi^0)) $ uniformly as $\varepsilon\to 0$, where $(\lambda_1^0, (\varphi^0, \psi^0)) $ is the principal eigenpair of the homogenised problem.   

		Note that in particular, there exists $\overline{\varepsilon}>0$ such that for $0<\varepsilon\leq\overline{\varepsilon} $, we have a true uniform lower bound on $(u^\varepsilon(x),v^\varepsilon(x))$:
		\begin{equation*}
			\inf_{(t, x)\in\mathbb R^2}\min(u(t, x),v(t, x))\geq\frac 12\cdot \frac{-\lambda_1^0}{2\max_{x\in\mathbb R}\big(\max(\kappa_u(x), \kappa_v(x))\big)}\min(\varphi^0, \psi^0)>0.
		\end{equation*}

		\step We show that $u^\varepsilon(t, x)$ and $v^\varepsilon(t, x)$ are uniformly bounded.

		Indeed, since $(u^\varepsilon,v^\varepsilon)$ is bounded, it follows directly from the maximum principle that 
		\begin{equation*}
			\underset{(t, x)\in\mathbb R^2}{\sup}\max(u(t, x),v(t, x))\leq \max\left(\frac{\max_{x\in\mathbb R}r_u(x)}{\min_{x\in\mathbb R}\kappa_u(x)}, \frac{\max_{x\in\mathbb R}r_v(x)}{ \min_{x\in\mathbb R} \kappa_v(x)}\right).
		\end{equation*}

		\step We derive the limit of $(u^\varepsilon, v^\varepsilon)$. \medskip

		We first remark that, since $(u^\varepsilon, v^\varepsilon)$ is uniformly bounded, the classical estimates for parabolic equations in divergence form with discontinuous coefficients (see {\it e.g.} \cite[Chapter II Theorem 10.1]{Lad-Sol-Ura-68}) imply that $(U^\varepsilon, v^\varepsilon)$ is locally uniformly bounded in $C^\alpha (\mathbb R\times\mathbb R)$, {\it i.e.} for any  $T>0$ an $R>0$ there exists $C>0$ (independent of $\varepsilon$) such that 
		\begin{equation*}
			\max\left(\Vert u^\varepsilon\Vert_{C^\alpha([-T, T]\times [-R, R])}, \Vert v^\varepsilon\Vert_{C^\alpha([-T, T]\times [-R, R])}\right) \leq C.
		\end{equation*}
		Then, a classical diagonal extraction process allows us to extract a subsequence along which $(u^{\varepsilon}, v^\varepsilon)$ converges locally uniformly in $C^{\alpha/2}(\mathbb R^2)$ to a limit $(u, v)$. It is then classical (see {\it e.g.} Remark 1.3 in Chapter 2 of \cite{Ben-Lio-Pap-11}) that $(u, v)$ satisfies weakly: 
		\begin{equation*}
			\begin{system}
				\relax &u_t=\overline{\sigma_u}^H u_{xx}+(\overline{r_u}-\overline{\kappa_u}(u+v))u+\overline{\mu_v}v-\overline{\mu_u}u \\
				\relax &v_t=\overline{\sigma_v}^H v_{xx}+(\overline{r_v}-\overline{\kappa_v}(u+v))v+\overline{\mu_u}u-\overline{\mu_v}v.
			\end{system}
		\end{equation*}
		Then the Schauder estimates imply that $(u(x), v(x))$ is in fact a classical solution to \eqref{eq:syst-hom-rd}. Since $(u(x), v(x)) $ is nontrivial and bounded from below (by Step 1), Theorem \ref{thm:entire-sol-hom} shows that $u(t,x)\equiv u^*$ and $v(t,x)\equiv v^*$.
	\end{stepping}
\end{proof}

\begin{lem}[Uniqueness of rapidly oscillating entire solution]\label{lem:rapid-osc-unique}
	Let $\overline{r_u}$, $\overline{r_v} $, $\overline{\kappa_u}$, $\overline{\kappa_v}$,  $\overline{\mu_u}$ and $ \overline{\mu_v} $ satisfy Assumption \ref{as:cond-instab-0}. There exists $\overline{\varepsilon}$ such that if $ 0<\varepsilon\leq \overline{\varepsilon}$, there exists a unique nonnegative nontrivial entire solution $(u^\varepsilon(x), v^\varepsilon(x)) $ associated with \eqref{eq:rapidosc}, which is bounded from above and from below.
\end{lem}
\begin{proof}
	We argue by contradiction and assume there exists a sequence $\varepsilon_n>0$ and  two sequences $(u_1^{\varepsilon_n}(t,x), v_1^{\varepsilon_n}(t,x))\not\equiv(u_2^{\varepsilon_n}(t, x), v_2^{\varepsilon_n}(t, x))$ of bounded nonnegative nontrivial stationary solutions to \eqref{eq:rapidosc}. We define $\delta_n:=\max\left(\Vert u_2^{\varepsilon_n}(t,x)-u_1^{\varepsilon_n}(t,x)\Vert_{BUC(\mathbb R)^2},\Vert v_2^{\varepsilon_n}(t,x)-v_1^{\varepsilon_n}(t,x)\Vert_{BUC(\mathbb R)^2}\right) $ and:
	\begin{align*}
		\varphi^{\varepsilon_n}(t,x)&:=\frac{1}{\delta_n}(u_2^{\varepsilon_n}(t,x)-u_1^{\varepsilon_n}(t,x)) \\
		\psi^{\varepsilon_n}(t,x)&:=\frac{1}{\delta_n}(v_2^{\varepsilon_n}(t,x)-v_1^{\varepsilon_n}(t,x)).
	\end{align*}
	Up to a shift in time and space we assume that  
	\begin{equation}\label{eq:rapid-osc-unique-norm}
		\frac{\delta_n}{2}\leq \sup_{x\in (0,L)}\big(\max(| u_2^{\varepsilon_n}(0,x)-u_1^{\varepsilon_n}(0,x)|, | v_2^{\varepsilon_n}(0,x)-v_1^{\varepsilon_n}(0,x)|)\big) \leq \delta_n.
	\end{equation}
	Then $(\varphi^{\varepsilon_n}(t,x), \psi^{\varepsilon_n}(t,x)) $ satisfy:
	\begin{equation*}
		\begin{system}
			\relax &\varphi_t-\sigma_u^{\varepsilon_n}(x)\varphi^{\varepsilon_n}_{xx} = (r_u^{\varepsilon_n}(x)-\mu^{\varepsilon_n}(x))\varphi^{\varepsilon_n} + \mu_v^{\varepsilon_n}(x)\psi^{\varepsilon_n}  - \kappa_u^{\varepsilon_n}(x)(2u^{\varepsilon_n}_2+v^{\varepsilon_n}_2)\varphi^{\varepsilon_n} - \kappa_u^{\varepsilon_n}(x)u^{\varepsilon_n}_2\psi^{\varepsilon_n}+o(1)  \\ 
			\relax &\psi_t-\sigma_v^{\varepsilon_n}(x) \psi^{\varepsilon_n}_{xx}=\mu_u^{\varepsilon_n}(x)\varphi^{\varepsilon_n} + (r_v^{\varepsilon_n}(x)-\mu_v^{\varepsilon_n}(x))\psi^{\varepsilon_n} - \kappa_v^{\varepsilon_n}(x)v^{\varepsilon_n}_2\varphi^{\varepsilon_n} - \kappa_v^{\varepsilon_n}(x)(u^{\varepsilon_n}_2+2v^{\varepsilon_n}_2)\psi^{\varepsilon_n}+o(1) \\ 
		\end{system}
	\end{equation*}
	Indeed, owing to Lemma \ref{lem:rapid-osc-lim}, there holds 
	\begin{equation*}
		(u^{\varepsilon_n}_1, v^{\varepsilon_n}_1) \to (u^*, v^*) \text{ and }(u^{\varepsilon_n}_2, v^{\varepsilon_n}_2) \to (u^*, v^*) \text{ in  } BUC(\mathbb R).
	\end{equation*}
	Since $\varphi^{\varepsilon_n}(t, x)$ and $\psi^{\varepsilon_n}(t, x) $ are bounded, classical homogenisation theory (see the proof of Theorem \ref{thm:large-diff} where a similar argument is sketched) then  leads to the convergence (up to an extraction) of $(\varphi^{\varepsilon_n}(t, x), \psi^{\varepsilon_n}(t, x))$ to  $(\varphi(t, x), \psi(t, x))$ solving 
	\begin{equation*}
		\begin{system}
			\relax &\varphi_t-\overline{\sigma_u}^H\varphi_{xx} = (\overline{r_u}-\overline{\mu_u})\varphi + \overline{\mu_v}\psi  - \overline{\kappa_u}(2u^*+v^*)\varphi - \overline{\kappa_u}u^*\psi  \\ 
			\relax &\psi_t-\overline{\sigma_v}^H \psi_{xx}=\overline{\mu_u}\varphi + (\overline{r_v}-\overline{\mu_v})\psi - \overline{\kappa_v}v^*\varphi - \overline{\kappa_v}(u^*+2v^*)\psi,
		\end{system}
	\end{equation*}
	and the convergence holds (at least) locally uniformly. Because of our normalisation \eqref{eq:rapid-osc-unique-norm}, the limit is non-trivial. Moreover,  $(\varphi(t, x), \psi(t, x))$ is bounded, which not possible since $(u^*, v^*)$ is locally asymptotically stable by Theorem \ref{thm:hom-local-stab}. The Lemma is proved. 
\end{proof}

\begin{proof}[Proof of Theorem \ref{thm:rapidosc}]
	Theorem \ref{thm:rapidosc} is a direct consequence of the two previous Lemma. Statements \ref{item:rapid-osc-unique} and \ref{item:entire-sol} are a direct consequence of Lemma \ref{lem:rapid-osc-unique}. As for Statement \ref{item:rapid-osc-cv}, it is also  a consequence of Lemma \ref{lem:rapid-osc-unique}.  

	Indeed, let $\varepsilon>0$ be sufficiently small, so that there exists a unique entire solution to \eqref{eq:rapidosc} which is uniformly bounded from below. Let $(u_0(x)\geq 0, v_0(x)\geq 0)$ be a nontrivial initial condition and $(u(t,x), v(t,x))$ be the corresponding solution to \eqref{eq:rapidosc}. We argue by contradiction and assume that there exists $\varepsilon>0$, $0<c<c^*_\varepsilon$ and  a sequences $t_n\to+\infty$ and $x_n\in\mathbb R$ such that $|x_n|\leq ct_n$ and 
	\begin{equation*}
		|u(t_n, x_n)-u^*|\geq\varepsilon>0.
	\end{equation*}
	Then, due to the classical parabolic estimates, the sequence $(u(t+t_n, x),v(t+t_n, x)$ converges locally uniformly and up to an extraction to an entire solution $(u^\infty(t, x), v^\infty(t, x))$ which satisfies $|u^\infty(0, 0)-u^*|\geq \varepsilon$. By Theorem  \ref{thm:main-lindet}, there exists $\delta>0$ such that $(u^\infty(t, x), v^\infty(t,x))\geq (\delta, \delta).$ Hence Theorem \ref{thm:entire-sol-hom} applies and we have $(u^\infty(t, x), v^\infty(t, x))\equiv (u^*, v^*)$. This is a contradiction. If $|v(t_n, x_n)-v^*|\geq \varepsilon$, we easily derive a similar contradiction.  Therefore, we have
	\begin{equation*}
		\lim_{t\to\infty}\sup_{|x|\leq ct}\max(|u(t,x)-u^*|, |v(t, x)-v^*|)=0.
	\end{equation*}
	This shows Statement \ref{item:rapid-osc-cv} and finishes the proof of Theorem \ref{thm:rapidosc}.
\end{proof}

\appendix
\begin{center}
	\Large\bf Appendix
\end{center}

\section{On the spreading speed in the presence of a drift}
\label{app:scalar}

Let us consider the equation:
\begin{align}\label{eq:scalar-drift}
	u_t&=\mathcal Lu -\kappa(x)u^2 \\ 
	\mathcal Lu:&=(\sigma(x)u_x)_x + q(x)u_x + r(x)u, \notag
\end{align}
where $\sigma(x)>0$, $\kappa(x)>0$, $r(x)$ and $q(x)$ are 1-periodic functions. It is known (see Nadin \cite{Nad-09, Wei-02}  that the rightward and leftward spreading speeds associated with \eqref{eq:scalar-drift} are given by the following minimization formula
\begin{equation*}
	c^*_{\text{right}}:=\inf_{\lambda>0}\frac{-k(\lambda)}{\lambda}\text{ and } c^*_{\text{left}}:=\inf_{\lambda>0}\frac{-k(-\lambda)}{\lambda}.
\end{equation*}
If $q\equiv 0$ then, as a consequence of the Fredholm alternative, the function $k(\lambda)$ is even and $c^*_{\text{right}}=c^*_{\text{left}}$, but this is also the case if $\int_0^1\frac{q(x)}{2\sigma(x)}\dd x=0$ \cite[Proposition 2.14]{Nad-09}, because the advection term can then be ``absorbed'' by a change of function. Further dependencies of the speed on the various coefficients involved in \eqref{eq:scalar-drift} are studied in \cite{Nad-11}. Here we are interested in a sufficient condition for the speeds $c^*_{\text{right}}$ and $c^*_{\text{left}}$ to be different, $c^*_{\text{right}}\neq c^*_{\text{left}}$.

It turns out that this can be achieved quite nicely, by considering the  function 
\begin{equation*}
	\mathcal Q(x):=\int_0^x\frac{q(y)}{2\sigma(y)}\dd y-x\int_0^1\frac{q(y)}{2\sigma(y)}\dd y \text{ for all } x\in \mathbb R.
\end{equation*}
Indeed, writing
\begin{equation*}
	\mathcal Lu=e^{-\mathcal Q(x)-x\int_0^1\frac{q(y)}{2\sigma(y)}\dd y}\left(\sigma(x)\left( e^{\mathcal Q(x)+x\int_0^1\frac{q(y)}{2\sigma(y)}\dd y}u\right)_x\right)_x + \left[-\frac{q_x(x)}{2}-\frac{q(x)^2}{4\sigma(x)}+r(x)\right]u, 
\end{equation*}
and computing the principal periodic eigenvalue of the operator $u\mapsto e^{\lambda x}\mathcal L(e^{-\lambda x}u)$, we find that 
\begin{equation*}
	k(\lambda)=\tilde k\left(\lambda-\int_0^1\frac{q(y)}{2\sigma(y)}\dd y\right), 
\end{equation*}
where $\tilde k(\lambda)$ is associated with the operator 
\begin{equation*}
	\tilde{\mathcal L} u =e^{-\mathcal Q(x)}\left(\sigma(x)\left( e^{\mathcal Q(x)}u\right)_x\right)_x + \left[-\frac{q_x(x)}{2}-\frac{q(x)^2}{4\sigma(x)}+r(x)\right]u, 
\end{equation*}
which is self-adjoint for the weighted scalar product $\langle u, v\rangle_{\mathcal Q} = \int_0^1 u(x)v(x)e^{Q(x)}\dd x$, and satisfies therefore $\tilde k(\lambda)=\tilde k(-\lambda)$. 

In particular, it is not difficult to see that
\begin{equation}\label{eq:orderspeed}
	c^*_{\text{right}}>c^*_{\text{left}}\text{ if }\int_0^1\frac{q(x)}{2\sigma(x)}\dd x<0\text{ and } c^*_{\text{right}}<c^*_{\text{left}}\text{ if } \int_0^1\frac{q(x)}{2\sigma(x)}\dd x>0.
\end{equation}

\printbibliography
\end{document}